\documentclass[10pt, a4paper]{article}

 \usepackage[margin=1.2in]{geometry}
 \usepackage{lipsum}
\usepackage{amsmath,mathtools,amsthm, enumitem, hyperref}
\usepackage{amssymb}
\usepackage{amsthm}
\usepackage{svg}
\usepackage{float}
\usepackage{pinlabel,graphicx}
\usepackage{caption}
\usepackage{array}
\usepackage{esvect}
\usepackage{xfrac, bigints}
\usepackage[normalem]{ulem}

\usepackage{stmaryrd}
\usepackage[nottoc]{tocbibind}
\usepackage{enumitem}
\usepackage[english]{babel}
\usepackage{makecell}

\usepackage{svg}
\usepackage{float}
\theoremstyle{plain}% default
\newtheorem{theorem}{Theorem}[section]
\newtheorem{lem}[theorem]{Lemma}
\newtheorem{prop}[theorem]{Proposition}
\newtheorem{coro}[theorem]{Corollary}

\theoremstyle{definition}
\newtheorem{defn}[theorem]{Definition}
\newtheorem{ex}[theorem]{Example}
\newtheorem{rem}[theorem]{Remark}

\numberwithin{figure}{section}
\numberwithin{equation}{section}

\newcommand\freefootnote[1]{%
	\let\thefootnote\relax%
	\footnotetext{#1}%
	\let\thefootnote\svthefootnote%
}

\newcommand{\bL}{\mathbb{L}}
\newcommand{\bX}{\mathbb{X}}
\newcommand{\la}{\langle}
\newcommand{\ra}{\rangle}
\newcommand{\bR}{\mathbb{R}}
\newcommand{\bC}{\mathbb{C}}
\newcommand{\bA}{\mathbb{A}}
\newcommand{\bS}{\mathbb{S}}

\newcommand{\bB}{\mathbb{B}}
\newcommand{\bE}{\mathbb{E}}
\newcommand{\bN}{\mathbb{N}}

\newcommand{\pa}{\partial}
\newcommand{\rt}{\rightarrow}
\newcommand{\dr}{d_\mathrm{rel}}

\newcommand{\mancac}{\mathbf{Man}^{\mathbf{c,ac}}}
\newcommand{\ssubset}{\subset\joinrel\subset}
\newcommand{\tu}{\tilde{u}}
\newcommand{\ep}{\varepsilon}
\newcommand{\tpsil}{\tilde{\Psi}_\bL}
\newcommand{\tul}{\tilde{U}_\bL}
\newcommand{\tU}{\widetilde{\Upsilon}}
\newcommand{\tPhi}{\widetilde{\Phi}}
\newcommand{\tL}{\widetilde{L}}

\newcommand{\cL}{\mathcal{L}}
\newcommand{\rel}{\mathrm{rel}}

\newcommand{\oE}{{\overline{E}}}
\newcommand{\chL}{\check{L}}

\newcommand{\wtnabla}{\widetilde{\nabla}
}

\newcommand{\ff}{\mathrm{ff}}
\newcommand{\botf}{\textrm{bf}}
\newcommand{\cC}{\mathcal{C}}
\newcommand{\cD}{\mathcal{D}}

\newcommand{\cF}{\mathcal{F}}
\newcommand{\bu}{\bar{u}}
\newcommand{\tw}{\tilde{w}}
\newcommand{\cK}{\mathcal{K}}

\newcommand{\cI}{\mathcal{I}}
\newcommand{\loc}{\mathrm{loc}}

\newcommand{\chC}{\check{C}}
\newcommand{\Bl}{\mathrm{Bl}(\mathbb{C}^m\times[0,\infty); (\mathbf{0},0))}
\newcommand{\tth}{\widetilde{\theta}}
\newcommand{\m}{\mathrm{m}}

\usepackage[T1]{fontenc}
\usepackage{imakeidx}
\makeindex
\date{}
\author{Spandan Ghosh 
	\footnote{ Mathematical Institute, University of Oxford, Oxford OX2 6GG. }
}

\title{
Short-time existence of Lagrangian mean curvature flow}
\begin{document}
\maketitle
\begin{abstract}
In his paper `Conjectures on Bridgeland Stability' \cite{Joy15}, Joyce asked if one can desingularise the transverse intersection point of an immersed Lagrangian using JLT expanders such that one gets a Lagrangian mean curvature flow via the desingularisations. Begley and Moore \cite{BM16} answered this in the affirmative by constructing a family of desingularisations and showing that a certain limit along their flows satisfies LMCF along with convergence to the immersed Lagrangian in the sense of varifolds. We prove that there exists a solution with convergence in a stronger sense, using the notion of \textit{manifolds with corners and a-corners} as introduced by Joyce \cite{Joy16}. Our methods are a direct P.D.E.\ based approach, along the lines of the proof of short-time existence for network flow by Lira, Mazzeo, Pluda and Saez \cite{LMPS23}.
\end{abstract}
\freefootnote{E-mail: {\tt ghoshs@maths.ox.ac.uk}}
\tableofcontents

\section{Introduction}
Lagrangian mean curvature flow (LMCF) is a way to deform a Lagrangian submanifold inside a Calabi--Yau manifold $X$, according to the negative gradient of the area functional. A major point of interest for studying this flow is that its stationary points are given by special Lagrangian submanifolds. However, one often runs into singularities when one tries to evolve a Lagrangian by LMCF. Motivated by string theory, Thomas and Yau \cite{TY02} proposed the Thomas--Yau conjecture, which roughly states that the question of whether a Lagrangian can be deformed to a special Lagrangian is related to an algebraic notion, namely that of stability. They conjectured that given a Lagrangian $L\subset X$ inside a Calabi--Yau manifold $X$ that is stable in the sense of Thomas--Yau, the mean curvature flow of $L$ exists for all time and converges to a special Lagrangian. A vast extension of this conjecture was proposed by Joyce \cite{Joy15}, giving a detailed conjectural picture of how singularities form in LMCF and how one might flow past them via surgery. 

In particular, one important point in the conjectural picture of \cite{Joy15} for evolving a Lagrangian via LMCF with surgery is that the Floer homology of the Lagrangian brane should remain unobstructed at all times. However, as described in \S 3.4 of \cite{Joy15}, the Floer homology can go from unobstructed to obstructed at a time $T$ due to a transverse intersection point of the immersed Lagrangian. As a way of resolving this, it is proposed to evolve the Lagrangian by doing a surgery at time $t=T$ by gluing in an expander, so that the flow can continue and the Floer homology of the new evolving Lagrangian remains unobstructed for time $t>T$. We formulate this as a general analytical problem, following Problem 3.14 of \cite{Joy15}.  

\paragraph*{\textbf{Overview of the problem.}} Let $(X,\omega,\Omega, J)$ be a given ambient Calabi--Yau manifold. Let $\iota: L \rightarrow X$ be a compact, oriented, conically singular Lagrangian embedding, with a finite number of conical singularities $\{x_i\}_{i=1}^k$, such that the corresponding smooth tangent cones $\{C_i\}_{i=1}^k$ at each singular point are a union of special Lagrangian cones. Suppose that for each cone $C_i$, there is an asymptotically conical expander $E_i$ that is asymptotic to the cone. Then, show that there is a family of Lagrangians $(L_t)_{t\in (0,\ep)}$ evolving by LMCF, such that $L_t$ converges to $L$ as $T\downarrow 0$ in a suitable sense, and we have $L_t\approx \sqrt{2t}E_i$ near each conical singularity $x_i$, and $L_t\approx L+tH_L$ away from the singularities, for $H_L$ being the mean curvature vector of $L$. Note that our formulation is more general than the previous problem as stated in \cite{Joy15}, which arises in the case the conical singularities are transverse intersection points of an immersed Lagrangian.

\paragraph*{\textbf{Outline of proof}} We approach this problem by first constructing an approximate solution. Following a similar construction to Joyce \cite{Joy04}, we construct a family of desingularizations of $L$ by gluing in the expanders at a scale of approximately $\sqrt{2t}$ to the singularities of $L$. Then, we perturb this approximate solution to an exact one. Since we are working with Lagrangians, we can invoke Lagrangian neighbourhood theorems as in \cite{Joy04}, which allows us to write the quasilinear P.D.E.\ for mean curvature flow at the level of $1$-forms, and then integrate it to obtain a fully nonlinear scalar P.D.E. We solve this nonlinear P.D.E.\ by the standard method of reformulating it as a fixed point problem between small balls in certain Banach spaces and invoking the Banach fixed point theorem. The main difficulty in this approach is to construct appropriate Sobolev spaces between which the linearised parabolic operator is an isomorphism. In particular, since the approximate solution we constructed degenerates at the singular point at the starting time, one cannot apply standard theory about solving parabolic P.D.E.s on space-time $M\times[0,\ep)$. Instead, we need to do our analysis on a modified space, on which the geometry of the approximate solution and the linearised P.D.E.\ becomes appropriately nonsingular. This is achieved by the notion of a \textit{parabolic blow-up} at the singular point in space-time, which turns out to be the natural space on which to do our analysis. In order to define these spaces globally, we introduce the notion of \textit{manifolds with corners}.

\paragraph*{\textbf{Manifolds with corners and a-corners}}
A major conceptual tool we use in this problem is the notion of manifolds with corners and a-corners, as introduced by Joyce in \cite{Joy16}, following earlier work by Melrose \cite{Mel90,Mel93,Mel96}. The main idea is that instead of thinking of the evolving Lagrangian as starting from a singularity, we can do a real oriented blow-up (as manifolds with corners) of both the domain and range. This allows us to write the evolving Lagrangian as well as the corresponding P.D.E.\ in a non-singular way on the blown-up spaces. The language of manifolds with corners and a-corners and $b$-geometry is a convenient way to package the correct notions of boundary behaviour that one needs to work with in these blown-up spaces. Our approach of using manifolds with corners in parabolic P.D.E.\ problems is also heavily inspired by the recent work of Lira, Mazzeo, Pluda and Saez \cite{LMPS23}, which we now briefly describe:

\paragraph*{\textbf{Short time existence for the network flow}} The problem of short-time existence for the network flow asks to show that given any \textit{planar network}, i.e.\ a planar graph with finitely many nodes and edges in a bounded domain in $\bR^2$, there exists a \textit{network flow} for a short time interval, i.e.\ a family of \textit{regular networks} evolving by curve shortening flow, such that as $t\downarrow 0$, the networks converge to the initial network in some sense (at least in $C^0_\mathrm{loc}$). One can further ask the set of possible networks that can flow out of a given one (in some sense of convergence). This question in the case of starting with a general irregular network was first answered by Ilmanen-Neves-Schulze \cite{INS19}, where they construct a family of regular networks by gluing in expanding solutions at the irregular vertices, then take a subsequential limit along the flows generated by this family to obtain a solution to the network flow. A new proof of their result, giving much more detailed information about how the flow `resolves' the irregular network appeared in the paper \cite{LMPS23}, where they stated and solved the problem using the language of manifolds with corners. In particular, they first construct a \textit{parabolic blow-up} $X$ of the range (which is space-time $\bR^2\times [0,\ep)$), by doing a real oriented blow up at the vertices of the initial network (c.f.\ \S 4. of this paper, where we blow up at the conical singularities).
They prove that, considering a solution to network flow as a manifold with corners $M$ along with a smooth map to $X$ satisfying the curve shortening equation, there always exists such a solution for short time. Moreover, the possible flow-outs are classified by the possible topological types of self-expanders for the curve shortening flow that are asymptotic to the tangent rays at the vertices of the initial network. In contrast to their result, we do not obtain a complete classification of flow-outs (we shouldn't expect so anyway, since the classification of possible self-expanders for higher dimensions is highly non-trivial). However, modulo the classification of self-expanders, our result is slightly stronger, since we show uniqueness as a manifold with corners and a-corners, which is a weaker smooth structure than a manifold with corners. 

\paragraph*{Isomorphism of linearised operator}
Given our approximate solution as a family of Lagrangian embeddings $\{L_t\}_{t\in(0,\ep)}$, the linearised operator can be written as \[
\pa_v-\Delta_{L_t},
\]
where $v$ is a `time-like' vector field on the family, i.e.\ it is the unique vector field that is orthogonal to the time-slices and whose flow generates the time-evolution of the family, and $\Delta_{L_t}$ is the Laplacian on $L_t$ in the induced metric from $X$. Once we compactify the approximate solution in the parabolic blow-up, a rescaling of the above P.D.E.\ becomes approximately translation-invariant near the \textit{front faces} of the compactification (note that front faces roughly correspond to asymptotically cylindrical ends in the natural \textit{$b$-metric}). Due to this, we can follow a general strategy to study Fredholmness of this operator as in Lockhart-McOwen \cite{LM85}, by introducing \textit{weighted Sobolev spaces} which are weighted by powers of the boundary-defining function for the front face. We first prove \textit{a-priori estimates} for these Sobolev spaces, by adding up \textit{interior estimates} over local models that approximate our operator with standard parabolic operators on Euclidean space. Then, we prove Fredholm results for a model space approximating the region near the front face away from some \textit{critical rates}, and glue this estimate with the interior region which is compact. The model operator on the model space is parabolic in a non-standard way, so we cannot directly apply the theory of Lockhart-McOwen to this setting. However, the general method we use to prove the result is similar, see for example the work of Agmon-Nirenberg \cite{AN63}. Note that the model space is the compactified family of expanders $\{\sqrt{2t}E\}_{t\in(0,\infty)}$, and the model operator is the same as above, i.e.\ $\pa_v-\Delta_{\sqrt{2t}E}$. So, our Fredholm result on the model space can be seen as part of a \textit{parabolic Fredholm theory} for expanders, an extension of the elliptic Fredholm theory for expanders as developed by Lotay--Neves \cite{LN13}. Indeed, we use their result as a key starting point in our proof.

\paragraph*{Main results} The main result we prove (Theorems \ref{thm_main}, \ref{thm_asymptotics}) can be stated as follows: Let $L$ be a compact, embedded conically singular Lagrangian with a finite number of singular points $x_j\in X$, such that each of its asymptotic cones is a union of special Lagrangian cones $C_j=\bigcup_{i=1}^n C_{ij}$ (which may have different phases). For each $j$, suppose that $E_j$ is a Lagrangian expander which is asymptotic to the cone $C_j$. Then, there exists a smooth family of embedded Lagrangians $\{L_t\}_{t\in (0,\ep)}$ evolving by Lagrangian mean curvature flow, such that $L_t\downarrow L$ in a strong sense to be described below, such that $L_t\approx \sqrt{2t}E_j$ near each singular point $x_j$. Moreover, the family is locally unique in the $C^2$-norm, and unique up to $t=0$ in a strong sense. Further, polyhomogeneous conormal asymptotics of the initial Lagrangian at its singularities leads to controlled polyhomogeneous conormal asymptotics of the solution at $t=0$.

Along the way to the proof and as the main analytical ingredient, we develop a parabolic Fredholm theory for asymptotically conical Lagrangian expanders in \S\ref{sec_parabolicfredholmexpanders}, which is an extension of the elliptic Fredholm theory for expanders as developed by Lotay--Neves \cite{LN13}, and may be of independent interest.

\paragraph*{Organisation of paper} We begin with recalling some background on Lagrangians, Lagrangian expanders, and neighbourhood theorems in section \ref{sec_background}. Section \ref{sec_mancac} recalls the basic definitions of manifolds with corners and a-corners we will be using in the paper, and standard estimates on parabolic Sobolev spaces. Section \ref{sec_approxsol} constructs the approximate solution to Lagrangian mean curvature flow, and computes the linearisation of the corresponding scalar P.D.E.\; Section \ref{sec_parabolicfredholmexpanders} proves a key Fredholm result for the linearised parabolic operator on asymptotically conical Lagrangian expanders. Section \ref{sec_linearproblem} proves isomorphism of the linearised operator for our problem. Section \ref{sec_nonlinearproblem} solves the full nonlinear P.D.E.\, proving our main theorem. Finally, subsection \ref{sec_polyhomogeneouscon} shows that the existence of a polyhomogeneous conormal asymptotic expansion of the initial Lagrangian implies a corresponding expansion for the solution.

\paragraph*{\textbf{Acknowledgements}} The author is grateful to his supervisor Dominic Joyce for his guidance throughout the project, and for introducing the problem as well as the main conceptual ideas. He thanks Thibault Langlais and Jason Lotay for helpful discussions, and Jason Lotay and Felix Schulze for comments on an earlier draft of this paper. This work was funded by a Clarendon Scholarship from the University of Oxford.

\subsubsection*{Notation} We collect here some of the non-standard notations used in this paper for easy reference.

\begin{table}[!h]
	\centering 
	\begin{tabular}{ c l }
		
		$\llbracket0,\infty)$ & the topological space $[0,\infty)$ endowed with an a-smooth structure
		\\
		
		$\bR^s_k$ & the model space $[0,\infty)^k\times \bR^{s-k}$  
		\\
		
		$\bR^{k,s}$ & the model space $\llbracket0,\infty)^k\times \bR^{s-k}$
	
		\\
		$\bR^{k,s}_l$ & the model space $\llbracket0,\infty)^k\times [0,\infty)^l\times \bR^{s-k-l}$
		\\
		$\mancac$ & the category of manifolds with corners and a-corners 
		\\
		$\Bl$ & the parabolic blow-up of the space $\bC^m\times[0,\infty)$ at the origin $(\mathbf{0},0)$
		\\
		{$\rho$} & \makecell[l]{the boundary-defining function for the front face of \\ $\Bl$, given by $\sqrt{|z|^2+t}$}
		
		\\
		$s$ & \makecell[l]{the boundary-defining function for the bottom face of \\ $\Bl$, given by $\frac{t}{|z|^2+t}$ }
		
		\\
		$\bX$ & \makecell[l]{the parabolic blow-up of the range $X\times[0,\ep)$ at the singularities $(x_i,0)$ \\ of $L$ at time $0$}
		\\
		$\bL$ & \makecell[l]{the compactification of the approximate solution starting from $L$ as a \\ manifold with corners and a-corners}
		\\
		
		$\bE$ & \makecell[l]{the compactification of the family of expanders $\{\sqrt{2t}E\}_{t\in(0,\infty)}$ in the \\ 
		blow-up $\Bl$ }
		
		\\
		$L^p_{k+2r,r}$ & \makecell[l]{the parabolic $L^p$-Sobolev space with $k+2r$ space derivatives and $r$ time \\derivatives}
		\\
		$C^{k+2r+\delta,r+\delta/2}$ & \makecell[l]{the parabolic H\"{o}lder space with $k+2r$ space derivatives, $r$ time \\derivatives and H\"{o}lder exponent $\delta$}
		\\
		$L^p_{k+2r,r,\lambda}$ & \makecell[l]{the parabolic Sobolev space consisting of $u\in L^p_{\loc}$ with $\rho^{-\lambda}u \in L^p_{k+2r,r}$}
		
		\\
		$C^{k+2r+\delta,r+\delta/2,\lambda}$ & \makecell[l]{the parabolic H\"{o}lder space consisting of $u\in C^{k+2r+\delta,r+\delta/2,\loc}$ with \\ $\rho^{-\lambda}u\in C^{k+2r+\delta,r+\delta/2}$}
		\\

		$^b\nabla_\rel$ & \makecell[l]{the relative Levi-Civita connection for a relative $b$-metric over a \\ $b$-fibration}
		\\
		$^{\m}\nabla$ & \makecell[l]{the m-Levi-Civita connection for the m-metric on $\bL$}
		\\
		$\pa_v$ & the time-like vector field on $\bL$
		\\
		$^bT_\rel\bL,{^bT^*}_\rel \bL$ & \makecell[l]{the relative tangent and cotangent bundles of $\bL$ over the $b$-fibration \\ $t: \bL\rt [0,\ep)$}
		\\
		$^{\m}T\bL,{^{\m}T^*\bL}$ & the mixed tangent and cotangent bundles of $\bL$
		
		\\		
		$\widetilde{L_0}$ & the radial blow-up of the initial singular Lagrangian at its singularities
		\\
		$L^p_{k,\lambda}(\widetilde{L_0})$ & \makecell[l]{the $L^p_{k}$-Sobolev space for the $b$-metric on $\widetilde{L_0}$, weighted by $\rho^{-\lambda}$}
		
	\end{tabular}
\end{table}

\section{Calabi--Yau and Lagrangian Geometry}\label{sec_background}
\subsection{Calabi--Yau geometry}
In this subsection, we recall some standard definitions from Calabi--Yau geometry. Some references for this subsection include the book by Gross, Huybrechts and Joyce \cite{GHJ03}.

\begin{defn}
	A \textit{Calabi--Yau manifold} is a five-tuple $(M,J,g,\omega,\Omega)$ where $M$ is a complex manifold of complex dimension $m$ with complex structure $J$, $\omega$ is a symplectic form on $M$ and $g$ is a Riemannian metric on $M$ such that $(J,g,\omega)$ are compatible with each other, and $\Omega$ is a nowhere-vanishing holomorphic $(m,0)$-form on $M$ satisfying \[
\frac{\omega^m}{m!} = (-1)^{m(m-1)/2}\left(\frac{i}{2}\right)^m \Omega\wedge\overline{\Omega}.	
\]
\end{defn} 

\begin{defn}
	Given a Calabi--Yau manifold $X$, 
	a \textit{Lagrangian} submanifold $L\subset X$ is a real $m$-dimensional submanifold such that $\omega|_{L}\equiv 0$. 
\end{defn}
\begin{rem}
	We can define Lagrangians more generally for symplectic manifolds.
\end{rem}
\begin{defn}
	Given a Calabi--Yau manifold $X$, a \textit{special Lagrangian} in $X$ is a Lagrangian submanifold $L\subset X$ such that $\Omega|_{L} = e^{i\theta}\textrm{vol}_L$ for a real constant $\theta$ called the \textit{phase} of $L$.
\end{defn}

\subsection{Lagrangian mean curvature flow}
In this subsection, we introduce the basic notions of Lagrangian mean curvature flow. Throughout, we assume we have a given Calabi--Yau manifold $X$. Some references for this subsection include the book by Mantegazza \cite{Man11} for general mean curvature flow, and the paper by Thomas--Yau \cite{TY02} and the survey by Smoczyk \cite{Smo12} for Lagrangian mean curvature flow.

\begin{defn}
	Given an immersed Lagrangian submanifold $L\subset X$, the \textit{mean curvature flow} of $L$ for time $t\in[0,\ep)$ is a smooth mapping $F: L\times [0,\varepsilon) \rt X $ such that it is an immersion on the time-slices, i.e.\ $F(\cdot,t)$ is an immersion,  and it satisfies \[
	\left(\frac{d}{dt}F(x,t)\right)^\perp = H_{F(\cdot,t)}(x),
	\]
	where $H$ is the mean curvature vector of the immersed submanifold, and $(\cdot)^\perp$ denotes the projection onto the tangent bundle of the immersed submanifold.
\end{defn} 

We have the following foundational result due to Smoczyk \cite{Smo96}, which shows that the evolution is in fact through Lagrangians:
\begin{theorem}
	Given a compact Lagrangian submanifold $L\subset X$, the mean curvature flow of $L$ exists for a positive time $\varepsilon>0$ and is unique. Further, the time-slices $F(\cdot,t)$ are Lagrangian submanifolds.
\end{theorem}
\begin{rem}
	The above result also holds if $X$ is K\"{a}hler-Einstein.
\end{rem}
Due to this result, we denote the flow of a Lagrangian via mean curvature flow as \textit{Lagrangian mean curvature flow}, abbreviated to LMCF.

\begin{rem}
    There is also a version of LMCF for \textit{almost Calabi--Yau} manifolds, c.f.\ the thesis of Behrndt \cite{Beh11}.
\end{rem}

\subsection{Cohomological invariants of LMCF}
In this section, we define some cohomology classes associated to a given immersed Lagrangian $L\subset  X$, and study how they change under LMCF. Some references for this section include the paper \cite{TY02}; see also the book by McDuff and Salamon \cite{MS98} for the relevant symplectic geometry. First, we recall the Lagrangian angle and the Maslov class.
\begin{defn}
	Let $L\subset X$ be an oriented Lagrangian. There is a \textit{grading}, $\theta: L \rightarrow \mathbb{S}^1$, defined using the formula $e^{i\theta}\mathrm{vol}_L= \Omega|_L$. We can also write it as a locally-defined real-valued function, well-defined up to multiples of $2\pi$, called the \textit{Lagrangian angle}. Corresponding to this, we have a globally-defined $1$-form $d\theta$ on $L$. The \textit{Maslov class} is its representative $[d\theta]$ in $H^1_{\mathrm{dR}}(L)$.
\end{defn}
\begin{defn}
    Let $L\subset X$ be an oriented Lagrangian as before. We say $L$ is \textit{zero-Maslov} if $[d\theta]$ vanishes in $H^1_{\mathrm{dR}}(L)$, or equivalently, the Lagrangian angle is a globally defined function $\theta: L\rt \bR$. We say $L$ is \textit{almost calibrated} if it is zero-Maslov and the Lagrangian angle satisfies $(\sup_{x\in L} \theta -\inf_{x\in L}\theta) < \pi$. 
\end{defn}

Let $U_L\subset T^*L$ be an open subset of the zero section of the cotangent bundle of $L$. 
Suppose we are given a Lagrangian embedding $\Psi: U_L \rightarrow X$. Then, we can identify Lagrangians in $X$ that are `nearby' to $L$ via small $1$-forms on $L$. More precisely, we have:  
\begin{theorem}
	There is a $1$-$1$ correspondence between small (in the $C^{1,\alpha}$-norm) closed $1$-forms $\alpha$ on $L$ and Lagrangians $L_\alpha$ nearby (in the $C^{1,\alpha}$-norm) to $L$, given by \[
	\alpha \mapsto L_\alpha:=\Psi(\Gamma_\alpha),
	\] 
	where $\Gamma_\alpha:=\{(x,\alpha(x) : x\in L)\}$ is the graph of the $1$-form $\alpha$.
	Moreover, two nearby Lagrangians $L_\alpha,L_\beta$ are Hamiltonian isotopic if and only if $\alpha-\beta$ is an exact $1$-form. Therefore, the Hamiltonian isotopy classes of nearby Lagrangians are in bijective correspondence with $C^1$-small forms in $H_{\mathrm{dR}}^1(L)$. In particular, the dimension of the `deformation space' of Lagrangians up to Hamiltonian isotopy is given by $b^1(L)$. 
\end{theorem}

In particular, we have that the Lagrangian given by the graph of an exact $1$-form over a Lagrangian is Hamiltonian isotopic to it. Now, we can consider the LMCF of the Lagrangian $L$, and write it using $1$-forms via the above correspondence for a short time:
\begin{theorem}\label{thm_mcfpot}
	Suppose that the family of Lagrangian immersions $\{L_t\}_{t\in(-\varepsilon,\varepsilon)}$ with $L_0=L$ satisfies LMCF. As before, let us choose a Lagrangian neighbourhood and embedding for $L$, and let the Lagrangians $L_t$ be the graphs of $1$-forms $\alpha_t$ over $L$. Then, the corresponding $1$-forms $\alpha_t$ satisfy \begin{equation}\label{class_evolve}
		\frac{d}{dt}\alpha_t\Big|_{t=0}  = d\theta,
	\end{equation}
	where $\theta$ is the Lagrangian angle of $L$. Conversely, if we have a family of Lagrangians $\{L_t\}_{t\in(-\varepsilon,\varepsilon)}$ such that we can write, for each $t\in(-\varepsilon,\varepsilon)$, $L_{t+\delta}$ as the graph of a $1$-form $\alpha_{\delta}$ over $L_t$ satisfying equation (\ref{class_evolve}) at $\delta=0$, then the family evolves by LMCF.
\end{theorem}

To prove this, we need a lemma about the geometry of the Lagrangian embedding:
\begin{lem}
	Suppose we have a Lagrangian embedding $\Psi$ for $L$ as before. Then, there is a bijection between the normal vector space to $L$ at a point $\iota(x)$ and the fibre of the cotangent bundle to $L$ at $x$, given by 
	\begin{equation}\label{eq_normalisomorphism}
		\alpha \mapsto v := (\Psi_*\alpha)^\perp.
	\end{equation}
	Note that $v$ is the unique vector in the normal bundle of $L$ such that $\omega(v,\iota_*w)=\alpha(w)$ for $w\in TL$. Under this bijection, we have that $J\nabla\theta$ corresponds to $d\theta$.
\end{lem}
\begin{proof}
	The map is well-defined due to $\omega$ being non-degenerate, and $\omega|_{\iota(L)}=0$. To see the correspondence, note that $\omega(J\nabla\theta,\iota_*w) = d\theta(w)$. To see the other characterisation of $v$, note that $\omega(\Psi_*\alpha^\perp,\iota_*w) = \omega(\Psi_*\alpha,\iota_*w) = \widehat{\omega}(\alpha,w) = \alpha(w)$, with the last equality arising from writing out $\widehat{\omega}$ in coordinates.
\end{proof}

Now, we can prove Theorem \ref{thm_mcfpot}:	

\begin{proof}
	Note that in LMCF, the Lagrangian deforms by the mean curvature vector, which is orthogonal to the Lagrangian and is given by $H=J\nabla \theta$ (for e.g.\ see Lemma 2.1 in \cite{TY02}). Therefore, making use of the isomorphism (\ref{eq_normalisomorphism}) and the above lemma, the variational field generated by $d\theta$ at $t=0$ gives rise to the same flow as $H$. Since the variational field generated by the family of graphs $\Gamma(\alpha_t)$ at time $t=0$ is given by $\frac{d}{dt}\alpha_t$, we are done.
\end{proof}

Taking the de Rham cohomology classes in (\ref{class_evolve}), we get
\begin{coro}\label{cor_classevolve}
    Suppose that the family of immersed Lagrangians satisfies LMCF. Then, we have that 
    \begin{equation}\label{eq_mcfpotcoh}
      \frac{d}{dt}[\alpha_t]=[d\theta] \; \textrm{ in } H^1_{\textrm{dR}}(L),
    \end{equation}
    for all times $t\in(-\varepsilon,\varepsilon)$. Note that the Maslov class is preserved since $\{L_t\}$ is a Lagrangian isotopy.
\end{coro}

\subsection{Asymptotically conical Lagrangian expanders}

First, we recall the definition of an expander of LMCF:
\begin{defn}
	An \textit{expander} $E\subset \mathbb{C}^m$ of LMCF is an embedded Lagrangian submanifold satisfying the elliptic P.D.E.
	\[
	(\overrightarrow{x})^\perp =  H(x),
	\]
	for $x\in E$, where $H(x)$ denotes the mean curvature vector at $x$. As a result, $E$ evolves under LMCF as $E_t = \sqrt{2(t+\frac{1}{2})}E$ for $t\in(-\frac{1}{2},\infty)$.
\end{defn}
Also, we recall the definition of asymptotically conical Lagrangians:
\begin{defn}\label{def_asymptoticallyconical}
    Let $C\subset \bC^m$ be a cone, with smooth link $\Sigma := C\cap \bS^{2m-1}$. A Lagrangian submanifold $E \subset \mathbb{C}^m$ is said to be \textit{asymptotically conical} with cone $C$ and rate $\gamma <2$, if there is a compact subset $K$ of $E$ and a diffeomorphism $\phi_E : \Sigma\times(R,\infty) \rightarrow E\backslash K$ such that \[
	|\nabla^k(\phi_E - \iota_C)| = O(r^{\gamma-1-k}),
	\]
	as $r\rightarrow \infty$, for all $k\geq 0$. Here, $\iota_C$ is the embedding of the cone $C$ in $\mathbb{C}^m$ given by $\Sigma\times(0,\infty) \ni (\sigma,r)\mapsto \sigma r$, and the norm is with respect to the induced metric on $\Sigma\times (0,R)$ from $\iota_C$.
\end{defn}

\subsubsection{Joyce--Lee--Tsui expanders}
Given an $m$-tuple of angles $\{\phi_1,\dots,\phi_m\}$, consider the oriented plane in $\bC^m$ spanned by the vectors $\{e^{i\phi_j}\overrightarrow{e_j}\}_{i=1}^m$, with orientation arising from this ordering. Note that $\Pi_0$ represents the plane $\{(z_1,\dots,z_m) : \mathrm{Im}(z_i)=0 \textrm{ for } i=1,\dots,m \}.$
Given a pair of oriented planes $\Pi_0,\Pi_\phi$, the oriented union $\Pi_\phi - \Pi_0$ is non-area-minimising if and only if the characterising angles $\{\tilde{\phi}_i\}_{i=1}^m$ between them satisfy $\phi = \sum_{i=1}^m\tilde{\phi}_i< \pi$, as shown by Lawlor \cite{Law89} (for a pair as above, the characterising angles $\{\tilde{\phi}_i\}_{i=1}^m$ are the angles $\{\phi_i\}_{i=1}^m$ in ascending order, hence $\phi = \sum_{i=1}^m\phi_i$). Given such a pair of non-area-minimising planes with a phase difference $\pi+\phi $ between them, there is an expander that is asymptotically conical with cone $\Pi_0\sqcup \Pi_\phi$, as constructed by Joyce, Lee and Tsui \cite{JLT10}, which we now describe:

\begin{theorem}
	Let $(\phi_1,\dots,\phi_m) \in (0,\pi)$ be the $m$-tuple of angles between the planes $\Pi_\phi$ and $\Pi_0$, such that $\phi=\sum{\phi_i} < \pi$. Then, there is an $m$-tuple of real constants $a_1,\dots,a_m> 0 $ such that we have \[
	\frac{1}{2}\phi_j = \int_0^\infty\frac{dt}{(\frac{1}{a_j}+t^2)\sqrt{P(t)}} \;\;\textrm{ for } i= 1,\dots,m,
	\]
	where
	$P(t) = \frac{1}{t^2}(\Pi_{k=1}^n (1+ a_k t^2)e^{t^2}-1)$. For $i=1,\dots,m$, define $w_j(y)=e^{iq_j(y)}r_j(y)$, where $r_j(y) = \sqrt{\frac{1}{a_j}+y^2}$ and $q_j(y) = \int_0^y\frac{dt}{(\frac{1}{a_j}+t^2)\sqrt{P(t)}} - \frac{1}{2}{\phi}_j$. Then, we have that \[
	L:= \{(x_1w_1(y),\dots,x_mw_m(y)) : (x_1,\dots,x_m)\in \mathbb{S}^{m-1}, y\in \bR \}
	\]
	is a closed, embedded Lagrangian expander diffeomorphic to $\bS^{m-1}\times\bR$ which is asymptotic to the two planes $\Pi_\phi$ and $\Pi_0$.
	
\end{theorem}
Moreover, this is the unique closed, exact expander asymptotic to those planes, as shown by Lotay and Neves \cite{LN13} for $m=2$, and Imagi, Joyce and Oliveira dos Santos \cite{IJO16} for $m\geq 3$.

\subsubsection{General asymptotic cones}
Consider an asymptotic cone $C=\bigcup_{i=1}^n C_i$ which is a union of special Lagrangian cones $C_i$ (which may have different phases). Let $E$ be a Lagrangian expander which is asymptotic to the cone $C$ at a rate $\gamma <2$. Due to $C$ being a union of special Lagrangian cones, we have the following from \cite{LN13}, Theorem 3.1 (Lotay--Neves prove it for $C=\Pi_\phi\sqcup \Pi_0$ being the union of two non-area-minimising planes, but the same argument works for a general special Lagrangian cone):
\begin{prop}\label{prop_expanderdecay}
	The expander decays with rate $O(e^{-\alpha r^2})$ to the cone $C$. As a result, it decays at rate  $\gamma= -\infty$ to $C$ in Definition \ref{def_asymptoticallyconical}.
\end{prop}

\begin{rem}
	This is an important geometric fact about these expanders that we will be using in our analysis. 
\end{rem}

\subsection{Conically singular Lagrangians}
In this section, we recall the definition of conically singular Lagrangians, following Definition 3.4 of \cite{Joy04}. Note that there are weaker notions of conical singularity possible, which may be relevant from the viewpoint of flowing past singularities. However, we do not consider them here.

\begin{defn}
	Let $\iota:L\rt X$ be a Lagrangian embedding, such that there exists a point $x\in X$ with Darboux chart $\Upsilon: B_R\rt X$, such that $\phi: \Upsilon^{-1}(\iota (L)) \rt B_R \subset \bC^m$ is a Lagrangian embedding.  Let $C\subset \bC^m$ be a Lagrangian cone with smooth link $\Sigma:= C\cap \bS^{2m-1}$, such that given the natural embedding $\iota_C: \Sigma \times (0,\infty) \rt \bC^m$ of $C$, there exists a diffeomorphism $\phi: \Sigma\times (0,R')\rt \Upsilon^{-1}(\iota(L))$ such that $|\nabla^k(\phi-\iota_C)|=O(r^{\mu-1-k})$ for $\mu \in (2,3)$, where the norm is computed with respect to the induced metric on $C$ from $\iota_C$. Then, we say that $L$ is a \textit{conically singular} Lagrangian, with tangent cone $C$ at $x$ and rate of decay $\mu$. 
\end{defn} 

\subsection{An overview of Begley--Moore}

Begley and Moore \cite{BM16} proved the following:
\begin{theorem}
	Let $L\subset X$ be a compact immersed Lagrangian submanifold in a Calabi--Yau manifold, with a finite number of singularities, such that each of them is asymptotic to a pair of transversely intersecting non-area-minimising planes. Then, there exists $T>0$ and an embedded Lagrangian mean curvature flow $\{L_t\}_{t\in (0,T)}$ such that as $t\downarrow T$, $ L_t\rt L$ as varifolds and in $C^\infty_{\mathrm{loc}}$ away from the singularities.
\end{theorem}

In order to prove this, they follow a similar strategy as in \cite{INS19}: they first construct an approximate family of initial Lagrangians $\{L^s\}$  to the flow via gluing in evolving expanders that are asymptotic to the tangent planes. Then, they allow these Lagrangians to flow for a short time, to get a family of flows $\{L^s_t\}_{t\in(0,\delta_s)}$. They get the desired flow $\{L_t\}_{t\in(0,\varepsilon)}$ by showing that there is a uniform lower bound for the time of existence $\delta_s$ for $s$ sufficiently small, and then taking a limit along the flows $\{L^s_t\}$. In order to show these bounds, they prove uniform estimates on the Gaussian density ratios and uniform curvature estimates for the evolving family. We remark that although Begley--Moore prove convergence as varifolds, their methods can directly lead to further strong convergence of parabolic rescalings in backward time of the flow to the expander.

\subsection{Lagrangian neighbourhood theorems}

In this subsection, we consider $\iota:L\rt X$ an immersed conically singular Lagrangian with a single conical singularity $x\in X$ with cone $C$. Let $\Sigma := C \cap \bS^{2m-1} \subset \mathbb{C}^m$ be the link of the cone, which is a smooth submanifold. We have the following neighbourhood theorem for the cone $C$, following Theorem 3.6 of \cite{Joy04}:
\begin{theorem}\label{thm1}
	There exists a Lagrangian neighbourhood of the zero section $U_C\subset T^*(\Sigma \times (0,\infty)) $ and a Lagrangian embedding $\Phi_C : U_C \rightarrow \mathbb{C}^m\backslash \{0\}$ such that $U_C$ is invariant under the $\mathbb{R}_+$-action given by 
	\begin{equation}\label{eq_epsilonscalingcone}
		\varepsilon \cdot (\sigma,r,\tau,s) = (\sigma, \varepsilon r , \varepsilon^2 \tau, \varepsilon s),
	\end{equation}
	and $\Phi_C$ is equivariant with respect to it, 
	\begin{equation}\label{eq_epsilonscalingequivariance}
		\Phi_C(\varepsilon \cdot (\sigma, r, \tau, s)) = \varepsilon \cdot \Phi_C(\sigma,r,\tau,s), 
	\end{equation}
	and 
	\begin{equation}\label{eq_epsilonscalingsymplectic}
		\varepsilon^*(\hat{\omega}) = \varepsilon^2\hat{\omega}
	\end{equation}
	(where $r,s$ are the coordinates $(\sigma,r,\tau,s)$ of $T^*(\Sigma\times(0,\infty))$, and $\hat{\omega}$ is the standard symplectic form).
\end{theorem}

Then, we have a Darboux neighbourhood theorem for the singularities of $L$, following Theorem 3.8 of \cite{Joy04}:
\begin{theorem}\label{thm3}
	Let $B_R(0) \subset \mathbb{C}^m$ denote the ball of radius $R$ in $\mathbb{C}^m$. Given an isolated conical singularity $x\in X$ of the Lagrangian $L$, there is an embedding $\Upsilon : B_R(0)\rightarrow X$, such that 
	\begin{enumerate}
		\item $\Upsilon(0)=x, \Upsilon^*\omega = \omega_{\mathbb{C}^m}, \Upsilon^*\Omega|_0=\Omega_{\mathbb{C}^m}$, $\Upsilon^*g_X = g_{\mathbb{C}^m} + O(r^2)$, where $\omega_{\bC^m}$ is the canonical symplectic form and $\Omega_{\bC^m}$ is the canonical holomorphic volume form on $\bC^m$. 
		\item $\Upsilon^{-1}(L)$ can be written as the graph of a function $\Gamma(df)$ over the cone $C\subset \bC^m$ in the image of the Lagrangian neighbourhood $\Phi_C(U_C)$, such that $|\nabla^kf|= O(r^{\mu-k})$ for $\mu \in (2,3)$, where the norm is computed with respect to the induced metric on $C$ from $\iota_C$.
	\end{enumerate}
\end{theorem}

Let $E$ be an asymptotically conical Lagrangian submanifold, with cone $C\cong\Sigma\times (0,\infty)$ as in Definition \ref{def_asymptoticallyconical}. Using Theorem \ref{thm1}, we can write $E\backslash K$ (where $K$ is a compact subset of $E$) as the graph of a closed $1$-form on $\Sigma \times (R,\infty)$ (making $R$ larger if necessary). So, there is a $1$-form $e$ on $\Sigma\times (R,\infty)$ such that $\phi(\sigma,r) = \Phi_C (\sigma, r, e_1(\sigma, r), e_2(\sigma, r))$, where $e_2 = e(\frac{\partial}{\partial r}), e_1 = e - e_1dr$.
\\

Now, we have a Lagrangian neighbourhood theorem for AC Lagrangians, following Theorem 4.5 and Theorem 4.7 of \cite{Joy04}:

\begin{theorem}\label{thm2}
	Let $E\subset \mathbb{C}^m$ be an asymptotically conical Lagrangian with a cone $C$ with link $\Sigma$, and rate of decay $\gamma<0$ as in Definition \ref{def_asymptoticallyconical}. Then, making $K,R$ larger if necessary, there is a Lagrangian neighbourhood of the zero section $U_E$, a Lagrangian embedding $\Phi_E: U_E \rightarrow \mathbb{C}^m$ and a $1$-form $e$ such that 
	
	\begin{enumerate}
		
		\item On $({\phi_E}_*)(U_E)\subset T^*(\Sigma\times (R,\infty))$, we have  \[
		\Phi_E \circ {\phi_E}_*(\sigma, r, \zeta, s) = \Phi_C (\sigma, r, \zeta + e_1(\sigma,r), s + e_2(\sigma, r)).
		\]
		(Note that ${\phi_{E}}_*  : T^*(\Sigma\times(R,\infty)) \rightarrow T^*(E\backslash K)$.) Therefore, taking the zero section, we have that $\Phi_E(E\backslash K) = \Phi_C(\sigma,r,e_1,e_2)$, i.e. $E\backslash K = \Phi_C(\Gamma(e))$.
		\item The $1$-form $e$ is exact, hence there is a potential function $\beta_{E^i}$ on each connected component $E_i$ of $E\backslash K$ such that $d\beta_{E^i} = e$, and $|\nabla^l \beta_{E^i}| = O(r^{\gamma-l})$ as $r\rightarrow \infty$, for every $l\geq 0$. In particular, we have $\beta_{E^i} \rightarrow 0$ as $r\rightarrow \infty$.
		
	\end{enumerate}
\end{theorem}

Since we will need to glue in scaled AC Lagrangians, we consider the scaling of $E$ by $\epsilon$, given as $\iota_{\epsilon E} := \epsilon \iota_E$, which is AC with the same cone and same rate. We have the following result as Corollary 2.8 of the paper by Su, Tsai and Wood \cite{STW24}, which describes the effect of scaling:

\begin{lem}\label{lem_scaling}
	Let $\Phi_E,U_E$ be as above. For $\epsilon >0$, let $f_\epsilon : T^*E \rightarrow T^*E$ be the diffeomorphism defined as $f_\epsilon (x,\alpha) := (x, \epsilon^{-2}\alpha)$. Then, the open neighbourhood $U_{\epsilon E} := f_\epsilon ^{-1}(U_E)\subset T^*E$ and the embedding \[
	\Phi_{\epsilon E} = \epsilon \Phi_E\circ f_\epsilon : U_{\epsilon E} \rightarrow \mathbb{C}^m
	\] 
	gives a Lagrangian neighbourhood of $\iota_{\epsilon E}$. Moreover, we have $\beta_{\epsilon E^i}(\sigma,r) = \epsilon^2\beta_{E^i}(\sigma, \epsilon^{-1}r)$, where $\beta$ denotes the potential defined in Theorem \ref{thm2}.
\end{lem}

\subsubsection{Compatibility of Lagrangian neighbourhoods}
In this subsection, we study how the Lagrangian neighbourhoods of the earlier section can be made compatible with a Lagrangian neighbourhood for the immersion $\iota: L\rt X$. We recall Theorem 3.9 of \cite{Joy04}.
\begin{theorem}\label{thm_compatible}
	There is a Lagrangian embedding $\Phi_L : U_L \rightarrow X$, where $U_L \subset T^*L$ is a neighbourhood of the zero section, such that if we write $\Upsilon^{-1}(L) \cap B_R$ as the graph of an exact $1$-form $dA$ over $\iota_C(\Sigma\times (0,\infty)) \equiv C$ in $B_R$ (as in Theorem \ref{thm1}), then we have that the image of the graph of a (sufficiently small) $1$-form $\alpha$ in $T^*L$ under $\Phi_L$ is given by the image of the graph of $(\iota^{-1}\circ \Upsilon\circ \phi)^*\alpha+dA$ over $C$ under the map $\Upsilon$. More precisely, we have that \[
	\Upsilon\circ \Phi_C( \sigma, r, \zeta + dA_1, s + dA_2) = \Phi_L (\sigma, r, \zeta, s),
	\]
	where $dA_2 = dA(\frac{d}{dr}),$ and $dA_1 = dA - dA_2dr$. 
\end{theorem}

 As a corollary of this theorem, we also have that the fibres of $T^*L$ are equivariant under the embeddings $\Phi_L, \Upsilon$ and $\Phi_C$.

\section{Analysis on manifolds with corners and a-corners}\label{sec_mancac}

\subsection{Manifolds with corners and a-corners}
In this section, we define some notions of differential geometry for manifolds with corners and a-corners, mainly following \cite{Joy16}, where more details can be found. These notions are based on previous work by Melrose, see for example \cite{Mel93,Mel96}.

\subsubsection{Manifolds with corners}
A manifold with corners is modelled on the space $\mathbb{R}^s_k:=[0,\infty)^k\times\bR^{s-k}$, with coordinates $(x_1,\dots,x_s)$ with $x_1,\dots,x_k$ in $[0,\infty)$ and $x_{k+1},\dots,x_s$ in $\bR$. This space has its smooth structure endowed from the inclusion $\bR^s_k\subset \bR^s$.

\begin{defn}
	A continuous map $f:U\rightarrow V$ between open subsets $U\subset \mathbb{R}^s_k, V\subset \mathbb{R}^n_l$ is called \textit{smooth} if all the derivatives of $f$ exist (we consider one-sided derivatives at the boundary) and are continuous.  
\end{defn}
A diffeomorphism between $U,V$ as above is defined in the usual way, via the existence of an inverse map. Now, we define a manifold with corners:
\begin{defn}
	A \textit{manifold with corners} is a Hausdorff, second countable topological space with charts to $\mathbb{R}^m_k$ such that the transition charts are diffeomorphisms, i.e.\ they are smooth and have smooth inverses.
\end{defn}

For a manifold with corners, we can define the notions of tangent bundle, cotangent bundle, metric, connections, etc. in a similar way as in ordinary manifolds. All these structures can be essentially obtained by restricting them from an ordinary smooth manifold.

\subsubsection{Manifolds with a-corners}
A \textit{manifold with a-corners} is modelled on the space $\bR^{k,s}:=\llbracket 0,\infty)^k\times\bR^{s-k}$, with coordinates $(x_1,\dots,x_s)$ with $x_1,\dots,x_k$ in $\llbracket0,\infty)$ and $x_{k+1},\dots,x_s$ in $\bR$. Here, the symbol $\llbracket 0,\infty)$ denotes a space homeomorphic to $[0,\infty)$, but endowed with an \textit{a-smooth structure} instead of a smooth structure.

The simplest example of such a space is $\Omega\times \llbracket 0,\ep)$, where $\Omega \subset \bR^{n-1}$ is a domain. Although this space is homeomorphic to $\Omega\times[0,\ep)$, it has a different smooth structure. For example, given coordinates $(x_1,\dots,x_n)$ on this space, the function $x_n^\alpha$ is not smooth for $\alpha\not \in \bN_0$, however it is a-smooth for all $\alpha\in \bC$ with $\Re(\alpha)>0$. A natural way to think about this space is to consider the boundary `at infinity', so that its smooth structure is similar to that of $\Omega \times (-\infty, \log(\ep))$ via the change of coordinate $x_n\mapsto \log(x_n)$, with suitable decay conditions near infinity.

For brevity, we shall not define manifolds with a-corners, but rather define \textit{manifolds with corners and a-corners}, which includes the case of manifolds with a-corners.

\subsubsection{Manifolds with corners and a-corners}

A \textit{manifold with corners and a-corners} is modelled on the space $\mathbb{R}^{k,s}_l = \llbracket0,\infty)^k\times [0,\infty)^l\times \mathbb{R}^{s-k-l}$, with coordinates $(x_1,\dots,x_s)$ with $x_1,\dots,x_k$ in $\llbracket 0,\infty)$, $x_{k+1},\dots,x_{k+l}$ in $[0,\infty)$, and $x_{k+l+1},\dots, x_{s}$ in $\mathbb{R}$. Here, the coordinates $(x_1,\dots,x_k)$ are boundary-defining functions for the \textit{a-boundaries}, while the coordinates $(x_{k+1},\dots, x_{k+l})$ are boundary-defining functions for the \textit{ordinary boundaries}. Near the a-boundaries, the smooth structure resembles that of a manifold with a-corners, while near the ordinary boundaries, the smooth structure resembles that of a manifold with corners. First, we define the notion of an a-smooth mapping, following Definition 3.18 of \cite{Joy16}.

\begin{defn}
	Let $U\subset \mathbb{R}^{k,s}_l$ be an open set and $f: U\rightarrow \mathbb{R}$ be a continuous map. It is said to have a \textit{mixed derivative} or \textit{$\m$-derivative}, denoted $^{\m}\partial f$, if the function $^{\m}\partial f:U\rightarrow \mathbb{R}^s$, written as $^{\m}\partial f= (^{\m}\partial_1 f,\dots ,^{\m}\partial f_s)$ for $^{\m}\partial_if : U\rightarrow\mathbb{R}$ defined by
	\[
	^{\m}\partial_i f (x_1,\dots, x_s)= \begin{cases}
		0, & x_i=0, \: i=1,\dots,k \\
		x_i\partial_{x_i}f &  x_i>0,\: i=1,\dots,k \\
		\partial_{x_i}f &  i = k+1,\dots, s
		
	\end{cases}
	\]
	(where we take one-sided derivatives for $x_i=0, i = k+1,\dots, k+l$) is a continuous function on $U$ for each $i=1,\dots, s$. We can iterate m-derivatives, if they exist, to get maps $^{\m}\pa^lf: U \rt \otimes^l\bR^s$ by taking m-derivatives of the components. Then,
	\begin{itemize}
		\item we say that $f$ is \textit{roughly differentiable} or \textit{$r$-differentiable} if $^{\m}\pa f$ exists, and \textit{$r$-smooth} if all the m-derivatives exist for $l=0,1,\dots$.
		\item we say that $f$ is \textit{a-differentiable}, if it is r-differentiable and we have that for any compact $S\subseteq U$ and $i=1,\dots, k$ , there are positive constants $C,\alpha$ such that \[
		|^{\m}\pa_if(x)| \leq Cx_i^\alpha \; \;\:\; \textrm{for all} \: (x_1,\dots,x_s)\in S.
		\]
		\item we say that $f$ is \textit{a-smooth} if the m-derivatives $^{\m}\pa ^lf$ are all a-differentiable for $l=0,1,\dots$.
	\end{itemize}	
\end{defn}

We note that a-smooth functions are closed under addition and multiplication, and the quotient of an a-smooth function by a non-vanishing a-smooth function is a-smooth. Also, we note that there exist functions which are r-smooth but not a-smooth, for example $f: \llbracket0,1/2) \rt \bR$ given by $f(x)= \log(x)^{-1}$, since $|^{\m}\pa ^lf| = O(|\log(x)|^{-{l-1}})$ which does not decay fast enough.

 Now, we define a-smooth maps between subsets of the model spaces, similar to Definition 3.19 in \cite{Joy16} (however, there are some key differences regarding the exponents that one allows in parts (a),(c) below -- this is to ensure that a-smooth maps are closed under composition):

\begin{defn}
	Let $U\subseteq \mathbb{R}^{k,s}_l$, $V\subseteq \mathbb{R}^{p,n}_q$ be open subsets, and $f=(f_1,\dots, f_n): U\rt V$ be a continuous map, so that $f_i:U\rt \llbracket 0,\infty)$ for $i=1,\dots p$, $f_i : U\rt [0,\infty)$ for $i = p+1,\dots, p+q$, and $f_i : U \rt \bR$ for $i=p+q+1,\dots, n$. Then,
	\begin{enumerate}
		\setcounter{enumi}{0}
		\item we say $f$ is \textit{$r$-smooth} if $f_j$ is r-smooth for $j=p+q+1,\dots, n$, and every $u = (x_1,\dots, x_s)\in U$ has an open neighbourhood $\tilde{U}\subseteq U$ such that for each $j=1,\dots, p$, we have either:
		\begin{enumerate}
			\item we can uniquely write \[f_j(\tilde{x}_1,\dots,\tilde{x}_s) = F_j(\tilde{x}_1,\dots,\tilde{x}_s)\tilde{x}_1^{a_{1,j}}\cdots \tilde{x}_{k}^{a_{k,j}} \]
			for all $\tilde{x}_1,\dots,\tilde{x}_s\in \tilde{U}$, where $\tilde{F}_j$ is a positive r-smooth function, and $a_{1,j},\dots, a_{k,j} \in [0,\infty)$ with $a_{i,j}= 0$ if $x_i\neq 0$; or
			\item $f_j|_{\tilde{U}}=0$,
		\end{enumerate}  
		and for each $j=p+1,\dots, p+q$, we have either:
		\begin{enumerate}
			\setcounter{enumii}{2}
			\item we can uniquely write \[f_j(\tilde{x}_1,\dots,\tilde{x}_s) = F_j(\tilde{x}_1,\dots,\tilde{x}_s)\tilde{x}_{1}^{b_{1,j}}\cdots \tilde{x}_{k}^{b_{k,j}}\cdot\tilde{x}_{k+1}^{b_{k+1,j}}\cdots \tilde{x}_{k+l}^{b_{k+l,j}} \]
			for all $\tilde{x}_1,\dots,\tilde{x}_s\in \tilde{U}$, where $F_j$ is a positive r-smooth function, $b_{1,j},\dots,b_{k,j} \in [0,\infty)$, and $b_{k+1,j},\dots, b_{k+l,j} \in \mathbb{N}$ with $b_{i,j}= 0$ if $x_i\neq 0$; or
			\item $f_j|_{\tilde{U}}=0$,
		\end{enumerate}  
		\item we say $f$ is \textit{a-smooth}, if the above definition holds with the functions $f_j$ for $j=p+q+1,\dots,n$ and $F_j$ for $j=1,\dots,p+q$ as above being a-smooth rather than r-smooth.
		\item we say $f$ is \textit{interior} if it is a-smooth and cases (b) or (d) in the definition do not occur.
		\item we say $f$ is \textit{$b$-normal} if it is interior, and in cases (a) and (c) in the definition, for each $i=1,\dots, k$ we have $a_{i,
			j}>0$ or $b_{i,j}>0$ for at most one $j=1,\dots, p$.
		\item we say $f$ is an \textit{a-diffeomorphism} if it is a-smooth with an a-smooth inverse.
	\end{enumerate} 
\end{defn}
We note that a-smooth functions defined as above contain identities and are closed under composition. Now, we can define manifolds with corners and a-corners:

\begin{defn}
	A \textit{manifold with corners and a-corners} is defined to be a second-countable Hausdorff topological space, with charts to the model spaces $\bR^{k,m}_l$, such that the transition maps are a-diffeomorphisms. We denote the category of manifolds with corners and a-corners, with morphisms a-smooth maps, by $\mancac$. 
\end{defn}

\begin{rem}
	It is possible to make a manifold with corners into a manifold with corners and a-corners, by defining some of its boundaries to be a-boundaries. However, not every manifold with corners and a-corners arises this way, as in Example 4.17 of \cite{Joy16}.
\end{rem}

Now, we recall the notion of ordinary boundaries and a-boundaries of a manifold with corners and a-corners, following Definition 4.1 of \cite{Joy16}:

\begin{defn}
Let $X$ be a manifold with corners and a-corners. Then, it has a natural stratification, $X= \sqcup_{k=0}^{\dim X} S^k(X)$, where $S^k(X)$ are open manifolds without corners of codimension $k$. A \textit{local $k$-corner component $\gamma$ of $X$ at a point $x$} is a choice of connected component $\gamma \in S^k(X)$ such that $x$ is contained in the closure of $\gamma$. The case $k=1$ corresponds to a \textit{local boundary component}. As sets, we define the \textit{boundary} and \textit{codimension $k$ corners} of $X$ by:
\[
\pa X := \{ (x,\beta): \beta \textrm{ is a local boundary component of $X$ at $x$}\},
\]
\[
C^k(X):= \{(x,\beta): \beta \textrm{ is a local $k$-corner component of $X$ at $x$}\}.
\]
The boundary $\pa X$ of $X$ admits a natural partition, $\pa X = \pa^{\textrm{o}}X \sqcup \pa^{\textrm{a}}X$, where $\pa^{\textrm{o}}$ is the \textit{ordinary boundary} of $X$ and $\pa^{\textrm{a}}X$ is the \textit{a-boundary} of $X$. These are defined as follows: in a chart $(U,\phi)$ for $X$ with $U\subset \bR^{k,s}_l$, the interior of $\pa X$ consists of points of the form $\phi(x_1,\dots,x_s)$ with exactly one of $\{x_1,\dots,x_{k+l}\}$ being zero, while the interior of $\pa^\textrm{o}$ consists of $\phi(x_1,\dots,x_s)$ with exactly one of $\{x_{k+1},\dots,x_{k+l}\}$ being zero, and the interior of $\pa^{\textrm{a}}X$ consists of $\phi(x_1,\dots,x_s)$ with exactly one of $\{x_1,\dots,x_k\}$ being zero.  
\end{defn}

\begin{ex}
	Consider the example of a manifold with corners and a-corners given by $X = [-1,1]\times \llbracket-1,1\rrbracket$. This has boundary the disjoint union of four pieces, given by $\pa X = (\{-1\}\times \llbracket-1,1\rrbracket) \sqcup (\{1\}\times\llbracket-1,1\rrbracket) \sqcup ([-1,1]\times \{-1\} )\sqcup ([-1,1]\times \{1\})$. The a-boundary is given by $([-1,1]\times \{-1\} )\sqcup ([-1,1]\times \{1\})$, while the ordinary boundary is given by $(\{-1\}\times \llbracket-1,1\rrbracket) \sqcup (\{1\}\times\llbracket-1,1\rrbracket)$.
\end{ex}

\subsubsection{Vector bundles, connections and metrics}
We define the notion of a vector bundle on elements of $\mancac$, as in Definition 4.4 of \cite{Joy16}:
\begin{defn}
	Let $X$ be a an element of $\mancac$. A \textit{vector bundle of rank $k$} over $X$ is an object $E\in \mancac$ with an a-smooth map $\pi : E\rt X$ such that each fibre $E_x:=\pi^{-1}(x)$ has the structure of a vector space, and there is a trivialising cover $\{U_i\}_{i\in I}$ of $X$ such that $\pi^{-1}(U_i) \cong U_i\otimes \bR^k$ and $\pi|_{\pi^{-1}(U_i)}$ is the projection $ U_i\otimes \bR^k \rt U_i$. A \textit{section} of $E$ is an a-smooth map $s: X\rt E$ such that $\pi\circ s = \mathrm{id}_X$, and we denote the vector space of a-smooth sections by $\Gamma^\infty(E)$.
\end{defn}
Now, we define the notion of the $b$-tangent bundle and the mixed tangent bundle, as in Definition 4.5$.$ of \cite{Joy16}:
\begin{defn}
	Let $X$ be an $s$-dimensional manifold with corners and a-corners. We define the \textit{mixed tangent bundle} or \textit{${\m}$-tangent bundle} of $X$ as follows: it is a rank $s$ vector bundle over $X$, whose sections on a local chart $U\subset \bR^{k,s}_l$ are spanned by $\{\{x_i\frac{\partial}{\partial x_i}\}_{i=1}^{k}\sqcup \{\frac{\partial}{\partial x_n}\}_{i=k+1}^s\}$, for $\{x_i\}_{i=1}^k$ being the boundary-defining coordinates for the a-boundary. We denote it by $^{\m}TX$.
\end{defn}
\begin{defn}
	 Let $X$ be an $s$-dimensional manifold with corners and a-corners. We define the \textit{$b$-tangent bundle} of $X$ as follows: it is a rank $s$ vector bundle over $X$, whose sections on a local chart $U\subset \bR^{k,s}_l$ are spanned by $\{\{x_i\frac{\pa}{\pa x_i}\}_{i=1}^{k+l} \sqcup \{\frac{\pa}{\pa x_i}\}_{i=k+l+1}^s\}$, for $\{x_i\}_{i=1}^k$ being the boundary-defining functions for the a-boundaries and $\{x_i\}_{i=k+1}^{k+l}$ being the boundary-defining functions for the ordinary boundaries. We denote it by $^bTX$. Note that there is a natural inclusion $^bTX \hookrightarrow {^{\m}TX}$.
\end{defn}
\begin{rem}
	Suppose that $X$ is a manifold with corners and a-corners, that also has the underlying structure of a manifold with corners, $\tilde{X}$. Then, we can think of the m-tangent bundle of $X$ to be a sub-bundle of the tangent bundle of $\tilde{X}$, spanned by the vector fields which are tangent to the a-boundaries and whose derivatives satisfy certain decay conditions at the boundary (which are satisfied if the vector field is smooth). We can think of the $b$-tangent bundle on $X$ in the same way, where the vector fields are tangent to both the ordinary boundaries and the a-boundaries.
\end{rem}

We can define the $b$-cotangent bundle and m-cotangent bundle of $X$, which are duals to the corresponding tangent bundles. As in the case of ordinary vector bundles, we can form new bundles by taking duals, tensor products, and direct sums. Now, we define the notion of an m-metric on $X$:
\begin{defn}
	An \textit{${\m}$-metric} $g$ on $X$ is a section of $\mathrm{Sym}^2(^{\m}T^*X)$, such that it is positive-definite at each point of $X$.
\end{defn}
\begin{rem}
	Note that if we consider the interior of $X$ as an ordinary manifold $X^\circ$, the m-metric restricts to an ordinary Riemannian metric, such that the a-boundaries are `at infinity' in $X^\circ$, while the ordinary boundaries are at a `finite distance'.
\end{rem}
We define the notion of an m-connection on a vector bundle, following Definition 4.14$.$ of \cite{Joy16}:
\begin{defn}
	An \textit{${\m}$-connection} on a vector bundle $E\rt X$ is an $\bR$-linear map $^{\m}\nabla : \Gamma^\infty(E)\rt\Gamma^\infty(E\otimes {^{\m}T}^*X)$, satisfying \[
	^{\m}\nabla(c\cdot e) = c\cdot {^{\m}\nabla e} + e\otimes {^{\m}\pa c},
	\]
	for $e\in \Gamma^\infty(E)$ and $c\in C^\infty(X)$.
\end{defn}
Given an m-metric, we can define a natural m-connection on $^{\m}TX$ analogously to the Levi-Civita connection, as in Proposition 5.2$.$ of \cite{Joy16}:
\begin{prop}
	Given an ${\m}$-metric $g$ on $X$, there is a natural ${\m}$-connection on $^{\m}TX$, the Levi-Civita connection, such that it restricts to the ordinary Levi-Civita connection on $X^\circ$ with respect to the Riemannian metric on $X^\circ$ given by the restriction of $g$ on $X^\circ$.
\end{prop}
Now, we recall the notion of a $b$-derivative, following definition 4.6 of \cite{Joy16}:
\begin{defn}
	Given an interior a-smooth map $f:X\rt Y$, it induces a well-defined mapping $^bdf: {^bTX}\rt {f^*(^bTY)}$, called the \textit{$b$-derivative}.
\end{defn}
\begin{rem}
	Note that for a general interior a-smooth mapping, there is no natural notion of derivative between the mixed tangent bundles, for e.g.\ consider the mapping $\llbracket 0,\ep)\rt [0,\ep)$ given by $x\mapsto x$. This is why we define the $b$-derivative only on the $b$-tangent bundles.
\end{rem}

Now, we define the following notions describing a-smooth maps, as in Definition 4.8$.$ of \cite{Joy16}:
\begin{defn}
	Let $f:X\rightarrow Y$ be an a-smooth, interior map between manifolds with corners and a-corners. It is called a \textit{$b$-submersion} if the induced $b$-derivative is surjective, and a \textit{$b$-immersion} if the induced $b$-derivative is injective. It is called is a \textit{$b$-fibration} if it is a $b$-submersion and $b$-normal. It is called a \textit{$b$-immersion} if the induced $b$-derivative is injective.
\end{defn}

\subsubsection{Embedded submanifolds}

\begin{defn}
	Let $f:X\rt Y$ be a $b$-immersion of manifolds with corners and a-corners. We call it a \textit{$b$-embedding} if $f$ is also $b$-normal.
\end{defn}
 \begin{rem}
 	This rules out $b$-immersions like $[0,\infty)\rt [0,\infty)^2$ mapping $x\mapsto (x,x)$.
 \end{rem} 
 
\subsubsection{Relative bundles over a $b$-fibration}

Let us work with a manifold with corners and a-corners $X$, along with an m-metric $g$, and a $b$-fibration of the form $\pi: X\rt [ 0,\varepsilon)$, representing a time variable. Some examples of $b$-fibrations are the mapping $f:\bR^{n-1}\times[0,\ep)\rt [ 0,\ep)$ given by $f(x_1,\dots,x_n)=x_n$, and the mapping $g: \bR^{n-2}\times\llbracket0,\ep)\times [ 0,\ep)\rt [0,\ep)$ given by $g(x_1,\dots,x_{n-1},x_n)=x_{n-1}x_n$. In this subsection, we will define the notions of the relative tangent bundle, the relative cotangent bundle, relative metric, and the relative Levi-Civita connection on $X$ with respect to the relative metric.
\begin{defn}
	The \textit{relative tangent bundle} of $X$, denoted $^bT_\mathrm{rel}X$, is defined as the sub-bundle of $^bTX$ given by $\mathrm{ker}(^bd\pi)$. This is an a-smooth sub-bundle of $^bTX$.
\end{defn}
Due to this, there are natural inclusions $^bT_\mathrm{rel}X\hookrightarrow {^bTX} \hookrightarrow {^{\m}TX}$. Now, we define the \textit{relative cotangent bundle} of $X$, $^bT^*_\mathrm{rel}X$, as the dual to the relative tangent bundle. Taking duals of the inclusion of $^bT_\mathrm{rel}X \hookrightarrow {^{\m}TX}$, we get a map $^{\m}T^*X \rt {^bT^*_\mathrm{rel}}X$ which is surjective on the fibres. This allows us to represent a-smooth sections of $ {^bT^*_\mathrm{rel}}X$ by using a-smooth sections of $^{\m}T^*X$.

\begin{defn}
	We define the \textit{relative $b$-differential} as a map \[
	^bd_\mathrm{rel} : C^\infty(X) \rt \Gamma^{\infty}(^bT^*_\mathrm{rel}X),
	\]
	by composing the $b$-differential $^bd$ with the inclusion map $^bT_\rel X\rt {^bT}X$.
\end{defn}
Now, we can similarly define the notion of a relative $b$-metric over a $b$-fibration, and use this to define a relative Levi-Civita connection.
\begin{defn}
	A \textit{relative $b$-metric} on $^bT_\rel X$ is a section of $\textrm{Sym}^2(^bT_\rel X)$ that is positive definite at each point. We define the \textit{relative Levi-Civita connection} associated to a relative $b$-metric $g$ as an $\bR$-linear map \[
	^b\nabla_\mathrm{rel}: \Gamma^\infty(^bT_\mathrm{rel}X) \rt  \Gamma^\infty(^bT_{\mathrm{rel}}X \otimes {^b}T^*_{\mathrm{rel}}X),
	\]
	defined as the ordinary Levi-Civita connection on each of the fibres.
\end{defn}
\begin{rem}
The restriction of an m-metric on $^{\m}TX$ to $^bT_\rel X$ gives us a relative $b$-metric. 	
\end{rem}

\begin{rem}
	Note that the embedding $X_t:=\pi^{-1}(\{t\})\hookrightarrow X$ gives us an isomorphism $TX_t \cong {^bT_\mathrm{rel}X}.$ Taking duals, we get an isomorphism of the cotangent bundles. 
\end{rem}
Now, given a $k$-differentiable function $u$, we can consider $^b\nabla^k_\mathrm{rel}u$ as a section of $^bT^{*\otimes k}_\mathrm{rel}X$. We have an embedding $^bT^*_\mathrm{rel}X\hookrightarrow {^{\m}T}^*X$, arising from the m-metric. Therefore, we can pushforward $^b\nabla^k_\mathrm{rel}u$ under this embedding, and take its further derivatives using the m-Levi-Civita connection, to obtain $^{\m}\nabla^l({^b\nabla_{\mathrm{rel}}^k}u)$ as a section of $^{\m}T^{*\otimes (k+l)}X$. This allows us to take the norm of $^{\m}\nabla^l({^b\nabla_{\mathrm{rel}}^k}u)$. Essentially, this is analogous to taking $l$ derivatives in time and $(l+k)$ derivatives in space.

\subsection{Interior estimates and Sobolev embeddings}
In this section, we recall some parabolic interior estimates and embedding theorems we will need, mainly following Krylov's books \cite{Kry08,Kry96}. 
\subsubsection{Parabolic Sobolev spaces}
We consider our model space as $\bR^{n+1}_0:=\{(x,t):x\in\bR^{n}, t\in [0,\infty)\}$, where $x$ represents the space variable and $t$ represents the time variable. 
\begin{defn}
	For a differentiable function $u$ having continuous derivatives of the form $\pa^k_{{i_1\dots i_k}}u$ for $(i_1,\dots, i_k)\in \{1,\dots, n\}$, we denote by $\nabla^ku$ the tensor given by \[
	\nabla^ku:= \sum_{1\leq i_1\leq \dots i_k\leq n} (\pa^k_{i_1\dots i_k}u)dx^{i_1}\dots dx^{i_k}.\]
	Similarly, for a function $u$ having continuous derivatives of the form $\pa_t^l\pa^k_{i_1\dots i_k}u$, we denote by $\pa_t^l\nabla^ku$ the tensor given by \[
	\pa_t^l\nabla^ku:=\sum_{1\leq i_1\leq \dots i_k\leq n} (\pa_t^l\pa^k_{i_1\dots i_k}u)dx^{i_1}\dots dx^{i_k}.
	\]
\end{defn}
Now, we define the $L^p_{k+2r,r}$-norms on our model space.
\begin{defn}
	Let $k,r$ be non-negative integers. For a locally integrable function $u$ having weak derivatives of the form $\pa_t^l\pa^m_{i_1\dots i_m}u$ for all $l\leq r$ and $m+2l\leq k+2r$, we define \[
	\|u\|_{L^p_{k+2r,r}}:= \sum_{i\leq l}\sum_{j\leq k+2r-2i}\|\pa_t^i\nabla^ju\|_{L^p}.
	\]
\end{defn}

We define the parabolic Sobolev spaces on $\bR^{n+1}_{0}$ as follows:
\begin{defn}
	The Sobolev space $L^p_{k+2r,r}$ is defined to be the space of locally integrable functions $u$ having weak derivatives of the form $\pa_t^l\pa^m_{i_1\dots i_m}u$ for all $l\leq r$ and $m+2l\leq k+2r$ and having a finite $L^p_{k+2r,r}$-norm.
\end{defn}
\subsubsection{Parabolic H\"{o}lder spaces}
 Let us denote the parabolic distance function $\rho(z_1,z_2)$ between two points $z_1= (x_1,t_1)$ and $z_2 = (x_2,t_2)$ as $\rho(z_1,z_2):= \sqrt{|x_1-x_2|^2+|t_1-t_2|}$. Then, we define the parabolic H\"{o}lder spaces:
\begin{defn}
	Let $Q\subset \bR^{n+1}_0$ be a domain. Then, we denote the norm \[
	[u]_{\delta,\delta/2;Q} := \sup_{z_1\neq z_2, \; z_i\in Q} \frac{|u(z_1)-u(z_2)|}{\rho^\delta(z_1,z_2)}.
	\] 
	We define the H\"{o}lder norm, for non-negative integers $k,r\geq 0$ as \[
	|u|_{C^{k+2r+\delta, r+\delta/2}(Q)} := |u|_{C^{k+2r,r}(Q)}+\sum_{2\alpha+|\beta|=k+2r, \alpha \leq r} [\pa_t^\alpha\nabla^\beta u]_{\delta,\delta/2; Q}.
	\]
\end{defn}

\subsubsection{Interior Sobolev estimates}
Let $Lu := a^{ij}\partial^2_{ij}u + b^i \partial_i u +cu$ be a second order, elliptic partial differential operator defined on a domain $\Omega \subset \mathbb{R}^n$ such that \[
\kappa |\xi|^2 \geq \sum_{i,j}a^{ij} \xi_i\xi_j \geq \kappa^{-1}|\xi|^2
\]  
for a fixed constant $\kappa$ for all $\xi \in \mathbb{R}^n$, and the $C^{k+2r-2,r-1,\alpha}$-norms of $a^{ij}, b^{i},c$ are bounded in $\Omega$ by a constant $\Lambda$. We will require two kinds of domains for our interior estimates, one near the boundary and one away from the boundary:
\begin{enumerate}
	\item Let $B_1 \ssubset B_2$ be two concentric Euclidean balls in $\Omega$ with radii $r_2>r_1>0$, and $I_i:= (-\ep_i,0)$ for $i=1,2$ be two open intervals, such that $\ep_1<\ep_2$. We denote $Q_i:=B_{i}\times I_i$ for $i=1,2$, so that $Q_1\subset Q_2$. These will be used for estimates away from the boundary $t=0$.
	\item Let $B_1\ssubset B_2$ be two open balls as before. Let $I:=[0,\varepsilon)$ be a half-open interval. We denote $P_i:=B_{i}\times I$ for $i=1,2$, such that $P_1\subset P_2$. These will be used for estimates near the boundary $t=0$.
\end{enumerate} 

\begin{theorem}\label{thm_ffestimate}
	Let $u \in L^p_{k+2r,r}(Q_2)$ and let us denote $h:= (\pa_t - L)u$. Then, we have that
	\begin{equation}\label{eq_ffestimate}
		\|u|_{Q_1}\|_{L^p_{k+2r,r} } \leq C\|h|_{Q_2}\|_{L^p_{k+2r-2,r-1}} + C'\|u|_{Q_2}\|_{L^p},
	\end{equation}
	for constants $C = C(\kappa,\Lambda)$ and $C' = C'(\kappa,\Lambda,B_1,B_2,\ep_1,\ep_2)$. 
\end{theorem}
\begin{proof}
	The case $r=1$ follows from the interior estimate in \cite{Kry08} Theorem 5.2.5, along with an estimate over an intermediate ball $Q_{R'}$. The case $r>1$ follows by induction, and compactness of mappings from a Sobolev space with higher derivatives to one with fewer derivatives.
\end{proof}

\begin{theorem}\label{thm_bfestimate}
	Let $u \in L^p_{k+2r,r}(P_2)$ and let us denote $h:= (\pa_t - L)u$. Let $v\in L^p_{k+2r,r}(P_2)$ such that $u-v\in \mathring{L}^p_{k+2r,r}(P_2)$ Then, we have that
	\begin{equation}\label{eq_bfestimate}
		\|u|_{P_1}\|_{L^p_{k+2r,r}} \leq C(\|h|_{P_2}\|_{L^p_{k+2r-2,r-1}} + \|v|_{ P_2}\|_{L^p_{k+2r,r}}) + C'\|u|_{P_2}\|_{L^p},
	\end{equation}
	for a constant $C = C(\kappa,\Lambda)$ and $C' = C'(\kappa,\Lambda,B_1,B_2,\ep)$. Moreover, the constant $C'$ is monotonically increasing in $\ep$. 
	
\end{theorem}
\begin{proof}
	The case $r=1$ follows from an a-priori estimate from \cite{Kry08} Theorem 5.2.10, along with using a cutoff function to get an interior estimate. The case $r>1$ follows by induction and compactness as before.
\end{proof}

\subsubsection{Parabolic Sobolev Embedding}

We recall some Sobolev embedding theorems for parabolic Sobolev spaces. The first two embeddings in the following result are as in Theorem 10.4 of the book by Besov, Il'in and Nikol'skiĭ \cite{BIN78}. 
\begin{theorem}\label{thm_sobolevembeddingparabolic}
	Define the quantity 
	\begin{equation}\label{eq_parabolicembeddingparameter}
	\xi:= \frac{s+(1/p)-(1/q)}{r}+n\cdot \frac{l+2s+(1/p)-(1/q)}{k+2r}.
	\end{equation}
	Let $\xi\leq 1$, and suppose that for $\xi=1$, either
	\[
	1<p\leq q <\infty\; \mathrm{ or }\; 1=p<q=\infty.
	\]
	Then, we have the embedding 
	\begin{equation}\label{eq_parabolicsobolevembedding}
	L^p_{k+2r,r}(\bR^{n+1}_0) \hookrightarrow L^q_{l+2s,s}(\bR^{n+1}_0).
	\end{equation}
	Now, suppose we define $q=\infty$ and $\xi$ as before. Then, we have the embedding 
	\begin{equation}\label{eq_parabolicckembedding}
	L^p_{k+2r,r}(\bR^{n+1}_0)\hookrightarrow C^{l+2s,s}(\bR^{n+1}_0).
	\end{equation}
	Moreover, if $\xi <1$, we have the embedding
	\begin{equation}\label{eq_parabolicholderembedding}
	L^p_{k+2r,r}(\bR^{n+1}_0)\hookrightarrow C^{l+2s+\delta,s+\delta/2}(\bR^{n+1}_0),
	\end{equation}
	for some $\delta\in(0,1)$.
	
\end{theorem}

\begin{proof}
	The main idea of the proof is define mollifications $u_\ep:= u *\phi_\ep$, and then to write $\pa_iu_\ep$ pointwise as an integral involving $\nabla^2u_\ep$ and $\pa_t u_\ep$. Then, we estimate the pointwise norms using Minkowski's and H\"{o}lder's inequalities. 
\end{proof}

\begin{rem}
	In \cite{BIN78}, the theorem is stated more generally for open subsets of $\bR^n$ satisfying a `weak $l$-horn condition'. The half-space $\bR^{n+1}_0$ satisfies this condition.
\end{rem}

\subsubsection{Traces in parabolic Sobolev spaces}

Consider the space $\bR^{n+1}_{\geq0}$. We note that there exists a bounded linear mapping, \[
T: L^p_1(\bR^{n+1}_{\geq0}) \rt L^p(\bR^n), 
\]
such that for $u\in L^p_1\cap C^0$, we have $T(u)=u|_{t=0}$. This mapping is called the trace, see for example \S 5.5 of Evans' book \cite{Eva10}. We can extend the trace map to parabolic Sobolev spaces for $p=2$ as in \S 2.5 of \cite{Kry08}, to define a bounded linear map\[
T: L^2_{k+2r,r}(\bR^{n+1}_{\geq 0})\rt L^2_{k+2r-1}(\bR^n).
\]
We also obtain a surjectivity result, following Krylov:
\begin{lem}
	Let $u\in L^2_1(\bR^n)$. Then, there exists a function $v\in L^2_{2,1}(\bR^{n+1}_{\geq0})$ such that $\|v\|_{L^2_{2,1}}\leq C\|u\|_{L^2_1}$ for a constant $C$ independent of $u$, and $T(u)=v$. 
\end{lem}
\begin{proof}
	We define $v$ to be the solution to the heat equation \[
	(\pa_t-\Delta+1)v=0, \; v_{t=0}=u.
	\]
	Then, taking Fourier transforms in space, we get for each $\xi\in \bR^n$,\[
	(\pa_t+|\xi|^2+1)\hat{v}(\xi)(t)=0,\; \hat{v}(\xi)(0)=\hat{u}(\xi).  
	\]
	This has unique solution $\hat{v}(\xi)(t)= e^{-(1+|\xi|^2)t}\hat{u}(\xi)$, from which we get 
	\begin{equation*}\begin{split}
			\|\nabla_x^2v\|_{L^2} &= \||\xi|^2e^{-(1+|\xi|^2)t}\hat{u}(\xi)\|_{L^2} = \int_{\bR^n}|\xi|^4\hat{u}^2(\xi) \int_0^\infty e^{-2(1+|\xi|^2)t}dtd\xi \\
			&= \int_{\bR^n}\frac{|\xi|^4}{2(1+|\xi|^2)}\hat{u}^2(\xi)d\xi \leq \frac{1}{2}\|\nabla_x u\|^2_{L^2}.
		\end{split}
	\end{equation*}
	Similarly, we can show that $\|v\|_{L^2_{2,1}}\leq C\|u\|_{L^2_1}$.
\end{proof}
Now, we can extend the above to domains of the form $\Omega\times[0,\ep)$, to get:
\begin{lem}
	Let $\Omega \subset \bR^n$ be a bounded domain with smooth boundary. Let $u\in L^2_1(\Omega)$. Then, there exists a function $v\in L^2_{2,1}(\Omega\times[0,\ep))$ such that $\|v\|_{L^2_{2,1}}\leq C\|u\|_{L^2_{1}}$, for a constant $C$ depending only on $\Omega$.
\end{lem}
\begin{proof}
	By Stein's extension theorem, we can first extend $u$ to a function $\tu$ on all of $\bR^n$. Then, we can use the previous lemma to get a function $v$ on $\bR^{n+1}_{\geq0}$, then restrict it to $\Omega\times[0,\ep)$.
\end{proof}

Therefore, we obtain:
\begin{lem}\label{lem_extensionoftrace}
	Let $\Omega \subset \bR^n$ be a bounded domain with smooth boundary. Let $u\in L^2_{k+2r-1}(\Omega)$. Then, there exists $v\in L^2_{k+2r,r}(\Omega\times[0,\ep))$ and a constant $C$ depending only on $\Omega$ such that \[
	\|v\|_{L^2_{k+2r,r}}\leq C\|u\|_{L^2_{k+2r-1}},\; v|_{t=0}=u.
	\]
\end{lem}
\begin{proof}
	Let us construct $v$ as a solution to the heat equation as before. Then, the estimates follow since the derivatives $\nabla^{k+2r-j}_xv$ for $2\leq j\leq k+2r$ satisfy the same heat equation, with initial data $\nabla^{k+2r-j}_xu$. 
\end{proof}

\section{Constructing the approximate solution}\label{sec_approxsol}
 
 Recall that we have a given ambient Calabi--Yau manifold $(X,\omega,\Omega, J)$, and a compact, oriented, conically singular Lagrangian embedding $\iota: L \rightarrow X$ with a finite number of conical singularities $\{x_i\}_{i=0}^{k-1}$. For simplicity of notation, we work with a single conical singularity $x_0$ -- the general case follows in the same way with only notational changes. Consider a Darboux neighbourhood and embedding $\Upsilon$ as constructed in Theorem \ref{thm3}. Suppose that the corresponding tangent cone $C$ to $\Upsilon^{-1}(\iota(L))$ at the origin is a union of special Lagrangian cones and there is an asymptotically conical expander $E$ that is asymptotic to $C$. In this section, we construct the approximate solution to LMCF starting at $L_0:=L$, by gluing in the expander $E$ at a scale of $\sqrt{2t}$ to a first order perturbation of the initial Lagrangian. First, we construct the approximate solution as a family of embedded Lagrangians in $X$. Then, we will parabolically blow-up the range, and construct the approximate solution as an embedded submanifold with corners and a-corners in the blow-up.
 \subsection{The approximate solution as a family of embeddings}\label{subsec_approxsolutionfamily}
 
 \subsubsection{Gluing the expanders}
 Recall that the expander $E$ has asymptotic cone $C=\bigcup_{j=1}^k C_j$, where the $C_j$ are special Lagrangian, and we denote the link of $C_j$ by $\Sigma_j$. 
 Let us denote the scale of gluing $\zeta(t):=\sqrt{2t}$. Choose two radii $R_1,R_2$ such that $R_1\gg0$ and $1\gg R_2>0$. For $t$ sufficiently small, we have that ${\zeta(t)}R_1<R_2$. Now, we choose two time-dependent cutoff functions $\chi_1(t,r)$ and $\chi_2(t,r)$, to be determined later (see Lemma \ref{lem_cutofffunctions}). They are assumed to satisfy the property that \begin{equation}\label{eq_cutoff}
 \chi_i(t,r) = 0 \textrm{ for } r\in[0,\zeta(t)R_1) \textrm{ and } \chi_i(t,r)=1 \textrm{ for }r\in (R_2,\infty). 
 \end{equation}
  We will construct our embedding in three pieces:
 \begin{enumerate}
 	\item The graph of a $1$-form $td\theta$ over $L\backslash \iota^{-1}(\Upsilon(B_{R_2}))$ under the embedding $\Phi_L$ as in Theorem \ref{thm_compatible}.
 	\item Consider the annulus $\Sigma_j \times ({\zeta(t)}R_1,R_2)\subset C_j$. Over these annuli, we take the graphs of 1-forms \[
 	d[ (1-\chi_1(t,r)){\zeta(t)}^2\beta_j(\sigma,{\zeta(t)}^{-1}r)+\chi_2(t,r)(A_j(\sigma,r)+t(\theta-\theta_j))]
 	\] under the embedding $\Upsilon\circ \Phi_C$ for $\Phi_C$ as in Theorem \ref{thm1}, where $\theta_j$ is the phase of the cone $C_j$ and $\theta$ is a lift of the Lagrangian angle of $L$ such that $|\theta-\theta_j|\rt 0$ at the singularity. 
 	
 	\item The scaled expander ${\zeta(t)}E \cap B_{\zeta(t)R_1}$ under the embedding $\Upsilon$.

 \end{enumerate}
 
 \begin{lem}
 	This construction defines an embedded Lagrangian submanifold in $X$ for a suitable choice of constants.
 \end{lem}
 \begin{proof}
 	Note that since we are working with graphs of closed $1$-forms under Lagrangian embeddings, the three pieces are Lagrangian submanifolds. Now, we need to check two things -- that the graphs lie in the Lagrangian neighbourhoods for embedding, and that the graphs agree on overlap. It is routine to check that these claims follow due to the choice of interpolating functions, similarly as in \S 6.1 of \cite{Joy04}. 
 \end{proof}
 We will denote the desingularisation at time $t$ by $\tL_t$.
 
 \subsubsection{Construction of Lagrangian neighbourhoods for the desingularisations}
 
 We define the Lagrangian neighbourhoods and embedding of the family of desingularisations $\{\tilde{L}_t\}_{t\in(0,\ep)}$, by describing them over the three pieces of the Lagrangian:
 \begin{itemize}
 	\item For the graph over $\iota(L_0)\backslash \Upsilon(B_{R_2})$, the Lagrangian embedding is given by $\Psi_t : U_L \rightarrow X$ such that $\Psi_t (x,v) = \Phi_L(x,v+td\theta)$.
 	\item For the embedding of the graph over $\Sigma_j\times (\zeta(t)R_1,R_2)$, we take the Lagrangian embedding $\Psi_t: U_L \rightarrow X$ such that $\Psi_t(x,v) = \Upsilon\circ\Phi_C(x,v + dp_t)$. 
 	
 	\item For the graph over the expander, we take the Lagrangian embedding $\Psi_t := \Upsilon\circ\Phi_{\zeta(t)E}: U_{\zeta(t)E}\rightarrow X$ as defined earlier in Lemma \ref{lem_scaling}, such that $\Psi_t(x,v) = \Upsilon \circ \zeta(t)\Phi_{E}(x, \zeta(t)^{-2} v).$
 \end{itemize}
 
 Note that this gives us a family of Lagrangian neighbourhoods $U_t \subset T^*(\tilde{L}_t)$ due to the compatibility Theorem \ref{thm_compatible}, and embeddings $\Psi_t : U_t \rightarrow X$. 
 
 \begin{rem} In our construction of approximate solutions and Lagrangian embeddings, the construction near the scaled expander does not extend smoothly at $t=0$, due to shrinking the expander. However, we will show that they can be extended a-smoothly over $t=0$ on doing the parabolic blow-up.
 \end{rem}
 
 \subsubsection{Hamiltonian isotopy class of desingularisations}
We will show that the constructed Lagrangians lie in the correct Hamiltonian isotopy class for LMCF, i.e.\ they satisfy \eqref{eq_mcfpotcoh} by writing $\tL_{t}$ as the graph of a $1$-form $\alpha_t$ over $\tL_{t_0}$. 

 First, we consider the scaled expander $\zeta(t)E\subset \bC^m$ as the graph of a 1-form $\Gamma(\alpha_{E_t})$ over $\zeta(t_0) E$. Then, we have that \[
 \alpha_{E_t} = (t-t_0)d\theta_E.
 \] 
 Therefore, we can write $\pa_t\alpha_t = d\theta_E$ on the scaled expander. Now, making use of the compatibility of Lagrangian neighbourhoods, we can write the form $\alpha_t$ on the gluing annulus as \[
 \alpha_t = d[(1-\chi_1(t,r)){\zeta(t)}^2\beta_j(\sigma,{\zeta(t)}^{-1}r)+\chi_2(t,r)(A_j(\sigma,r)+t(\theta-\theta_j))] - \alpha_{t_0}.
 \]
 Therefore, over the annulus we can write
 \[
 \pa_t\alpha_t = d[(1-\chi_1(t,r))(\theta_E-\theta_{C_k}) + \chi_2(t,r)(\theta_L-\theta_{C_k})],
 \]
 where $\theta_{C_k}$ denotes the choice of phase for the cone $C_k$, and $\theta_E,\theta_L$ are lifts of the Lagrangian angles so that $\theta_E-\theta_{C_k}$ goes to zero as $r\rt \infty$, and $\theta_L-\theta_{C_k}$ goes to zero as $r\rt 0$. Outside the ball $B_{R_2}$, we can write $\pa_t\alpha_t = d\theta_L$. Therefore, we can write $\pa_t\alpha_t = d\tth_t$, where $\tth_t$ is a locally-defined function which is globally well-defined up to multiples of $2\pi$, defined as:
 \begin{equation*}
 	\tth_t:=
 	\begin{cases} 
 		\theta_E & \textrm{on the scaled expander} \\
 		\theta_{C_k}+(1-\chi_1(t,r))(\theta_E-\theta_{C_k}) + \chi_2(t,r)(\theta_L-\theta_{C_k}) & \textrm{on the gluing annulus} \\
 		\theta_L & \textrm{outside the ball $B_{2R}$}
 	\end{cases}
 \end{equation*}
 
  Now, we need to show that \begin{equation}\label{eq_cohomologyclass}
 [d\tth_t]=[d\theta_t].
 \end{equation}

 Note that the cohomology class $[d\tth_t]$ is independent of $t$; this is essentially because of the decay properties of $\theta_E-\theta_{C_k}$ and $\theta_L-\theta_{C_k}$. The cohomology class $[d\theta_t]$ is also independent of $t$, since the family $\{\tL_t\}$ is a Lagrangian isotopy. Now, due to the construction of the family $\{\tL_t\}$, one can show similarly as in Proposition 6.4 of \cite{Joy04} that the difference $\tth_t-\theta_t$ modulo constant multiples of $2\pi$ has $C^0$-norm going to zero as $t\downarrow 0$. Therefore, the equality \eqref{eq_cohomologyclass} holds as $t\downarrow 0$, and due to the cohomology classes being constant in time, the equality holds at all times, i.e.\ we have shown that:
 \begin{prop}\label{prop_hamiltonisotopy1}
 	The constructed family of Lagrangians lies in the correct Hamiltonian isotopy class for LMCF. 
 \end{prop}

\subsection{The parabolic blow-up of the domain}
In this section, we define the target space of our immersion as a manifold with corners and a-corners. In particular, we will construct a parabolic blow-up of the space $X\times [0,\ep)$ at the point $(x_0,0)$ (c.f.\ \S 3 of \cite{LMPS23}).

\begin{figure}[!htb]
	\fontsize{9}{9}\selectfont
	\centering
	\def\svgwidth{\columnwidth}
	%% Creator: Inkscape 1.2.2 (b0a8486541, 2022-12-01), www.inkscape.org
%% Accompanies image file 'Parabolic blow-up.pdf' (pdf, eps, ps)
\begingroup%
  \makeatletter%
  \providecommand\color[2][]{%
    \errmessage{(Inkscape) Color is used for the text in Inkscape, but the package 'color.sty' is not loaded}%
    \renewcommand\color[2][]{}%
  }%
  \providecommand\transparent[1]{%
    \errmessage{(Inkscape) Transparency is used (non-zero) for the text in Inkscape, but the package 'transparent.sty' is not loaded}%
    \renewcommand\transparent[1]{}%
  }%
  \providecommand\rotatebox[2]{#2}%
  \newcommand*\fsize{\dimexpr\f@size pt\relax}%
  \newcommand*\lineheight[1]{\fontsize{\fsize}{#1\fsize}\selectfont}%
  \ifx\svgwidth\undefined%
    \setlength{\unitlength}{492.41319551bp}%
    \ifx\svgscale\undefined%
      \relax%
    \else%
      \setlength{\unitlength}{\unitlength * \real{\svgscale}}%
    \fi%
  \else%
    \setlength{\unitlength}{\svgwidth}%
  \fi%
  \global\let\svgwidth\undefined%
  \global\let\svgscale\undefined%
  \makeatother%
  \begin{picture}(1,0.27031948)%
    \lineheight{1}%
    \setlength\tabcolsep{0pt}%
    \put(0,0){\includegraphics[width=\unitlength,page=1]{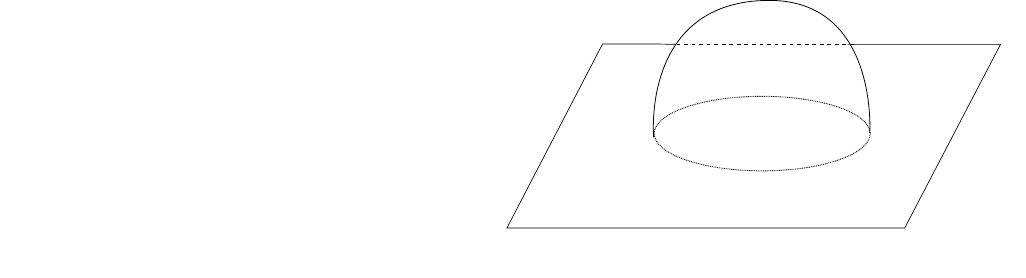}}%
    \put(0.95270202,0.16207806){\color[rgb]{0,0,0}\makebox(0,0)[lt]{\lineheight{1.25}\smash{\begin{tabular}[t]{l}$t$\end{tabular}}}}%
    \put(0.57789468,0.00898442){\color[rgb]{0,0,0}\makebox(0,0)[lt]{\lineheight{1.25}\smash{\begin{tabular}[t]{l}$\Bl$\end{tabular}}}}%
    \put(0,0){\includegraphics[width=\unitlength,page=2]{Parabolic_blow-up.pdf}}%
    \put(0.21089704,0.11753274){\color[rgb]{0,0,0}\makebox(0,0)[lt]{\lineheight{1.25}\smash{\begin{tabular}[t]{l}$(\mathbf{0},0)$\end{tabular}}}}%
    \put(0.12486961,0.01075109){\color[rgb]{0,0,0}\makebox(0,0)[lt]{\lineheight{1.25}\smash{\begin{tabular}[t]{l}$\mathbb{C}^m\times[0,\infty)$\end{tabular}}}}%
    \put(0,0){\includegraphics[width=\unitlength,page=3]{Parabolic_blow-up.pdf}}%
    \put(0.21451402,0.20606094){\color[rgb]{0,0,0}\makebox(0,0)[lt]{\lineheight{1.25}\smash{\begin{tabular}[t]{l}$t$\end{tabular}}}}%
  \end{picture}%
\endgroup%

	\caption{Parabolic blow-up of $\bC^m\times[0,\infty)$}
	\label{fig2}
\end{figure}
\subsubsection{Blow-up of $\bC^m\times[0,\infty)$} Let us write each point $((z_1,\dots,z_m),t)\in \mathbb{C}^m\times[0,\infty)$ with respect to parameters $\rho\geq0, \sigma\in \mathbb{S}^{2m}_{\geq0}$, such that we have a blow-down map $\pi: \mathbb{S}^{2m}_{\geq0}\times [0,\infty) \rightarrow \mathbb{C}^m\times[0,\infty)$, sending \[((\sigma_1,\sigma_2,..,\sigma_{2m},\sigma_{2m+1}),\rho) \mapsto ((\rho\sigma_1,\rho\sigma_2,\dots,\rho\sigma_{2m}),\rho^2\sigma_{2m+1}^2).\]
We denote this blow-up by $\Bl$. Now, we endow the blow-up with the natural a-smooth structure on $\mathbb{S}^{2m}_{\geq0}\times \llbracket 0,\infty)$, with a single a-boundary at $\rho=0$, a single ordinary boundary at $\sigma_{2m+1}=0$, and a codimension $2$ corner which is the intersection of these two boundaries. This defines the parabolic blow-up of $\mathbb{C}^m\times[0,\infty)$ at the point $(\mathbf{0},0)$, as an object in $\mancac$.

\begin{defn}The exceptional divisor, or the \textit{front face}, is the submanifold $\pi^{-1}(\mathbf{0}\times 0) \cong \mathbb{S}^{2m}_{\geq0}\times\{0\}$. The \textit{bottom face} is the submanifold $\overline{\pi^{-1}((\mathbb{C}^m\backslash \mathbf{0})\times\{0\})}$. 
\end{defn}
	
	We can write down projective coordinates on the interior of the front face and the intersection of front and bottom faces:
\begin{itemize}
    \item The front face: away from the corner, this has projective coordinates in a neighbourhood given by $(w^1,\dots,w^{2m},\tau)\in \mathbb{C}^m\times \llbracket0,\epsilon)$, such that $(w^1,\dots,w^{2m},\tau)\mapsto (w^1\tau,\dots,w^{2m}\tau,\frac{1}{2}\tau^2)$ in $\mathbb{C}^m\times[0,\infty)$ along the blow-down map.
    \item The intersection of the bottom and front faces: this has a neighbourhood with local coordinates $(\sigma_1,\dots,\sigma_{2m},r,s) \in \mathbb{S}^{2m-1}\times \llbracket0,\epsilon)\times [0,\epsilon)$, such that $(\sigma_1,\dots,\sigma_{2m}),r,s)\mapsto (\sigma_1r,\dots,$ $\sigma_{2m}r, r^2s)$ in $\mathbb{C}^m\times[0,\infty)$ along the blow-down map.
\end{itemize}

Note that these projective coordinates also locally define the smooth structure on our blow-up, i.e a function is smooth if and only if it is smooth in these projective coordinates locally.
\begin{rem}
    Note that $\tau$ is not an a-smooth function on the blow-up, since near the corner, it is given by $\tau=r (s/2)^{1/2}$, which is not a-smooth in $s$.
\end{rem}

\begin{defn}
	The front face has a natural boundary-defining function $\rho$, given by the lift of the function $\sqrt{|z|^2+t}$. The bottom face has a boundary-defining function $s=\frac{t}{|z|^2+t}$. Both functions are a-smooth on the blow-up.
\end{defn}

\begin{figure}[!htb]
	\centering
	\fontsize{9}{9}\selectfont\def\svgwidth{0.6\columnwidth}
	%% Creator: Inkscape 1.2.2 (b0a8486541, 2022-12-01), www.inkscape.org
%% Accompanies image file 'boundary faces.pdf' (pdf, eps, ps)
\begingroup%
  \makeatletter%
  \providecommand\color[2][]{%
    \errmessage{(Inkscape) Color is used for the text in Inkscape, but the package 'color.sty' is not loaded}%
    \renewcommand\color[2][]{}%
  }%
  \providecommand\transparent[1]{%
    \errmessage{(Inkscape) Transparency is used (non-zero) for the text in Inkscape, but the package 'transparent.sty' is not loaded}%
    \renewcommand\transparent[1]{}%
  }%
  \providecommand\rotatebox[2]{#2}%
  \newcommand*\fsize{\dimexpr\f@size pt\relax}%
  \newcommand*\lineheight[1]{\fontsize{\fsize}{#1\fsize}\selectfont}%
  \ifx\svgwidth\undefined%
    \setlength{\unitlength}{272.8390575bp}%
    \ifx\svgscale\undefined%
      \relax%
    \else%
      \setlength{\unitlength}{\unitlength * \real{\svgscale}}%
    \fi%
  \else%
    \setlength{\unitlength}{\svgwidth}%
  \fi%
  \global\let\svgwidth\undefined%
  \global\let\svgscale\undefined%
  \makeatother%
  \begin{picture}(1,0.44409409)%
    \lineheight{1}%
    \setlength\tabcolsep{0pt}%
    \put(0,0){\includegraphics[width=\unitlength,page=1]{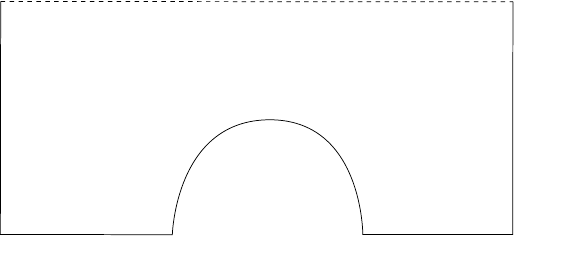}}%
    \put(0.07816814,0.00284734){\color[rgb]{0,0,0}\makebox(0,0)[lt]{\lineheight{1.25}\smash{\begin{tabular}[t]{l}$\textrm{bottom face}$\end{tabular}}}}%
    \put(0.39296462,0.14987279){\color[rgb]{0,0,0}\makebox(0,0)[lt]{\lineheight{1.25}\smash{\begin{tabular}[t]{l}$\textrm{front face}$\end{tabular}}}}%
    \put(0.92153862,0.03014943){\color[rgb]{0,0,0}\makebox(0,0)[lt]{\lineheight{1.25}\smash{\begin{tabular}[t]{l}$t=0$\end{tabular}}}}%
    \put(0.91776085,0.42626163){\color[rgb]{0,0,0}\makebox(0,0)[lt]{\lineheight{1.25}\smash{\begin{tabular}[t]{l}$t=\ep$\end{tabular}}}}%
    \put(0,0){\includegraphics[width=\unitlength,page=2]{boundary_faces.pdf}}%
    \put(0.33148052,0.29449781){\color[rgb]{0,0,0}\makebox(0,0)[lt]{\lineheight{1.25}\smash{\begin{tabular}[t]{l}$\rho$\end{tabular}}}}%
    \put(0.14959578,0.12240514){\color[rgb]{0,0,0}\makebox(0,0)[lt]{\lineheight{1.25}\smash{\begin{tabular}[t]{l}$\mathrm{s}$\end{tabular}}}}%
  \end{picture}%
\endgroup%

	\caption{Boundary-defining functions on the blow-up}
\end{figure}

\subsubsection{The m-metric on $\Bl$}
We define an m-metric on $\Bl$ as:
\[
g:= \frac{g_{\bC^m}}{\rho^2}+\frac{dt^2}{\rho^4}.
\]

In order to show this is a well-defined m-metric, we compute in local coordinates near the front face and near the corner (it is clearly an ordinary metric away from the a-boundary where $\rho >0$). Note that in $(w^1,..,w^{2m},\tau)$ coordinates, we have \[
g_{\mathbb{C}^m} = \sum {dz^i} ^2 = \tau^2 \sum {dw^i}^2 + \sum \tau^2 w^i(\tau^{-1}d\tau \otimes dw^i+dw^i\otimes \tau^{-1}d\tau)+ \tau^2\cdot\left(\sum {w^i}^2\right)(\tau^{-1}d\tau)^2,
\]
\[
dt^2 = \tau^4(\tau^{-1}d\tau)^2.
\]
And, we also have that $\rho(x,t) = \sqrt{t+|z|^2} = \tau (\frac{1}{2} + |w|^2)^{1/2}$. So, we have \[
g = s \left[ \sum {dw^i}^2 + \sum  w^i(\tau^{-1}d\tau \otimes dw^i+dw^i\otimes \tau^{-1}d\tau)+ \left(s+\sum {w^i}^2\right)(\tau^{-1}d\tau)^2\right],
\]
where $s= \frac{t}{\rho^2}=(\frac{1}{2}+|w|^2)^{-1}$.
Then, this is a positive-definite form on the front face, since the determinant of the associated matrix is $(\frac{1}{2}+|w|^2)^{-2m-2}$. Now, we can compute at the codimension-2 corner in $(\sigma,r,s)\in \bS^{2m-1}\times\llbracket 0 ,\epsilon)\times [0,\epsilon)$ coordinates with $t=r^2s$, $\rho=r$, and write the metric there to be\[
g_{\bC^m}=r^2g_{\bS^{2m-1}}+dr^2,\; \; dt^2= 4r^2s^2dr^2 + r^4ds^2+2r^2s(dr\otimes ds+ds\otimes dr),
\]
so that we have \[
g= g_{\bS^{2m-1} } + (1+4s^2)(r^{-1}dr)^2+ds^2+2s(r^{-1}dr\otimes ds + ds\otimes r^{-1}dr).
\]
This is also a positive-definite $2$-form. Therefore, $g$ is an m-metric.

\subsubsection{Defining the target space}

Consider the point of transverse intersection of $L_0=L$ in the target manifold, $x_0\in X$. We will define a parabolic blow-up $\mathbb{X}$ of $X\times [0,\ep)$ at $(x_0,0)$. Consider the Darboux map as earlier, $\Upsilon: B_R\rightarrow X$. This extends to a map $\Upsilon\times \textrm{id} : B_R(\mathbf{0})\times [0,\ep)\rightarrow X\times [0,\ep)$. \\

Since the blow-down mapping $\pi: \Bl \rightarrow \mathbb{C}^m\times[0,\infty)$ is a diffeomorphism away from $\pi^{-1}(\mathbf{0},0)$,  we have a diffeomorphism onto an open set given by
\[ 
(\Upsilon\times \textrm{id}) \circ \pi : \pi^{-1}((B_R(\mathbf{0})\times [0,\ep))\backslash (\mathbf{0},0)) \rightarrow (X\times[0,\ep)) \backslash (x_0,0).
\]

\begin{defn}\label{def_ballblowup}
	We denote $\bB:= \pi^{-1}((B_R(\mathbf{0})\times [0,\ep)))$, i.e.\ it is the parabolic blow-up of the ball $B_R\times[0,\ep)\subset \bC^m\times[0,\ep)$ at the point $(0,0)$.
\end{defn} Gluing the pieces $(X\times[0,\ep)) \backslash (x_0,0)$ and $\bB$ along $(\Upsilon\times \textrm{id}) \circ \pi$, we have defined our blown-up space $\mathbb{X}$.

Let $\pi_\mathbb{X} : \mathbb{X} \rightarrow X$ be the blow-down mappinng. We note that $\mathbb{X}$ has $\overline{\pi_\mathbb{X}^{-1}(X\backslash \{x_0\}\times \{0\})}$ as the ordinary boundary, and $\pi_\mathbb{X}^{-1}(x_0,0)$ as the a-boundary. 

\subsubsection{Choice of m-metric on $\mathbb{X}$}
We define an m-metric on $\mathbb{X}$ as the pullback of a certain metric on $(X\times[0,\ep)) \backslash (x_0,0)$, given by 
\begin{equation}\label{eq_bmetric}
g := \frac{1}{\rho (x,t)^2}g_X + \frac{1}{\rho(x,t)^4}dt^2,
\end{equation}
where $g_X$ is the Calabi--Yau metric on $X$. We define $\rho(x,t)$ by extending the earlier boundary defining function on $\Bl$:
\begin{defn}
    The boundary defining function $\rho$ for the front face of $\bX$ is defined to be
    \[
	\rho(x,t)= \begin{cases}
		\sqrt{|x|^2+t}, & (x,t)\in \pi^{-1}(\Upsilon\times \mathrm{id} (B_{R/2}(\mathbf{0})\times[0,\ep))), \\
		1, & (x,t) \not\in \pi^{-1}(\Upsilon\times \mathrm{id} (B_{R}(\mathbf{0})\times[0,\ep))),
	\end{cases}
	\]
 and $\rho$ is smoothly extended to a positive the remaining region.
\end{defn}

\subsubsection{The Symplectic Form on $\bX$}
We have a relative symplectic form on $\bX$, given by $\omega_\bX$. However, it is a degenerate form, since near the front face it is given by $\tau^2 \sum dw^i\wedge d\bar{w}^i$ in $(\tau,w^i)$-coordinates, and thus it vanishes on the front face to order $\rho^2$.

\subsection{The approximate solution in $\mancac$}\label{subsec_manac}
We will show how to compactify the approximate solution $(\tL_t)_{t\in(0,\varepsilon)}$ constructed in \S \ref{subsec_approxsolutionfamily} into a submanifold with corners $\mathbb{L}\subset\mathbb{X}$. Following the construction of the family $\{\tL_t\}$, we will construct $\mathbb{L}$ by considering the a-smooth embeddings of three pieces into $\bX$, and noting that the pieces glue together on overlaps to give a manifold with corners and a-corners embedded in $\bX$. The main upshot of this construction is stated in Proposition \ref{prop_constructionacorners}.

Due to the Darboux neighbourhood $\Upsilon$ as constructed in Theorem \ref{thm3}, there exists a smooth embedding, $\tU : \bB\rt\bX$ (where $\bB$ is as in Definition \ref{def_ballblowup}), which is a diffeomorphism onto its image and is given by $\Upsilon$ on the time-slices for $t\in(0,\ep)$.

Now, we will construct a portion of $\bL$ as an a-smooth submanifold in $\bB$, and compose it with $\tU$ to get an embedded portion of $\bL$ in $\bX$, which will glue with another part in order to give the whole embedded submanifold $\bL$.

\subsubsection{Construction in $\bB$}\label{subsubsec_cornersmooth}
Note that $\bB$ has one a-boundary, one ordinary boundary, and one codimension 2 corner. We will use coordinates on $\bB$ similarly as defined earlier on $\Bl$.

\begin{itemize}
	\item 
	Near the front face: Here,consider the scaled expander $\zeta(t)E$ inside the ball $B_{\zeta(t)R}\times\{t\}$. We can write it in projective coordinates $(w^i,\tau)$ as $(\frac{\zeta(t)}{\sqrt{2t}} E, \tau)$ inside the ball $B_{\frac{\zeta(t)}{\sqrt{2t}}R}\times \{\tau\}$ near the front face. 
	Since we have that $\zeta(t)=\sqrt{2t}$, the balls $B_{\frac{\zeta(t)}{\sqrt{2t}}R}\times \{\tau\}=B_{R}\times \{\tau\}$ compactify to a region $\{\tau \in \llbracket0,\ep), |w^i|<R\}$ near the front face, and $(\zeta(t)E\cap B_{\zeta(t)R})\times\{t\} = (E\cap B_{R})\times\{\tau\}$ for $\tau\in(0,\ep)$ also compactifies to a smoothly embedded submanifold in $(w^i,\tau)$-coordinates by extending to $(E\cap B_R) \times \{0\}$ at $\tau=0$. Therefore, the scaled expanders compactify to an a-smooth submanifold near the front face, which is diffeomorphic to $(E\cap B_R) \times \llbracket 0, \ep)$.
	 
    \item Near the corner: 
    First, let us note that a neighbourhood of the codimension $2$ corner in $\bB$ is given by a union of annuli $(B_{R_2}\backslash B_{\sqrt{2t}R_1})\times \{\tau\}$ for $\tau \in (0,\ep)$, which corresponds to $\{(\sigma,r,s):\sigma\in \bS^{2m-1}, s\in [0,\epsilon), r\in \llbracket0,\epsilon)\}$ in $(\sigma,r,s)$-coordinates introduced earlier. Now in the annulus, we defined $\tL_t$ as a graph over the cone $C \cong \Sigma\times(0,\infty)$. So, we first note that the cones $(C\cap (B_{R_2}\backslash B_{\sqrt{2t}R_1})) \times \{t\}$ compactify to an a-smooth submanifold in $\bB$ near the corner, that is a-diffeomorphic to $\mathbb{A} := \{ \sigma^i \in \Sigma, s\in [0,\varepsilon), r \in \llbracket 0,R_2)\}$ in projective coordinates.
    
    First, we need to study how the relative cotangent bundle of $\bA$ (over the $b$-fibration $t:\bA\rt [0,\ep)$) embeds into $\bB$ under the Lagrangian embedding for the cone $\Phi_C$. Let us write a-smooth coordinates on $^bT^*_\mathrm{rel}\bA$. We parametrize it as \[
    ^bT^*_\mathrm{rel}\bA = \{(\sigma,r,s,\nu,\mu): \sigma\in \Sigma, r\in \llbracket0,\ep), s\in [0,\ep), \nu\in T^*\Sigma, \mu \in \bR\},
    \]
    such that $(\sigma,r,s,\nu,\mu)$ denotes the relative $1$-form $\nu + \mu (r^{-1}dr)\in {^b}T^*_{\mathrm{rel}(\sigma,r,s)}\bA$ (note that we are implicitly using the surjection $^bT^*\bA \rt {^bT^*_\rel}\bA$), where we are considering $\nu\in T^*\Sigma$ under the mapping $\Sigma \mapsto \Sigma\times\{(r,s)\}$.
    Now, the action of $\ep$ (c.f.\ \eqref{eq_epsilonscalingcone}) in projective coordinates is given by \[
    \ep: (\sigma,r,s)\mapsto (\sigma,\ep\cdot r, \ep^{-2}s).
    \]
    Therefore, we must have from the equivariance \eqref{eq_epsilonscalingequivariance} that \[
    \ep \cdot \Phi_{C}(\sigma,r,s,\nu,\mu) = \Phi_{C}(\sigma, \ep\cdot r, \ep^{-2}s, \ep^2\nu,\ep^2\cdot \mu)
    \]
    (we get $\ep^2$ in the last term since $r^{-1}$ gains a factor of $\ep^{-1}$).
    
    Hence, we can put together the embeddings $\Phi_{C}$ for $t\in[0,\ep)$ to get a mapping, $\widetilde{\Phi}_{C} : U_\bA \subset {^bT}^*_\mathrm{rel}\bA\rt \bB$ such that we have
    \begin{equation}\label{eq_embeddingB}
    \tPhi_C(\sigma,r,s,\nu,\mu) = \left(\frac{\tPhi_C(\sigma,1,r^2s,r^{-2}\nu,r^{-2}\mu)}{\|\tPhi_C(\sigma,1,r^2s,r^{-2}\nu,r^{-2}\mu)\|}, r\cdot \|\tPhi_C(\sigma,1,r^2s,r^{-2}\nu,r^{-2}\mu)\|,s\right) \in \bB.
    \end{equation}
   	Clearly, this mapping is smooth in $(\sigma,r,s,\tau,\mu)$ away from $r=0$. However, it cannot be extended to an a-smooth function at $r=0$, due to the terms $r^{-2}\nu,r^{-2}\mu$. So, we pre-compose $\tPhi_C$ with multiplication by $r^{2}$ on each of the fibres of $^bT^*_\mathrm{rel}\bA$, to get an a-smooth embedding $\tPhi_C\circ r^2:\tilde{U}_\bA\rt\bB$. Now, given a function $u:\bA \rt \bR$ over the compactified annuli, the graphs $\Gamma(d u_t)$ over the Lagrangian embeddings $\Phi_C$ corresponds to the embedding of the graph $\Gamma(r^{-2}d_\mathrm{rel}u)$ over $^bT^*_\rel\bA$ under $\widetilde{\Phi}_C$. Therefore, we can write the family $(\tL_t)_{t\in(0,\ep)}$ near the corner as the embedding of a graph $\Gamma(r^{-2}d_\rel f)$ under $\Phi_C$, of the function $f : \mathbb{A}\rightarrow \mathbb{R}$ given in projective coordinates near the corner by \begin{equation}\label{eq_cornerdefiningfunction}
	f(\sigma, r, s) = (1-\chi_1(r,s))2r^2s\beta_{E_t}(\sigma,s^{-1/2})+\chi_2(r,s)(A_j(\sigma,r)+r^2s(\theta-\theta_j)),
	\end{equation}
	for $\beta_{E_t}$ being a smooth function that vanishes to all orders as $s\downarrow 0$. Here, $\chi_1,\chi_2$ are the cutoff functions from earlier which we will define now, as follows:
	\begin{lem}\label{lem_cutofffunctions}
		There exist smooth functions $\chi_1,\chi_2$ defined over a sufficiently small time interval $t\in [0,\ep)$, such that they satisfy \eqref{eq_cutoff}:
		 \[
		\chi_i(t,r) = 0 \textrm{ for } r\in[0,\zeta(t)R_1) \textrm{ and } \chi_i(t,r)=1 \textrm{ for }r\in (R_2,\infty),
		\] and
		\begin{equation}\label{eq_chi12}
			\chi_1(r,s) \equiv 0 \textrm{ for } r=0,\; \chi_2(r,s)\equiv 1 \textrm{ for } s=0,
		\end{equation}
		 and the corresponding function $r^{-2}f(\sigma,r,s)$ from \eqref{eq_cornerdefiningfunction} is a-smooth.
	\end{lem}
	\begin{proof}
		The first condition \eqref{eq_cutoff} translates in the coordinates $(r,s)$ near the corner to \[
		\chi_i(r,s)=0 \textrm{ for } s> (2R_1^2)^{-1}, \; \chi_i(r,s) =1 \textrm{ for } r >R_2.
		\]
		To choose such functions, we first take $\chi_1$ as a cutoff function in $r$, and $\chi_2$ as a cutoff function in $s$, so that they satisfy \eqref{eq_chi12}. Then, we reduce the time interval of existence to ensure \eqref{eq_cutoff} holds.
	\end{proof}
	
	 Now, by the following Lemma, it follows that $\Upsilon_t^{-1}(\tilde{L}_t)$ near the corner, given by the embedding of $\Gamma(r^{-2}d_\mathrm{rel}f)$ under $\tPhi_c\circ r^2$, is an a-smooth embedded submanifold in $\bB$. 
    \begin{lem}
        The relative $1$-form $r^{-2}d_\mathrm{rel}f$ extends to an a-smooth  section of $^bT^*_\mathrm{rel}\bA$, if $r^{-2}f$ is an a-smooth function on $\bA$.
    \end{lem}
    \begin{proof}
        This follows since $d_\mathrm{rel}(r^{-2}) = O(r^{-2})$, as can be seen by computing in local coordinates.
    \end{proof}
    Therefore, the embeddings $\Upsilon_t^{-1}(\widetilde{L}_t)$ near the corner compactify to an a-smooth embedding $\bA\rt\bB$. Since this overlaps with the scaled expander near the front face as considered in the previous point, we will get a gluing between $\mathbb{A}$ and $(E_0\cap B_R)\times \llbracket 0,\ep)$, which is a-smooth. Composing this with $\tU$ gives us the a-smooth embedding of $\bL$ in a neighbourhood of the front face in $\bX$.

	\item Finally, we provide charts for $\mathbb{L}$ away from the front face; This is the region outside $B_{R}$. The charts here are given simply by mapping $L_0\times[0,\ep)\rt \bX$, $(x,t)\mapsto \Gamma(td\theta(x))$. There are smooth transition maps to the image of $\mathbb{A}$.
\end{itemize}
This defines the structure of a manifold with corners $\mathbb{L}$, by gluing up the images of $(E\cap B_R)\times\llbracket0,\ep), \bA$ and $\tilde{L}\times[0,\ep)$, along with an a-smooth embedding $\iota: \mathbb{L}\rightarrow \mathbb{X}$ such that $\iota(\mathbb{L}\backslash \partial\mathbb{L}) = \sqcup_{t\in(0,\varepsilon)}\widetilde{L}_t$. 

\begin{defn}
    The \textit{bottom face} of $\bL$ is its ordinary boundary, which is closure of the radial blowup of $L= L_0$ at $x_0$, and the \textit{front face} of $\bL$ is its a-boundary, which is the closure of the expander $E$. We define an m-metric on $\bL$ via the pullback of the m-metric on $\mathbb{X}$ along $\iota$.
\end{defn}
\begin{rem} The main reason why we do not define the bottom face of $\bL$ as an a-boundary is that near the bottom face, the linearised operator is a multiple of the standard parabolic operator $\pa_t-\Delta$ -- if the bottom face were an a-boundary, we would expect the operator near it to be of the form $t\pa_t-\Delta$.
\end{rem}

\begin{rem}
	By abuse of notation, we will continue to denote the boundary-defining function for the bottom face as $s$ and the boundary-defining function for the front face as $\rho$.
\end{rem}

\subsubsection{The symplectic form and Lagrangian embedding}

Since we have the $b$-fibration $t:\bL\rt [0,\ep)$, we can define $^bT^*_\mathrm{rel}\bL$ as a vector bundle over $\bL$, and it is itself a manifold with corners and a-corners along with a $b$-fibration $t: {^bT^*_\rel\bL\rt[0,\ep)}$. 
\begin{defn}There is a canonical relative symplectic form $\hat{\omega}$, i.e a non-degenerate element of $\Lambda^2(^bT_\mathrm{rel}^*(^bT^*_\mathrm{rel}\bL))$. Given local coordinates $(x^i,y^i,t)$ on $^bT^*_\mathrm{rel}\bL$, we can write it as $\hat{\omega}=\sum_{i=1}^mdx^i\wedge dy^i$.
\end{defn}

Under the embedding $T^*L_t\hookrightarrow {^bT^*_\mathrm{rel}}\bL$, this symplectic form pulls back to the standard symplectic form on $T^*L_t$. Also, note that it is non-degenerate, i.e$.$ we have that $\hat{\omega}(v,\cdot) \not \equiv 0$ for $v\in {^bT_\mathrm{rel}}({^bT^*_\mathrm{rel}\bL})$ and $v\not = 0$.

The Lagrangian embeddings $\Psi_t$ on the time-slices for $t\in(0,\ep)$ combine to give us an embedding of a neighbourhood of $^bT^*_\mathrm{rel}\bL$ over the interior of $\bL$, as $\Psi_{\bL^\circ}: U_{\bL^\circ}\rt\bX$. However, this does not extend to an a-smooth embedding over the a-boundary, because of the scaling by $\rho^{-2}$ near the corner (as explained in the second point of the previous subsection \S \eqref{subsubsec_cornersmooth}). In order to make it a-smooth, we need to pre-compose this map with multiplying each fibre by $\rho^{2}$. 

\begin{lem}
	Consider the map \[
	\rho^{2} : T^*_\mathrm{rel}\bL \rt T^*_\mathrm{rel}\bL, 
	\]
	given by multiplying each fibre by $\rho(x)^{2}$, i.e.\ it sends $(x,\beta)\mapsto(x,\rho(x)^2\beta)$. Then, we have that $\tilde{U}_\bL:=\overline{(\rho^{2})^{-1}(U_{\bL^\circ})}$ is an open neighbourhood of the zero section of $T^*_\mathrm{rel}\bL$, and $\tilde{\Psi}_\bL:= \Psi_{\bL^\circ}\circ \rho^{2}$ extends to an a-smooth embedding of $\tilde{U}_\bL$ into $\bX$.
\end{lem}
\begin{proof}
	This follows in the same way as we argued in \S \ref{subsubsec_cornersmooth}. One needs to work in local coordinates near the front face of $\bL$, and note the form of the embedding $\Psi_\bL$ near the front face, as given in \eqref{eq_embeddingB}, to conclude that $\tpsil$ is a-smooth.
\end{proof}
Note that the graph of a function $\Gamma(d_\mathrm{rel}u)\subset U_\bL$ corresponds to the graph $\Gamma(\rho^{-2}d_\mathrm{rel}u)\subset \widetilde{U}_\bL$.

\begin{rem}
	We can write the modified map $\tilde{\Psi}_\bL$ in a cleaner way by considering it as a map from $^bT^*_\mathrm{rel}\bL\otimes L_2 $ to $\bX$, where $L_2$ is the weight $2$ line bundle on $\bL$, as defined in \S 4.4 of \cite{Joy16}. $L_2$ has a flat connection $\nabla$ and a flat section $s$, such that if we identify $L_2$ with the trivial line bundle $\bR$ over $\bL$, then $s$ corresponds to a section with a zero of order $2$ at the front face. Then, we can define a map,\[
	^bT^*_\mathrm{rel}\bL \otimes L_2\rt {^bT}^*_\mathrm{rel}\bL,
	\]
	as $\mathrm{id}\otimes (s\mapsto \rho^2)$.  Composing this map with the Lagrangian embeddings on the time-slices $\Psi_t$, we get a map $(\mathrm{id}\otimes (s\mapsto \rho^2))\circ \Psi$, which gives us an a-smooth embedding of a neighbourhood of the zero-section of $^bT^*_\mathrm{rel}\bL\otimes L_2$.
\end{rem}

Therefore, we have the following:
\begin{prop}\label{prop_constructionacorners}
	Let $\iota: L \rt X$ be a given embedded conically singular Lagrangian, with a single conical singularity at $x_0\in X$ such that the asymptotic cone is a union of special Lagrangian cones, and let $E$ be an asymptotically conical expander with the same asymptotic cone. We can construct a manifold with corners and a-corners $\bL$, along with an embedding $\bL \hookrightarrow \bX$, such that $\bL|_{\botf} = \overline{\widetilde{L_0}}$, and $\bL|_{\ff} = \overline{E}$. This induces a natural $b$-fibration $t:\bL\rt[0,\ep)$. Moreover, there is a Lagrangian neighbourhood $\widetilde{U}_\bL$ of the zero section of $^bT^*_\rel \bL$, along with an a-smooth embedding $\tpsil: \widetilde{U}_{\bL} \rt \bX$, which is Lagrangian up to a factor of $\rho^2$, i.e.\ such that $\tpsil^*(\omega_\bX) = \rho^{2}\hat{\omega}$. 
\end{prop}
\begin{figure}[!htb]
	\centering
	\fontsize{9}{9}\selectfont\def\svgwidth{0.7\columnwidth}
	%% Creator: Inkscape 1.2.2 (b0a8486541, 2022-12-01), www.inkscape.org
%% Accompanies image file 'compactification.pdf' (pdf, eps, ps)
\begingroup%
  \makeatletter%
  \providecommand\color[2][]{%
    \errmessage{(Inkscape) Color is used for the text in Inkscape, but the package 'color.sty' is not loaded}%
    \renewcommand\color[2][]{}%
  }%
  \providecommand\transparent[1]{%
    \errmessage{(Inkscape) Transparency is used (non-zero) for the text in Inkscape, but the package 'transparent.sty' is not loaded}%
    \renewcommand\transparent[1]{}%
  }%
  \providecommand\rotatebox[2]{#2}%
  \newcommand*\fsize{\dimexpr\f@size pt\relax}%
  \newcommand*\lineheight[1]{\fontsize{\fsize}{#1\fsize}\selectfont}%
  \ifx\svgwidth\undefined%
    \setlength{\unitlength}{321.56554215bp}%
    \ifx\svgscale\undefined%
      \relax%
    \else%
      \setlength{\unitlength}{\unitlength * \real{\svgscale}}%
    \fi%
  \else%
    \setlength{\unitlength}{\svgwidth}%
  \fi%
  \global\let\svgwidth\undefined%
  \global\let\svgscale\undefined%
  \makeatother%
  \begin{picture}(1,0.58271795)%
    \lineheight{1}%
    \setlength\tabcolsep{0pt}%
    \put(0,0){\includegraphics[width=\unitlength,page=1]{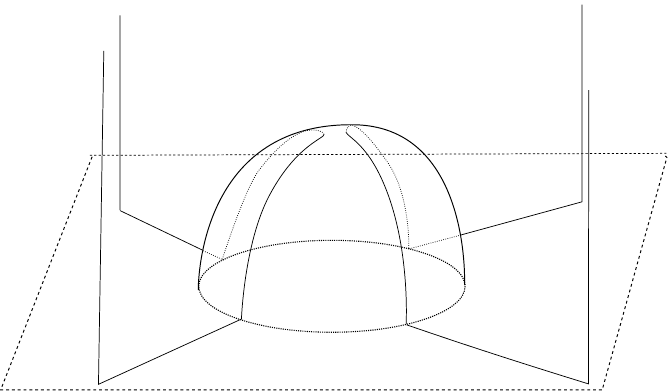}}%
    \put(0.89277133,0.56884727){\color[rgb]{0,0,0}\makebox(0,0)[lt]{\lineheight{1.25}\smash{\begin{tabular}[t]{l}$t=\ep$\end{tabular}}}}%
    \put(0.95435985,0.14847742){\color[rgb]{0,0,0}\makebox(0,0)[lt]{\lineheight{1.25}\smash{\begin{tabular}[t]{l}$t=0$\end{tabular}}}}%
    \put(0,0){\includegraphics[width=\unitlength,page=2]{compactification.pdf}}%
  \end{picture}%
\endgroup%

	\caption{The compactification $\bL \subset \bX$. This has one a-boundary and one ordinary boundary.}

\end{figure}

\subsection{The Geometry of $t:\bL \rt [0,\ep)$}
We have that the time variable on $\bL$ defines an a-smooth mapping, $t: \bL\rt [0,\ep)$, which is a $b$-fibration. We can study the geometry of this fibration via its `local models'.

\subsubsection{The local models for $\bL$}\label{subsubsec_localmodels}
In this subsection, we will study what our metric looks like locally on $\bL$. More precisely, we construct `model neighbourhoods' of $\bL$, where the metric is of $\bL$ is close to a collection of model metrics for manifolds with corners and a-corners. First, we define the local model spaces:
\begin{enumerate}
	\item $U\times(a,b)$, where $U\subset \bR^m$ is an open ball, we have a time variable $t\in(a,b)$, and we have a metric given by $g_0 + dt^2$.
		\item  $U\times[0,\delta)$, where $U\subset \bR^m$ is an open ball and $U\times \{0\}$ is an ordinary boundary, we have a time variable $t\in[0,\delta)$, and we have a metric given by $g_0 + dt^2$.
	\item $U\times\llbracket0,\delta)$, where $U\subset \bR^m$ is an open ball and $U\times \{0\}$ is an a-boundary, we have a time variable $t\in \llbracket 0,\delta)$, and we have a $b$-metric given by $g_0 + t^{-2}dt^2$.
	\item $V\times \llbracket0,\delta)\times [0,\delta)$, parametrized by $(\sigma,r,s)$, where $V\subset \bR^{m-1}$ is an open ball, $r=0$ is an a-boundary, $s=0$ is an ordinary boundary, we have a time variable $t=r^2s$, and we have an m-metric $g_0+ds^2+r^{-2}dr^2$.

\end{enumerate}
Now, we claim the following: 
\begin{lem}\label{lem_modelneighbourhood}
	Let $\varepsilon_1>0$ be a given positive constant. Then, there is a finite collection of local models $\{(M_i,g_i):i\in I\}$, along with maps $\Psi_i: M_i\rt \bL$ from local models such that their images are open, they are a-diffeomorphisms onto their images, $\bigcup_{i\in I}\Psi_i(M_i)$ gives a covering of $\bL$, the time variable of $M_i$ matches with the time variable of $\bL$, and $\|\Psi_i^*g_\bL-g_i\|_{C^0} < \varepsilon_1$. Moreover, we have that $\|\nabla^k(\Psi_i^* g_\bL-g_i)\|_{C^0} < C_k$ for $k\geq 1$ for uniform constants $C_k$, where the connection and norm is with respect to the metric $g_i$.
\end{lem}
\begin{proof}
	We use an inductive process to choose the local models. To choose a finite cover, we work with the compactification $\overline{\bL}$ of $\bL$ given by adding the time-slice at $t=\ep$.
	\begin{enumerate}
		\item Let $\{U_i\}_{i=1}^n$ be a finite covering of the codimension-two corner $\bS^{m-1}\sqcup\bS^{m-1} $ of $\bL$. Then, we can find charts $\Psi_i: U_i\times\llbracket0,\delta)\times [0,\delta) \rt \bL$, mapping $(\sigma,r,s)\mapsto (\sigma r, r^2s)$ such that $\Psi_i$ is an isometry at a point $\{\sigma_0,0,0\}$. This is because we can find such charts for the compactification of the expander near the corner, and $\bL$ is an a-smooth perturbation of it. Then, it follows that $|\Psi_i^*g_\bL-{^bg}_0|<\varepsilon$, $|^b\nabla^k(\Psi_i^*g_\bL-{^bg}_0)|<C$, and the time variable matches.
		\item Let $\{V_i\}_{i=1}^m$ be a finite cover of the front face minus the previous covering, such that we have charts $\Phi_i : V_i\times\llbracket0,\delta) \rt \bL$, mapping $(w^i,\tau)\mapsto (\tau w^i,\frac{1}{2}\tau^2)$, such that it is an isometry at a point $(x_0,0)$. We can pre-compose it with the a-diffeomorphism $t\mapsto t^{\frac{1}{2}}=\tau$ so that the time variable matches.
		\item Let $\{W_i\}_{i=1}^p$ be a finite  cover of the bottom face minus the first covering, so we have charts $\Phi_i : W_i \times [0,\delta) \rt \bL$ by flowing along the time-like vector field such that it is an isometry at a point $(x_0,0)$, and the time variable matches.
		\item The rest of the manifold can now be covered by finitely many balls of the form $U\times(a,b)$ (since $\overline{\bL}$ is compact) such that the metric is close in them, and the time variable matches.
	\end{enumerate}
	This gives us a finite covering of $\bL$ by model neighbourhoods.
\end{proof}
\begin{rem}
	The above proof can be carried out for more general $b$-fibrations, by considering local models with an arbitrary number of ordinary boundaries and a-boundaries.
\end{rem}

Having these local models, statements on $\bL$ can be proven by proving corresponding statements for the local models.

\subsubsection{The time-like vector field}
\begin{defn}We define the \textit{time-like vector field} over the interior $\mathbb{\mathring{L}}$ as a vector field $v\in \Gamma(T\mathring{\bL})$ such that $\pi_*(v)=\pa_t$, and $v$ is orthogonal to the time-slices in the m-metric.
\end{defn}
 This is not an a-smooth vector field on $\bL$, since it does not vanish at the front face. However, we can make it a-smooth by multiplying with a weight:
\begin{prop}
	 The vector field $\rho^2v$ is a non-vanishing a-smooth section of $^{\m}T\bL$. 
\end{prop}
\begin{proof}
	To prove this, it is enough to prove the statement for the corresponding local models near the front face, since the vector fields are clearly smooth away from the front face. Then, we have:
	\begin{enumerate}
		\item For the model $U\times\llbracket 0,\delta)$, $v = \pa_t, \rho = C\sqrt{t}$ for a non-vanishing a-smooth function $C$, hence $\rho^2\pa_v= Ct\pa_t$, which is a-smooth and non-vanishing.
		\item For the model $U\times \llbracket0,\delta)\times[0,\delta)$, the time-like vector field is $(r^{-2}\pa_s+2sr^{-1}\pa_r)(1+4s^2)^{-1}$, and $\rho = Cr$. Therefore, $\rho^2\pa_v=C(1+4s^2)^{-1}(\pa_s-2sr\pa_r)$, which is a-smooth and non-vanishing.
	\end{enumerate}
\end{proof}
\begin{rem}
	The fact that $\rho^2\pa_v$ is non-vanishing as an a-smooth vector field is relevant to the linearised operator being strongly parabolic.
\end{rem}
Now, considering the embedding $\{\tL_t\}_{t\in(0,\ep)}\subset X\times(0,\ep)$ we can decompose 
\begin{equation}\label{eq_vdecomposition}
v=\pa_t+v^\perp,
\end{equation}
such that $v^\perp$ lies in the normal bundle of $\tL_t$. 

\subsection{Sobolev and H\"{o}lder spaces on $\bL$}\label{subsec_sobholspaces}
In this section, we define some natural Sobolev and H\"{o}lder spaces on $\bL$ over the $b$-fibration $t:\bL\rt [0,\ep)$ using the a-smooth structure on $\bL$ in which $\rho=0$ is the a-boundary, with the corresponding m-metric on $\bL$. As before, we can consider the relative tangent bundle of $\bL$, given by $\mathrm{ker}(^bd\pi)$, which is a sub-bundle of $^{\m}T\bL$. There is a natural relative metric on $^bT_\rel\bL$ induced from the m-metric on $\bL$, and corresponding natural Levi-Civita connections $^{\m}\nabla$ on $^{\m}T\bL$ and $^b\nabla_\mathrm{rel}$ on $^bT_\mathrm{rel}\bL$.

\begin{prop}
	We have the following about the relative constructions on $\bL$:
	\begin{enumerate}
		\item The relative tangent bundle $^bT_\rel\bL$ is an a-smooth vector sub-bundle of $^{\m}T\bL$, and is therefore itself an a-smooth vector bundle on $\bL$. As a corollary, so are its duals and tensor products.
		\item The relative $b$-metric is an a-smooth section of $\textrm{Sym}^2(^bT^*_\rel\bL)$. 
	\end{enumerate}
\end{prop}
\begin{proof}
	As before, it suffices to check a-smoothness on the local models for $\bL$ near the front face.
	\begin{enumerate}
		\item In the local model $U\times \llbracket0,\ep)$, the relative tangent bundle is the sub-bundle spanned by $\pa_{x_i}$, which are a-smooth sections. The relative $b$-metric is given by $\sum \pa_{x_i}^2$.
		\item In the local model $V\times\llbracket0,\ep)\times[0,\ep)$, the relative tangent bundle is spanned by $\pa_{x_i}, (\rho\pa_\rho-2s\pa_s)$. The relative $b$-metric is given by $\sum d{x_i}^2+(\rho^{-1}d\rho-sds)^2$.  These sections are a-smooth.
 \end{enumerate}
\end{proof}
\begin{coro}
	As a corollary to the relative metric and tangent bundle being a-smooth, we get that the relative Levi-Civita connection is an a-smooth $b$-connection.
\end{coro}

 For $1\leq p \leq \infty$, and integers $k,r\geq 0$, we define our Sobolev spaces on $\bL$ as $L^p_{k+2r,r}(\bL)$, using the norm: \[
\|u\|_{L^p_{k+2r,r}}:= \left( \sum_{\beta\leq r} \left[  \sum_{|\alpha|\leq k-2\beta}\int_{\mathbb{L}} |{^{\m}\nabla}^\beta{^b\nabla}_{\textrm{rel}}^\alpha u|^p dV_{\mathbb{L}} \right] \right)^{1/p} = \left(\int_0^\ep \int_{\tL_t} |{^{\m}\nabla}^\beta{^b\nabla}_{\textrm{rel}}^\alpha u|^p \rho^{-m}dV_{\tL_t}dt \right)^{1/p},
\]
where $dV_{\bL}$ is the m-volume form induced on $\bL$ from the m-metric as a section of $\Lambda^{n}({^{\m}T^*\bL})$, and $dV_{\tL_t}$ denotes the volume form on $\tL_t$ induced from its embedding into $X$. Note that we have the following H\"{o}lder's inequality for our spaces:\[
\|fg\|_{L^r_{k+2r,r}} \leq \|f\|_{L^p_{k+2r,r}}\|g\|_{L^q_{k+2r,r}}
\]
for $\frac{1}{r}=\frac{1}{p}+\frac{1}{q}$. 

\begin{rem}
	When $k=r=0$, we will often drop the subscripts $k,r$ and denote our space by $L^p(\bL)$. 
\end{rem}

We will define the H\"{o}lder distance between points $x,y\in \bL$ for $d_\bL(x,y)$ sufficiently small as follows:\[
d(x,y) := \sqrt{d_{\bL}(x,y)^2+\int_0^l|\la\gamma'(s),\rho^2\pa_v(\gamma(s))\ra|ds },
\] 
where $d_{\bL}$ is the distance function for the $b$-metric, and $\gamma: [0,l]\rt \bL$ is the geodesic in the $b$-metric joining the points $x,y$. The key point about this definition is that it localises correctly to the usual notion of H\"{o}lder distances on the model space $\bR^{n+1}_{\geq 0}$. Then, we can also define H\"{o}lder spaces on $\mathbb{L}$, given by $C^{k+2r+\delta,r+\delta/2}(\mathbb{L})$, with norm\[
\|u\|_{C^{k+2r+\delta,r+\delta/2}(\mathbb{L})} := \sup_{x\in \mathbb{L}, \alpha + 2\beta \leq k+2r, \beta \leq r}|^{\m}\nabla^\beta{^b\nabla}_{\textrm{rel}}^\alpha u| + [^{\m}\nabla^r{^b\nabla}_{\textrm{rel}}^k u]_{\delta/2;\delta},
\]
where $[\cdot]_{\delta/2;\delta}$ is the norm given by \[
[v]_{\delta/2,\delta}:= \sup_{d_\bL(x,y)<\textrm{inj}_\bL, \; x \neq y}\frac{|v(x)-v(y)|}{d(x,y)^\delta}.
\]

\subsubsection{Weighted spaces}\label{subsubsec_weightedordinaryspaces}
Now, for a weight $\lambda \in \mathbb{R}$, define the weighted Banach spaces $L^p_{k+2r,r,\lambda}(\mathbb{L})$ using the norm: \[
\|u\|_{L^p_{k+2r,r,\lambda}}:= \|\rho^{-\lambda}u\|_{L^p_{k+2r,r}}.
\]

Note that we have the following H\"{o}lder's inequality for our spaces:\[
\|fg\|_{L^r_{k+2r,r,\lambda}} \leq \|f\|_{L^p_{k+2r,r,\lambda'}}\|g\|_{L^q_{k+2r,r,\lambda''}}
\]
for $\frac{1}{r}=\frac{1}{p}+\frac{1}{q}$ and $\lambda = \lambda'+\lambda''$. 

Similarly, we define weighted H\"{o}lder spaces on $\mathbb{L}$, given by $C^{k+2r+\delta,r+\delta/2,\lambda}(\mathbb{L})$, with norm\[
\|u\|_{C^{k+2r+\delta,r+\delta/2,\lambda}(\mathbb{L})} := \sup_{x\in \mathbb{L}, \alpha + 2\beta \leq k+2r, \beta \leq r}|\rho^{-\lambda}\cdot{^{\m}\nabla}^\beta{^b\nabla}_{\textrm{rel}}^\alpha u| + [\rho^{-\lambda}\cdot{^{\m}}\nabla^r{^b}\nabla_{\textrm{rel}}^k u]_\delta.
\]
\begin{rem}
	\begin{enumerate}
		\item We have that $^{\m}\nabla^k(\rho^{-\lambda}u) = \rho^{-\lambda} {^{\m}\nabla^k}u + \sum C_j {^{\m}\nabla^j}(\rho^{-\lambda}){^{\m}\nabla^{k-j}}u $, in which $|{^{\m}\nabla^j}\rho^{-\lambda}| = O(\rho^{-\lambda})$ and $C_j$ are bounded a-smooth functions (similarly for $b$-derivatives). Therefore, we can equivalently define the weighted spaces using norms \[
		\|u\|_{L^p_{k+2r,r,\lambda}}:= \left( \sum_{\beta\leq r} \left[  \sum_{|\alpha|\leq k-2\beta}\int_{\mathbb{L}} |\rho^{-\lambda}\cdot{^{\m}\nabla}^\beta{^b\nabla}_{\textrm{rel}}^\alpha u|^p dV_{\mathbb{L}} \right] \right)^{1/p}.
		\]
		\item We have that $\rho^{\lambda}\in L^p_{0,0,0}$ if and only if $\lambda > 0$. This is because by computation of $\int_{\mathbb{L}}\rho^\lambda dV_\bL$, near the front face we get a term of the form $\int_0^\varepsilon\rho^{\lambda}\rho^{-1}d\rho$, which is finite if and only if $\lambda >0$.
		\item 	When $k=r=0$, we will often drop the subscripts $k,r$ and denote our weighted spaces by $L^p_{\lambda}(\bL)$. 
	\end{enumerate}
\end{rem}

We also note another useful result about our Sobolev spaces:
\begin{lem}
	Let $C^\infty_{\ff}$ be the subset of compactly supported smooth functions on $\mathbb{L}$ supported away from the front face. Then, we have that it is dense in $L^p_{k+2r,r,\lambda}$ for $1\leq p<\infty$.
\end{lem}
\begin{proof}
	The key point is that since we use a natural m-metric and m-volume form on $\bL$, we can construct cutoff functions which interpolate to $0$ in a neighbourhood of the front face, $1$ outside of a neighbourhood of the front face, and such that all its derivatives are uniformly bounded and supported in an arbitrarily small neighbourhood of the front face. Multiplying a given function with such cutoff function proves the claim.  
\end{proof}

\subsection{The LMCF equation on potentials and linearisation}
Let $u  : \bL\rt\bR$ be a function such that $\|u\|_{C^{2,0,2}}$ is bounded by a constant that ensures the graphs $\Gamma(d_{\mathrm{rel}}u_t)\subset T^*\tL_t$ are contained in the Lagrangian neighbourhoods for $\tL_t$. In this subsection, we will write the LMCF equation as a P.D.E. in terms of $u$, such that the embeddings of the graphs of $u_t$ under the Lagrangian embeddings, i.e.\ $\Psi_t(\Gamma(d_\rel u_t))$, is a family of Lagrangians evolving by LMCF. We think of this as a perturbation of the constructed approximate solution $\bL$ to an exact solution. We also compute the linearisation of this P.D.E.\ at the approximate solution $\bL$.

\subsubsection{The P.D.E.\ for LMCF}

 Let $u$ be a function as before, with $\|u\|_{C^{2,0,2}}$ sufficiently small. There is a natural diffeomorphism $\sigma_{d_\mathrm{rel}u}: {^bT}^*_\mathrm{rel}\mathbb{L} \rightarrow {^bT}^*_\mathrm{rel}\mathbb{L}$, given by $(x,\alpha(x)) \mapsto (x,\alpha(x) +d_{\mathrm{rel}}u(x))$, preserving the relative $2$-form $\hat{\omega}$. In particular, this gives a map $\sigma_{\dr u} : \bL^\circ = \Gamma(0)\rt\Gamma(\dr u)$. Now, we have the Lagrangian embedding over the interior, $\Psi : T^*_{\mathrm{rel}}\bL^\circ \rt \bX$. Consider the pushforward along the diffeomorphism $\Psi\circ \sigma_{\dr u}$ of the time-like vector field $v$ on $\mathbb{L}$, given by $(\Psi\circ \sigma_{\dr u})_*v$. Then, $(\Psi\circ \sigma_{\dr u})_*v$ is a variational vector field in $\bX$ for the Lagrangian embedding of $\Gamma(\dr u)$. So, $(\Psi\circ \sigma_{\dr u})_*v - \pa_t $ is a relative vector field on $\bX$, such that under the projection $\pi : \bX\rt X$ away from the front face, it gives a variational vector field for the perturbed Lagrangians. Hence, the perturbations satisfy LMCF if and only if we have\[
d(\theta_{\Gamma(\dr u)}) = -(\Psi\circ \sigma_{\dr u})^*\omega_\bX((\Psi\circ \sigma_{\dr u})_*v - \partial_t, \rule{0.25cm}{0.05mm}),
\] 
where $\theta(\Gamma(\dr u)) : \bL^\circ \rt \bR$ is the Lagrangian angle corresponding to the perturbation $\Gamma(\dr u)$. Now, due to $(\Psi\circ \sigma_{\dr u})$ being a symplectomorphism with respect to the relative symplectic forms, we can write \[
(\Psi\circ \sigma_{\dr u})^*\omega_\bX((\Psi\circ \sigma_{\dr u})_*v - \partial_t, \rule{0.25cm}{0.05mm})= \hat{\omega}(v - (\Psi\circ \sigma_{\dr u})^*\partial_t, \rule{0.25cm}{0.05mm}).
\] 
Putting these together, we get the P.D.E.
\begin{equation}\label{1}
	d_\mathrm{rel}(\theta_{\Gamma(du)}) = -\hat{\omega}(v- {(\Psi\circ \sigma_{\dr u})^*(\partial_t)}, \rule{0.25cm}{0.05mm} )|_{\bL},
\end{equation}
which relates relative $1$-forms on $\mathbb{L}$. We will compute the decomposition of this P.D.E.\ into zeroth order, linear and higher order nonlinear terms. First, we introduce the following:
\begin{defn}
	For a function $u:\bL\rt\bR$, define a function $\tu$ on $^bT^*_\rel\bL$, as $\tu(x,\beta):=u(x)$. Via the Lagrangian embedding, this extends to a function on a neighbourhood $\tu:\tpsil(U_{\bL})\rt \bR$.
\end{defn}

The main result we will prove is the following:
\begin{prop}
	We have the following decompositions into zeroth order, linear and nonlinear parts:
	\begin{enumerate}
		\item the Lagrangian angle decomposes as  \begin{equation}\label{eq_decomptheta}
			d_\rel(\theta_{\Gamma(du)})= d_{\rel}\theta_0+d_{\rel}\left(\Delta_{L_t} u + H\tu \right) + d_\rel(Q_1[\rho^{-2}d_\rel u, \rho^{-2}\nabla^2_\rel u]),\end{equation}
		
		\item the symplectic form decomposes as\begin{equation}
			-\hat{\omega}(v- {(\Psi\circ \sigma_{\dr u})^*(\partial_t)}, \rule{0.25cm}{0.05mm} )|_{\bL} = \hat{\omega}(v^\perp,\_)|_{\bL} + d_\rel(\pa_vu+v^\perp \tu) +d_\rel(Q_2[\rho^{-2}d_\rel u]),
		\end{equation}
	\end{enumerate}

	where $H$ is the mean curvature vector, $v^\perp$ is the component of $v$ as defined in $\eqref{eq_vdecomposition}$, and $Q_1,Q_2$ are nonlinear a-smooth functions without any zeroth order or linear terms.
\end{prop}

By our construction of approximate solution, we have the following corollary to Proposition \ref{prop_hamiltonisotopy1}:
\begin{coro}
	The relative $1$-form $-d_\rel\theta_0+ \hat{\omega}(v^\perp,\_)|_{\bL}$ is exact.  
\end{coro}

Now, we will prove the claimed decompositions.

\subsubsection{The m-connection on $^{\m}T(^bT^*_\rel\bL)$}

First, we introduce an a-smooth connection on the relative cotangent bundle, analogously to \S5.2 of \cite{Joy04}. Let $\overline{\nabla}$ be an a-smooth m-connection on $^bT^*_\rel \bL\downarrow \bL$. This allows us to define `horizontal subspaces' in $^{\m}T(^bT^*_\rel\bL)$, via parallel transport of relative $1$-forms. Using this, we have a canonical decomposition \begin{equation}\label{eq_decomposition}{^{\m}}T_{(x,\beta)}(^bT^*_\rel\bL) \cong {^{\m}T_x\bL} \oplus {^b}T^*_{x,\rel}\bL,
\end{equation}  
with the first term $^{\m}T_x\bL$ corresponding to the horizontal subspace. Now, we define a connection $\wtnabla$ on $^{\m}T(^bT^*_\rel\bL)$ by defining differentiation along the horizontal subspace using $\overline{\nabla}$ and $^b\nabla$ (which is the $b$-Levi-Civita connection on $^bT\bL$), and partial differentiation along the vertical subspace as in the vector space $^bT^*_{x,\rel}\bL$. This is well-defined, since $^bT^*_\rel\bL$ is naturally trivial along each fibre. Then, since we have defined this invariantly, and since the connections $\overline{\nabla},{^b\nabla}$ are a-smooth, we have:
\begin{lem}
	The above defines an a-smooth ${\m}$-connection $\wtnabla$ on $^{\m}T(^bT^*_\rel\bL)$.
\end{lem}
\begin{rem}
	Note that this connection may have non-vanishing torsion.
	\end{rem}

Now, we will define the connection $\overline{\nabla}$, by `restricting' the Levi-Civita connection on $^bT^*\bL$ to the relative cotangent bundle:
\begin{defn}
	Define the connection $\overline{\nabla} : \Gamma^\infty(^bT^*_\rel\bL) \rt \Gamma^\infty(^bT^*_\rel\bL \otimes {^{\m}T^*\bL})$, by taking the connection $^b\nabla_{\rel}$ and composing it with ${^bT^*_\rel \bL} \rt {^{\m}T}^*\bL$ induced from projection against the m-metric. This is a well-defined a-smooth m-connection.
\end{defn}

\subsubsection{The Lagrangian angle}
Now, we prove that the Lagrangian angle decomposes as claimed. Instead of the embedding $\Psi$, we work with the embedding $\tpsil = \Psi \circ \rho^2$, since this is an a-smooth embedding of the relative cotangent bundle. As before, let $u:\bL\rt \bR$ be an a-smooth function such that $\Gamma(\rho^{-2}d_\mathrm{rel}u)$ lies in the neighbourhood $\tilde{U}_\bL$ and embeds into $\bX$ under $\tilde{\Psi}_\bL$.

\begin{prop}
	There is an underlying a-smooth function $F : \mathrm{Sym}^2(^bT^*_\mathrm{rel}\bL)\oplus \tul\rt \bR$ such that $\theta(\Gamma(du)) = \theta_0+ F(\rho^{-2}{^b\nabla^2_\mathrm{rel}u}, \rho^{-2}{^b\nabla}_\mathrm{rel}u)$, for $\|u\|_{C^{2,0,2}}$ being sufficiently small.
\end{prop}
\begin{proof}
	Using the isomorphism \eqref{eq_decomposition} induced by $\widetilde{\nabla}$, we will associate to each point $(\alpha,\beta,x)\in \mathrm{Sym}^2(^bT^*_\mathrm{rel}\bL)\oplus \tul$, a Lagrangian $m$-plane $L'\subset T_{\mathrm{rel},(\alpha,x)}(T^*_\mathrm{rel}\bL)$. To define this, take a basis of vectors $\{e_i\}\in T_{\mathrm{rel},x}\bL$, and consider $\{\alpha(e_i)\}\in T^*_{\mathrm{rel},x}\bL$. The plane $L'$ is spanned by the vectors $\{e_i+\alpha(e_i)\}$ via the above isomorphism. Note that $L'$ is a Lagrangian plane due to $\alpha$ being symmetric.
	
	Now, we define $\theta(L')$ to be the Lagrangian angle of $L'$ under the embedding $\tpsil$ into $\bX$. Then, we define $F(\alpha,\beta,x)$ to be the difference $\theta(L')- \theta_0(x)$. In order to show $F$ is a-smooth, we note that the Lagrangian plane $L'$ depends a-smoothly on $(\alpha,\beta,x) $.
\end{proof}

Now, we can write $F$ on an open subset of each fiber of the vector bundle $\mathrm{Sym}^2(^bT^*_\mathrm{rel}\bL)\oplus {^bT^*}_\mathrm{rel}\bL$ as the sum of a linear term and a nonlinear, second order term:
\[
F(\alpha,\beta,x)= L(\alpha,\beta,x)+Q(\alpha,\beta,x),
\]
such that $Q(0,0,x)=\pa_\alpha Q(0,0,x)=\pa_\beta Q(0,0,x)=0$. Following Proposition 9.10 of \cite{Beh11}, the linearization is given by \[
\Delta_{L_t} u - \la \nabla \theta, V(du)\ra,
\]
where $\Delta_{L_t}$ is the Laplacian on $L_t$ for the induced metric from the embedding, and $V(du)$ is the `deformation field' for the embedding, i.e.\ it is the vector $\pa_s\Phi(x,sdu(x))|_{s=0}$ in $X$. By the identity $H=J\nabla\theta$, the claimed linearization \eqref{eq_decomptheta} follows.

\subsubsection{The symplectic form}
 Now, we will show an analogous theorem for the symplectic form:

\begin{prop}
	We have the following decomposition,\[
	-\hat{\omega}(v- {(\Psi\circ \sigma_{\dr u})^*(\partial_t)}, \rule{0.25cm}{0.05mm} )|_{\bL} = \hat{\omega}(v^\perp,\rule{0.25cm}{0.05mm})|_{\bL} + d_\rel(\pa_vu +v^\perp(\tu) +Q[\rho^{-2}d_\rel u]),
	\]
	for $Q$ being a nonlinear function with no linear or constant terms depending a-smoothly on $\rho^{-2}d_\rel u$.
\end{prop} 

\begin{proof}

	Let us  consider local coordinates on $^bT^*_\rel\bL$ given by $(x^i,y^i,t)$, such that $v=\pa_t$, $y^i=dx^i$. Then, the symplectic form is $\sum dx^i\wedge dy^i$. Now, let us write the pull back vector field $\Psi^*(\pa_t)$ as \[
	\Psi^*(\pa_t) = \pa_t + a^i\pa_{x^i}+b^i\pa_{y^i}, 
	\]
	for locally-defined functions $a^i,b^i$. Then, we note that due to $\Psi$ being a family of Lagrangian embeddings, the relative $1$-form given by $\hat{\omega}(a^i\pa_{x^i}+b^i\pa_{y^i}, \rule{0.25cm}{0.05mm})$ is closed on the time-slices $U_{L_t}$. This implies that there is a locally-defined function $f: {^bT^*_\rel\bL}\supset U_\bL \rt \bR$ such that $d_\rel f = \hat{\omega}(a^i\pa_{x^i}+b^i\pa_{y^i}, \rule{0.25cm}{0.05mm})$ on $U_\bL$ (note that we are taking the relative derivative on the whole cotangent bundle, rather than just on $\bL$).
	Now we note that in these coordinates,
	\begin{equation}
		\begin{split}
	-\hat{\omega}(v- {(\Psi\circ \sigma_{\dr u})^*(\partial_t)}, \rule{0.25cm}{0.05mm} ) &= \pa^2_{it}u\pa_{y^i}+a^i(d_\rel u)\pa^2_{ij}u\pa_{y^j}-a^i(d_\rel u)\pa_{x^i} - b^i(d_\rel u)\pa_{y^i}\\
	&= d_\rel(\pa_t u + f(d_\rel u)).
	\end{split}
	\end{equation}
	Then, we can compute the zeroth order and linear terms of $f$, to obtain the claimed decomposition. To prove a-smoothness, we work on the embedding $\tpsil$, which is a-smooth. Note that the factor $\rho^{-2}$ appears in the expression since we need to do this pull-back.
\end{proof}

\subsubsection{The linearised equation }

Now, we note the linearised form of the LMCF equation. By the above computations, and the $1$-form being exact by construction, we can write equality at the level of functions (note that we have multiplied the equation by $\rho^2$ in order to make the linear part nonsingular),
 \begin{equation}\label{eq_potentialequation}
\rho^2E + \rho^2\pa_vu - \rho^2\Delta_{L_t} + \rho^2(H-v^\perp)(\tu) +\rho^2Q(\rho^{-2}d_\mathrm{rel}u,\rho^{-2}{^b\nabla_\mathrm{rel}^2u}) = 0,
\end{equation}
for $E$ being the zeroth order error term, and $Q$ being the nonlinear term. Therefore, we have that: 
\begin{defn}
	The linearised LMCF operator is given by \begin{equation}\label{eq_mapping}
	u \mapsto \rho^2\pa_vu - \rho^2\Delta_{L_t}u+ \rho^2(H-v^\perp)(\tu).
	\end{equation}	
\end{defn}

Henceforth, we will often write the relative elliptic part of (\ref{eq_mapping}) as a uniformly elliptic operator $\mathcal{L}$.

\subsection{Interior and a-priori estimates}
In this section, we will construct a `good cover' of $\bL$ by taking balls inside the model neighbourhoods as constructed in \S \ref{subsubsec_localmodels}, such that our operator behaves like a standard parabolic operator and the metric in these balls is close to the Euclidean metric. Having such a cover will allow us to prove global a-priori estimates on $\bL$ by adding up local interior estimates.

\subsubsection{A good cover of $\bL$}
\begin{defn}\label{def_goodcover}
    Let $R_1,R_2,R_3,R_4,\varepsilon_1,\ep_2,\ep_3,\ep$ be fixed positive constants with $R_1<R_2, R_3<R_4, \ep_1<\ep_2$, and let $N$ be a fixed positive integer. Associated to this, we define two kinds of model pairs of open sets $(U,V)$ with $V\subset U$, as:
    \begin{itemize}
    	\item $U=B_{R_2}\times (-\ep_2,0)$, $V=B_{R_1}\times (-\ep_1,0)$.
    	\item $U=B_{R_4}\times[0,\ep_3)$, $V=B_{R_3}\times[0,\ep_3)$.
    \end{itemize}
    
    A countable collection of pairs of model open sets $\{(U_i,V_i)\}_{i\in I}$ with $V_i\subset U_i$ as above, along with a-diffeomorphisms to $\bL$ of the form $\Psi_i: U_i\rt \Psi_i(U_i)$, is called a \textit{good cover}, if these charts satisfy:
\begin{enumerate}
	\item The images $\Psi_i(U_i)$ are open in $\bL$ and $\bigcup_{i\in I} \Psi_i(V_i)= \bL\backslash \textrm{ff}$. 
	\item For each point $x\in \bL$, there are at most $N$ open sets $\Psi_i(U_i)$ covering $x$.
 \item The image of $B_R\times\{s\}$ under $\Psi_i$ is embedded in a time-slice in $\bL$, and $\pa_t\Psi_i^*(t)>0$ on each $U_i$ (i.e.\ the time-slices and the direction of time are preserved).
    \item The charts satisfy $\|\Psi_i^*g_\bL-g_0\|_{C^0} < \varepsilon$, and $\|\nabla^k(\Psi_i^* g_\bL-g_0)\|_{C^0} < C_k$ for $k \geq 1$, where the connection and norm is with respect to the metric $g_0$ on $U_i$. 
 \end{enumerate}
\end{defn}

Now, we prove the existence of a good cover on $\bL$:
\begin{lem}
	Let $\ep >0$ be a given constant. There exists parameters $R_1,R_2,R_3,R_4,\ep_1,\ep_2,\ep_3$ and a corresponding good cover of $\bL$ as described earlier, with the pullback metrics being $\ep$ close in the $C^0$-norm to the domain metrics.
\end{lem}
\begin{proof}
	
	Due to the presence of finitely many model neighbourhoods as in Lemma \ref{lem_modelneighbourhood}, it is sufficient to find a good cover in each of the model neighbourhoods and then compose them with the model neighbourhood charts. Clearly, the model neighbourhoods away from the front face and corner themselves satisfy the conditions above. Therefore, we need to find a good cover for the model neighbourhoods at the front face and at the corner. 
	
	First, let us consider a model neighbourhood at the front face $M = U\times\llbracket0,\varepsilon)$, where $U\subset \bR^m$ is an open ball and $U\times \{0\}$ is an a-boundary, and we have an m-metric given by $g_0 + t^{-2}dt^2$. In this case, we can consider charts $\Psi_{(x_0,t_0)}: B_R\times (-\delta,\delta)\rt M$, defined as $\Psi_{(x_0,t_0)}(x,t):=(x+x_0, t_0(1+t))$.
	
	Now, let us consider a model neighbourhood near the corner $M = V\times \llbracket0,\varepsilon)\times [0,\delta)$, parametrized by $(\sigma,r,s)$, where $V\subset \bR^{m-1}$ is an open ball, $r=0$ is an a-boundary, $s=0$ is an ordinary boundary, we have a time variable $t=r^2s$, and we have an m-metric $g_0+ds^2+r^{-2}dr^2$. Here, we consider the time-like vector field $\tilde{v}$ that is orthogonal to the time-slices -- note that this is $\ep$-close to the pullback of the time-like vector field from $\bL$. Note that in our coordinates, we have $\tilde{v}=\rho\pa_\rho-2s\pa_s$. Now, we consider a chart $\xi_{x_0}:B_R\rt V\times \llbracket 0,\ep) \times \{0\}\subset M$ that is in normal coordinates with respect to the m-metric and $\xi_{x_0}=(x_0,0)$, and then flow this chart along the time-like vector field $\tilde{v}$ for time $t\in [0,\delta \rho_0^2)$, where $\delta>0$ is a fixed constant. This gives us a chart, $\Xi_{(x_0,0)}: B_R\times [0,\delta)\rt M$, defined as $\Xi_{(x_0,t_0)}(x,t):= f_{\rho_0^2t}\circ \psi_{x_0}(x)$, where $f_t$ is the flow along $\tilde{v}$ at time $t$. 
	
	Now, we observe that both these charts above are invariant under the action of $\bR^+$, given by $t\mapsto ct$ in the first case and $r \mapsto cr$ in the second case, for a positive $c>0$ -- this is because both the m-metric and the time-like vector field $\tilde{v}$ are invariant under this transformation. Therefore, these charts have uniformly bounded derivatives of the metric. Finally, these charts are an isometry at the origin, so by choosing the domain sufficiently small, we can bound the $C^0$-norms of the pullback metric. 
	
	Now, we can cover the region $t\in [\ep/2,\ep]$ in the first case and $r\in[\ep/2,\ep]$ in the second case by finitely many charts as above. Then, translating these charts by the action of $c=2^{-k}$ for $k\geq 1$ as above, we get charts for the whole region $M$. Due to construction, each point in $M$ is covered by a uniformly bounded number of these charts. This constructs a good cover for $M$. 
\end{proof}

\subsubsection{Estimates for a good cover}

First, we note the following about comparison of norms of tensors:

\begin{lem}
	Let $g^1_{ij},g^2_{ij}$ be two norms such that $|g^1-g^2|_{g^1} < \delta<1$. Then, we have that:
	
	\begin{enumerate}
		\item For an $n$-tensor $\alpha$ there exists a positive constant $C_\delta$, such that  \[
		C_\delta^{-1}|\alpha|_{g^2}\leq |\alpha|_{g^1}\leq C_\delta|\alpha|_{g^2}.
		\]
		\item For a function $f$, and the volume measures $\mu_i = \sqrt{\det(g^i)}$, there exists a constant $C_\delta$, such that \[
		C_\delta^{-1}|f|_{L^1(\mu_1)} \leq |f|_{L^1(\mu_2)} \leq C_\delta|f|_{L^1(\mu_1)}.
		\]
	\end{enumerate} 
\end{lem}

The utility of a good cover of $\bL$ is explained in the following lemma: 
\begin{lem}\label{lem_goodcoverestimates}
    Let $\{(U_i,V_i),\Psi_i\}_{i\in I}$ be a good cover of $\bL$, satisfying the $C^0$-estimate on the pullback metric with a positive constant $\ep >0$. Then, there exist positive constants $\delta, C_{i,j}>0$ for $i,j\geq0$ (where $\delta$ can be made as small as we like by taking $\ep$ small), such that for $k,r\geq0$ and any function $u:\bL\rt\bR$, we have the pointwise estimate:
    \begin{equation}
    	\begin{split}  (1+\delta)^{-1}|\pa_t^r\nabla^{k}_x u| &\leq | {^{\m}\nabla^{r}}({^b\nabla^k_\rel} u)| + \sum_{j=0}^{k}\sum_{i=1}^{r}C_{i,j}|{^{\m}\nabla^{r-i}}({^b\nabla_\rel^{j+i}}u)
    	| \\
    	&\leq (1+\delta) |\pa_t^r\nabla^{k+r}_x u| + \sum_{j=0}^{k}\sum_{i=1}^{r}2C_{i,j}|{^{\m}\nabla^{r-i}}({^b\nabla_\rel^{j+i}}u)
    	|.\end{split}
    	\end{equation}
    Here, the norm on the covariant derivatives $\nabla_x$ and $\pa_t$ is the standard Euclidean norm, and the norm on $^b\nabla_\rel$ and $^{\m}\nabla$ is the ${\m}$-norm.
\end{lem}
\begin{proof}
	This follows from computations in the local models for $\bL$. Consider local coordinates $(x^i,t)$ on $\bL$, then we can write, \[
	^b\nabla_\rel u = \pa_{x_i}udx_i,
	\]
	as a section of the relative cotangent bundle. Now, note that the relative cotangent bundle embeds into the m-cotangent bundle as \[
	dx_i \mapsto dx_i - g_{it}dt,
	\]
	for $g_{it}={g}(\pa_{x_i},\pa_t)$ where $g$ is the m-metric. Then, we have\begin{equation*}
		\begin{split}
	^{\m}\nabla(^b\nabla_\rel u) = {^{\m}\nabla}(\pa_{x_i}u(dx_i-g_{it}dt)) &= \pa^2_{x_ix_j}udx_j\otimes(dx_i-g_{it}dt)+ \pa^2_{it}udt\otimes(dx_i-g_{it}dt)\\
	&+\textrm{terms involving lower order derivatives}.
	\end{split}
	\end{equation*}
	Computing similarly for higher derivatives, the claim follows.
\end{proof}

\begin{coro}\label{coro_compnorms}
	We obtain a local comparison of norms, \[
		C_{k,r}^{-1}(\|u|_{\Psi_i(V_i)}\|_{L^p_{k+2r,r}(\bL)} )\leq \|u\circ\Psi_i\|_{L^p_{k+2r,r}(U_i)} \leq C_{k,r}(\|u|_{\Psi_i(U_i)}\|_{L^p_{k+2r,r}(\bL)}).
	\]
\end{coro}

\subsubsection{A global a-priori estimate}

Now, we have that:
\begin{lem}\label{lem_chartsoperator}
	Consider a chart of a good cover $\Psi_i: U_i \rt \bL$. Let us denote the pullback function $u^*:=u\circ\Psi_i$. Then, we have that \[
	\rho^2\pa_v u = c_1\pa_tu + c^j_2\pa_{x^j}u,
	\]
	for $c_1,c^i_2$ being uniformly bounded, and $c_1$ being uniformly positive away from $0$. Moreover, all the derivatives of $c_1,c^j_2$ are also uniformly bounded. Further, we have that \[
	^b\nabla_{\rel,j}u = c_3^i\pa_{x_j}u,
	\]
	for $c_3^j$ being uniformly bounded away from $0$. As a corollary, we have \[
	^b\nabla^2_{\rel,jk}u = c_3^jc_3^k\pa^2_{x_jx_k}u + c_3^j\pa_{x_j}c^k_3\pa_{x_k}u.
	\]
	Therefore, we get that for $a^{jk}{^b\nabla_{\rel,jk}}$ being a uniformly positive definite relative elliptic operator, its pullback is also a uniformly positive definite relative elliptic operator over the good cover, i.e.\ in each $U_i$.
\end{lem}

Using the above, we get the following a-priori estimate for the linearised operator (\ref{eq_mapping}):
\begin{lem}\label{lem_globalestimate}
	Let $u\in L^p_{k+2,1}(\bL)$, and $v\in L^p_{k+2,1}(\bL)$ such that $u-v\in \mathring{L}^p_{k+2,1}(\bL)$ (i.e.\ the Sobolev space which is the completion of functions that are zero on the bottom face of $\bL$).  Then, we have that 
	\begin{equation}\label{eq_basicestimate}
		\|u\|_{L^p_{k+2,1}}
		\leq C_{k}\left(\|(\rho^2\pa_v-\cL)u\|_{L^p_{k,0}}+\|u\|_{L^p}+\|v\|_{L^p_{k+2,1}}\right).
	\end{equation}
 As a result, for $u$ satisfying $u|_{\textrm{bf}}\equiv 0$ (where we consider the trace of $u$ on the bottom face), we get \begin{equation}\label{eq_basicestimate2}
   	\|u\|_{L^p_{k+2,1}}
   	\leq C_{k}\left(\|(\rho^2\pa_v-\cL)u\|_{L^p_{k,0}}+\|u\|_{L^p}\right).
   \end{equation} 
\end{lem}
\begin{proof}
    Consider a good cover of $\bL$ as before, given by $\{(U_i,V_i),\Psi_i\}_{i\in I}$. By Lemma \ref{lem_chartsoperator} the pullback of the operator $\mathcal{L}$ under $\Psi_i$ will be an elliptic operator with uniformly bounded norms of coefficients and constants of ellipticity over the good cover. By the estimates (\ref{eq_ffestimate}) and (\ref{eq_bfestimate}), we get local interior estimates for the standard parabolic operator on each open set $U_i$. We transfer these estimates to the pullback of $\rho^2\pa_v-\mathcal{L}$ by using Lemma \ref{lem_goodcoverestimates} for the good cover while taking $\ep$ sufficiently small, and transfer these estimates from $U_i$ to local estimates on $\bL$. Now, we add them up as follows: raising the above estimates to the $p^{\mathrm{th}}$ power, we have an $L^p$-estimate of the form \[
\|u|_{\Psi_i(V_i)}\|^p_{L^p_{k+2,1}} \leq C\left(\|(\rho^2\partial_v - {\mathcal{L}})u|_{\Psi_i(U^i)}\|^p_{L^p_{k,0}} + \|u|_{\Psi_i(U_i)}\|^p_{L^p_{0,0}} + \|v|_{\Psi_i(U_i)}\|^p_{L^p_{k+2,1}}\right) .
\]

Summing up all these estimates over the cover $\{\Psi_i(U_i)\}$, and using the fact that each point of $\bL$ is covered by at most $N$ open sets of the good cover, we would get that the RHS of the above is bounded above by \[
CN(\|(\rho^2\partial_v - {\mathcal{L}})u\|^p_{L^p_{k,0}}+ \|u\|^p_{L^p}+\|v\|^p_{L^p_{k+2,1}}),
\]
while the LHS is bounded below by $\|u\|^p_{L^p_{k+2,1}}$. Therefore, taking $p^{\mathrm{th}}$ roots, we get a global estimate of the form (\ref{eq_basicestimate}).

\end{proof}

\begin{rem}\label{rem_modelvsgood}
	An equivalent, but conceptually simpler, way to get the above result is to define parabolic operators on the model neighbourhoods themselves, and prove the parabolic interior estimates by considering good covers of the model neighbourhoods. Then, the above estimate follows by adding up finitely many estimates on $\bL$ arising from the model neighbourhoods.
\end{rem}

In the case $p=2$, we have the following lemma regarding the trace function on $\bL$, which follows easily from Lemma \ref{lem_extensionoftrace}:
\begin{lem}
	There exists a constant $C$, such that for $u\in L^2_{k+2,1}(\bL)$, there exists a function $v\in L^2_{k+2,1}(\bL)$ such that  $u|_{t=0}=v|_{t=0}$ in $L^2_{k+1}(\widetilde{L_0})$, and \[
	\|v\|_{L^2_{k+2,1}(\bL)} \leq C\|u|_{t=0}\|_{L^2_{k+1}(\widetilde{L_0})}.
	\]
\end{lem}
As a result, we can state a slightly improved a-priori estimate for $p=2$ which does not require an implicit function:
\begin{prop}
	Let $u\in L^2_{k+2,1}(\bL)$, and $\cL$ be a uniformly elliptic second-order partial differential operator. We have that \begin{equation}\label{eq_basicestimate3}
		\|u\|_{L^2_{k+2,1}}
		\leq C_{k}\left(\|(\rho^2\pa_v-\cL)u\|_{L^2_{k,0}}+\|u\|_{L^2}+\|u|_{\mathrm{bf}}\|_{L^2_{k+1}}\right),
	\end{equation}
	for a constant $C_{k}>0$.
\end{prop}

Then, we can take higher derivatives to show that:
\begin{prop}
	For integers $k,r\geq1$, let $u\in L^2_{k+2r,r}(\bL)$, and $\cL$ be a uniformly elliptic second-order partial differential operator. We have that \begin{equation}\label{eq_basicestimate4}
		\|u\|_{L^2_{k+2r,r}}
		\leq C_{k,r}\left(\|(\rho^2\pa_v-\cL)u\|_{L^2_{k+2r-2,r-1}}+\|u\|_{L^2}+\|u|_{\mathrm{bf}}\|_{L^2_{k+2r-1}}\right),
	\end{equation}
	for a constant $C_{k,r}$.
\end{prop}

\subsubsection{Sobolev embedding}\label{subsubsec_Sobolevembedding}

As a corollary to the above construction of model neighbourhoods and a good cover, we can get a global version of Theorem \ref{thm_sobolevembeddingparabolic}:
\begin{theorem}
	Consider the weighted Sobolev spaces defined in \S \ref{subsubsec_weightedordinaryspaces}. Given parameters $k,r,l,s,p$ and taking $q=\infty$, let $\xi$ as defined in \eqref{eq_parabolicembeddingparameter} satisfy $\xi<1$. Then, we have a continuous embedding $L^p_{k+2r,r,\lambda}(\bL) \hookrightarrow C^{l+2s+\delta,s+\delta/2,\lambda}(\bL)$.
\end{theorem}
\begin{proof}
	Note that it suffices to prove the embedding with weight $\lambda=0$, since we can obtain the weighted estimates by applying the unweighted estimates to $\rho^{-\lambda}u$. Now, due to the hypotheses on the parameters, we have an embedding \begin{equation}\label{eq_embedding}
	L^p_{k+2r,r,\lambda}(\bR^{n+1}_0) \hookrightarrow C^{l+2s+\delta,s+\delta/2,\lambda}(\bR^{n+1}_0),
	\end{equation}
	 by Theorem \ref{thm_sobolevembeddingparabolic}. We know that there exists a good cover of $\bL$ with model pairs $\{(V_i\subset U_i)\}_{i\in I}$ and maps $\Psi_i : U_i \rt \bL$ satisfying the properties outlined in Definition \ref{def_goodcover}. Now, we can slightly modify the argument and obtain a good cover with model pairs $\{(V_i\subset U_i)\}_{i\in I}$ which are of the form \begin{itemize}
	 	\item $U=B_{R_2}\times (-\ep_2,0)$, $V=B_{R_1}\times (-\ep_1,0)$.
	 	\item $U=B_{R_4}\times[0,\ep_3)$, $V=B_{R_3}\times[0,\ep_4)$,
	 \end{itemize}
	 where $\ep_4<\ep_3$ is a new parameter. Then, in each case we can write down a cutoff function $\psi$ such that in each of the above model pairs, $\psi|_{V}\equiv 1$, $\psi|_{U\backslash V'}\equiv 0$ for a fixed open set $V' \supset  \overline{V}$, and $|\nabla^k\psi|<C_k$ for $k\geq 0$. Using these, we can apply the embedding \eqref{eq_embedding} for the functions $\{\psi\cdot\Psi_i^*u\}_{i\in I}$ and use Lemma \ref{lem_goodcoverestimates} and Corollary \ref{coro_compnorms} to obtain the claim. 
\end{proof}

\subsubsection{Estimates with shortened time interval}
In this subsection, we will formulate estimates such that we can shorten the time interval over which the solution exists without changing the constants in the estimates. These will be useful later in proving uniqueness of solutions to the nonlinear problem. Henceforth, we will consider a positive $\eta \in (0,\ep)$, and take all functions and norms over the shortened time interval $[0,\eta)$, i.e.\ we prove estimates for operators on the space $\bL|_{\pi^{-1}([0,\eta))}$. 

\begin{lem}\label{lem_apriorishort}
	Let $f = (\rho^2\pa_v-\cL)u$. Given any $\eta \in (0,\ep)$, 
	we have the following a-priori estimate, 
	\begin{equation}\label{eq_aprioridecreasing}
	\|u\|_{L^2_{k+2r,r,\lambda}}\leq C\left(\|f\|_{L^2_{k+2r-2,r-1,\lambda}}+\sum_{i=1}^{r} \|\pa^i_sf|_{s=0}\|_{L^2_{k+2(r-i)-1,\lambda}}+\|u\|_{L^2_\lambda}\right),
	\end{equation}  and the constant is independent of $\eta$. 
\end{lem}
\begin{proof}
	In order to obtain this estimate, we need to choose a good cover for the restriction $\bL|_{[0,\eta)}$, and then obtain estimates for the model neighbourhoods that have constant independent of $\eta$. First, we note that when shortening the time interval, we can choose a good cover as in Definition \ref{def_goodcover}, such that the only parameter that needs to be changed with $\eta$ is $\ep_3$ -- we need to choose $\ep_3 = O(\eta)$. The rest of the parameters of the good cover do not need to be changed, since we can cover the region $\{s\geq \frac{1}{2}\eta\}\cap \bL$ with uniformly sized neighbourhoods. Therefore, we only need to obtain estimates on model neighbourhoods of the form $U=B_{R_4}\times [0,\ep_3), V = B_{R_3}\times[0,\ep_3)$, with constant independent of $\ep_3 \downarrow 0$. We can obtain such estimates by applying Theorem \ref{thm_bfestimate} and Lemma \ref{lem_extensionoftrace}, noting that in both of these estimates the constant is non-increasing on decreasing $\ep_3$. We add up these estimates, along with similar estimates for higher derivatives, to get the estimate \eqref{eq_aprioridecreasing} on the local models with constant independent of $\eta$, then add over $\bL$.
\end{proof}

\begin{lem}\label{prop_sobolevtime}
	Let us assume we are given $\eta \in (0,\ep)$. Let $u\in C^\infty_\loc$ be a function such that $\pa_s^ku|_{s=0} \equiv 0$ for $k\geq 0$, i.e.\ $u$ vanishes to all orders at the bottom face. Then,  we have the Sobolev embedding \[
	\|u\|_{C^{l+2s+\alpha,s+\alpha/2}} \leq  C (\|u\|_{L^{2}_{k+2r,r}}),
	\]
	and the constant is independent of $\eta$.
\end{lem}
\begin{proof}
	To prove the Sobolev embedding, one needs to multiply by some cutoff functions to make the function supported away from the boundary. We need to ensure that these cutoff functions have norms of their derivatives bounded independent of shortening the time interval. The key point is that due to $u$ vanishing to all orders at the bottom face, we only need to use cutoff functions in the spatial direction -- since the function $\Psi_i^*u$ on $B_R\times [0,\eta)$ can be extended by zero to a function on $B_R\times (-\infty, \eta)$. Therefore, we obtain a constant independent of $\eta$.
\end{proof}

\section{The linearised problem near the front face}\label{sec_parabolicfredholmexpanders}
Recall the parabolic blow-up of $\bC^m\times[0,\infty)$ we constructed earlier as $\Bl$. We will work interchangeably in two sets of coordinates -- coordinates away from the corner given by $(\tau,w^i)$ for $\tau\in \llbracket 0,\infty), w^i\in \bC^m$, and coordinates in a neighbourhood of the front face given by $(\rho,\sigma^i)$ for $\rho\in\llbracket 0,\infty), \sigma^i\in \bS^{2m}_+$. Recall that these coordinates are related to the $(x^i,t)$-coordinates in the space $\bC^m\times[0,\infty)$ by the blow-down mappings 
\begin{equation}\label{eq_blowdownmappingssecfive}
(\tau,w^i) \mapsto \left(\tau w^i, \frac{1}{2}\tau^2\right),\; (\rho,\sigma^i) \mapsto \left(\rho\sigma^i, \rho^2\sigma_{2m+1}^2 \right).
\end{equation}
Also, note that $\rho$ is the boundary-defining function for the front face, given by $\rho=\sqrt{|x|^2+t}$ (thus $\rho$ is unbounded on $\Bl$, in contrast to $\bX$).  
\begin{defn}
	We define the manifold with corners and a-corners $\bE$ as the compactification of the family $\{(\sqrt{2t}E,t) : t\in (0,\infty)\} \subset \bC^m\times[0,\infty)$ in the manifold with corners and a-corners $\Bl$, in a similar way as we defined $\bL$. Thus, $\bE$ has one ordinary boundary, which we call the bottom face, and one a-boundary, which we call the front face.
\end{defn}
Similarly as $\bL$, there is a time-like vector field on $\bE$, as the unique vector field $v$ satisfying $\pi_*(v)=\pa_t$, and $v$ being orthogonal to the time-slices. Then. we similarly define the m-metric, m-volume form, m-Levi-Civita connection,relative tangent bundles, relative Levi-Civita connection, and weighted Sobolev spaces and H\"{o}lder spaces on $\bE$ as in \S \ref{subsec_sobholspaces}. We consider the corresponding linearised LMCF operator on $\bE$ given by \begin{equation}\label{eq_linearisedoperatorE}
\rho^2\pa_v-\cL,
\end{equation}
 for \[
\cL = \rho^2\Delta_{\sqrt{2t}E},
\]
noting that the extra first order linear term vanishes for this family as $H=v^\top$ for this family. Then, we can consider for integers $k,r \geq 1$, and a weight $\lambda\in \bR$, the linear mapping on $\bE$ between $\rho$-weighted Sobolev spaces as defined in \S \ref{subsubsec_weightedordinaryspaces}, 
\begin{equation}\label{eq_sobolevmappingE}
	T^2_{k+2r,r,\lambda}:= ((\rho^2\pa_v- \cL),|_{\botf}) :L^2_{k+2r,r,\lambda}(\bE) \rt L^2_{k+2r-2,r-1,\lambda}(\bE)\oplus L^2_{k+2r-1,\lambda}(\bE|_{\botf}).
\end{equation}

By the same method as earlier, we have an a-priori estimate on $\bE$ for the above mapping, \begin{equation}\label{eq_aprioriestimateE}
\|u\|_{L^2_{k+2r,r,\lambda}}\leq C_{k,r,\lambda}\left(\|(\rho^2\pa_v-\cL)u\|_{L^2_{k+2r-2,r-1,\lambda}}+\|u|_{\botf}\|_{L^2_{k+2r-1,\lambda}}+\|u\|_{L^2_\lambda}\right).
\end{equation}

The main result we will prove is the following:
\begin{theorem}\label{thm_mappingE}
\begin{enumerate}
	\item The mapping $T^2_{k+2r,r,\lambda}$ is semi-Fredholm and injective for all $\lambda \in \bR\backslash \cD_E$, where $\cD_E \subset \bR$ is a certain discrete set of rates (Corollary \ref{coro_semiFredholm}). 
	
	\item The set $\cD_E = \Re(\cC_E)$ satisfies $\cD_E\subset (-\infty, \min(2-m,0)]$.
\end{enumerate}
\end{theorem}
Here, the sets $\cD_E,\cC_E$ are defined in \S \ref{subsec_Fredholmtheory}. As a corollary to the above, we obtain:
\begin{coro}
The mapping $T^p_{k+2r,r,\lambda}$ is semi-Fredholm and injective for $\lambda > \min(2-m,0)$.
\end{coro}

In their paper \cite{LN13}, Lotay--Neves study the deformation equation for Lagrangian expanders, and prove isomorphism theorems about the operator on certain natural Sobolev spaces. While they studied an (degenerate) elliptic operator, our work in this section is in a sense an extension and generalization of their results to the parabolic case. We describe their relevant results in \S \ref{subsec_LN}. In order to prove Theorem \ref{thm_mappingE}, it will be useful to first work with $k=0,r=1$, and then improve it to $k,r\geq 1$ by the following result:

\begin{lem}
	Suppose that the mapping $T^2_{2,1,\lambda}$ is semi-Fredholm and injective. Then, the mapping $T^2_{k+2r,r,\lambda}$ is also semi-Fredholm and injective, for $k,r\geq1$. 
\end{lem}
\begin{proof}
	First, we have the global a-priori estimate on $\bE$, as \[
	\|u\|_{L^2_{k+2r,r}}
	\leq C_{k,r}\left(\|(\rho^2\pa_v-\cL)u\|_{L^2_{k+2r-2,r-1}}+\|u\|_{L^2}+\|u|_{\mathrm{bf}}\|_{L^2_{k+2r-1}}\right).
	\]
	Then, knowing that $T^2_{2,1,\lambda}$ is semi-Fredholm and injective gives us the estimate \[
	\|u\|_{L^2}
	\leq C'\left(\|(\rho^2\pa_v-\cL)u\|_{L^2_{0,0}}+\|u|_{\mathrm{bf}}\|_{L^2_{2}}\right),
	\]
	combining the above two estimates, we are done. 
\end{proof}

By this lemma, it is enough to work with $k=0,r=1$, which we assume henceforth.

 \subsection{Fourier analysis}
 
To prove our result, we will invoke Fourier analysis in order to reduce the problem to analysing the isomorphism properties of a parametric operator on the front face. Note that we work with complex-valued $L^2$-spaces throughout.                                                We recall the definition we will be using here.
 
 \begin{defn}
 	Let $M$ be a smooth, compact manifold, and let $f : M\times \bR \rt \bR$ be a function in $L^2(M\times \bR)$. The \textit{Fourier transform} of $f$ is the function $\hat{f} \in L^2(M\times \bR)$, defined as:
 	\[
 	\hat{f}(p,\xi):= \frac{1}{2\pi}\int_{\bR}e^{-i\xi\cdot x}f(p,x)dx.
 	\] 
 \end{defn}
 
A basic property of the Fourier transform is:
 \[
 	\cF(\pa_x f)(p,\xi) = i\xi\hat{f}(p,\xi).
 	\]
One of the key results about the Fourier transform is that:
\begin{theorem}[Plancherel's theorem]
	The Fourier transform is an isometry on $L^2(M\times \bR)$ up to a constant.
\end{theorem}

 \subsection{Fredholm theory on $\bE$}\label{subsec_Fredholmtheory}

Note that in $(\tau,w^i)$-coordinates, $\bE$ restricts to the expander $E \subset \bC^m$ at $\tau=0$. In $(\rho,\sigma^i)$-coordinates, $\bE$ restricts to a compact manifold with boundary at the front face $\bS^{2m}_+$, denoted $\oE$. We will work with the parametrisation of $\bE$ given by $\oE\times\llbracket0,\infty)\subset \bS^{2m}_+ \times \llbracket0,\infty)$ in $(\sigma^i,\rho)$-coordinates. Since the front face $\bS^{2m}_+$ in $(\rho,\sigma^i)$-coordinates is a compactification of $\bC^m$, $\oE$ is the compactification of the expander $E\subset \bC^m$ via the same identification. There is also a natural boundary-defining function on $\oE$, given by restricting the boundary-defining function $s := {\sigma^{2m+1}}^2 =  (1-\sum_{i=1}^{2m} {\sigma^i}^2)$ from $\bS^{2m}_+$. Now, we can write a diffeomorphism, \[
\Phi : \oE^\circ \rt E,
\]  
given by \begin{equation}\label{eq_phi}
\Phi: (\sigma^i) \mapsto \left( \frac{\sigma_i}{\sqrt{2}\sigma^{2m+1}}\right),
\end{equation}
for $\sigma^{2m+1} = ({1-\sum_{i=1}^{2m}{\sigma^i}^2})^{1/2}$.
We note that for a differential operator or vector field on $E$, we can pull it back via the diffeomorphism $\Phi$ to obtain a corresponding differential operator or vector field on the interior of $\overline{E}$, which may be possibly extended to the boundary. 

\subsubsection{Restrictions of quantities to the front face}
We will define restrictions of objects defined on $\bE$ to the front face $\oE$, which is a manifold with an ordinary boundary. First, we define the restriction of a metric:

\begin{defn}
	The m-metric on $\bE$ inherited from $\Bl$ defines an ordinary metric $g_{\oE}$ on $\oE$, in the following way: there is a projection map, $^{\m}T\bE\rt T\oE$ because of writing $\bE$ as $\oE\times \llbracket0,\infty)$. The inner product of two vectors in $T_x\oE$ is defined as the inner product of their pull-backs at a point, which is independent of choice of pull-back due to the translation-invariance of the metrics in $\rho$.
\end{defn}
\begin{defn}
	The relative $b$-metric on $\bE$ defines a $b$-metric $^bg_{\oE}$ on $\oE$ with $\pa\oE$ as an a-boundary, in the following way: there is a surjective projection map $^bT_\rel \bE \rt {^bT}\oE$, defined by composing the inclusion $^bT_\rel\bE\hookrightarrow {^bT}\bE$ with the projection $^{b}T\bE \rightarrow T\oE$. Then, the inner product of two vectors is the inner product of their pull-backs at a point. This is well-defined as the inner product is independent of choice of pull-back by translation-invariance of the metrics in $\rho$.
\end{defn}
We have the following relating the $b$-metric on $\oE$ to the natural metric on the expander:
\begin{lem}
	We can write the $b$-metric as \[
	^bg_{\oE}=s\Phi^*(g_E),
	\] i.e.\ a scaling of the pullback of the natural metric on $E$ along the diffeomorphism $\Phi$ as in \eqref{eq_phi}.
\end{lem}
 
Now, we define the Levi-Civita connections in these metrics. 
\begin{defn}
	We denote the ordinary Levi-Civita connection for $g_\oE$ by $\nabla_\oE$, and the $b$-Levi-Civita connection for $^bg_\oE$ by $^b\nabla_{\oE}$. 
\end{defn}
\begin{defn}
	We define the volume form on $\oE$ by $dV_{\oE}$, as the volume form for the metric $g_{\oE}$.
\end{defn}
Now, since we have a choice of boundary defining function $s$ on $\oE$, we can define an isomorphism $T\oE \rightarrow {^bT}\oE$, sending $\pa_s\rightarrow s\pa_s$. This allows us to define an isomorphism $\pi_{\oE}: {^bT}^*\oE \rt {^bT}\oE$, sending $s^{-1}ds\rt ds$. Then, we define:

\begin{defn}
	For a smooth function $u$ on $\oE$, define $\nabla_{\oE}^\alpha {^b\nabla_{\oE}^\beta} u \in \Gamma((T^*\oE)^{\alpha+\beta})$ in the following way: first, we have ${^b\nabla_{\oE}^\beta}u \in \Gamma((^bT^*\oE)^\beta)$. Then, by the above isomorphism, we can consider $\pi_{\oE}({^b\nabla_{\oE}^\beta}u)\in \Gamma((T^*\oE)^\beta)$. Finally, we define \[\nabla_{\oE}^\alpha {^b\nabla_{\oE}^\beta} u := \nabla_{\oE}^\alpha\left(\pi_{\oE}({^b\nabla_{\oE}^\beta}u)\right).
	\]
\end{defn}
Now, the following key lemma explains the above definition:
\begin{lem}
	Consider a function $u(\sigma^i,\rho) \equiv u(\sigma^i)$ in the coordinates $\oE\times \llbracket0,\infty)$ on $\bE$, i.e.\ it is invariant in $\rho$. Then, we have that \[
	|^m\nabla^\beta {^b\nabla^\alpha_\rel}u(\sigma,r)|  = |\nabla^\beta_{\oE}{^b\nabla^\alpha_\oE}u(\sigma)|.
	\]
\end{lem}
Thus, we can define the following norm on $\oE$, as a `restriction' of the $C^{k+2r,r}$-norm on $\bE$:
\begin{defn}
	For a function $u\in C^\infty(\oE)$, we define its pointwise $\chC^{k+2r,r}$-norm as \[
	|u|_{\chC^{k+2r,r}} := \sum_{\beta\leq r}  \sum_{|\alpha|\leq k+2r-2\beta}|\nabla_\oE^\beta {^b}\nabla^\alpha_{\oE} u|
	\]
\end{defn}

Following this, we will define Sobolev spaces on $\oE$, which are in a sense a `restriction' of the Sobolev spaces on $\bE$.
\begin{defn} The Sobolev spaces $\check{L}^2_{k+2r,r}(\oE)$ on $\oE$ are defined as the spaces of functions with weak derivatives having finite norm given by:
	\begin{equation}
		\|u\|^2_{\check{L}^2_{k+2r,r}(\oE)}:=  \sum_{\beta\leq r}  \sum_{|\alpha|\leq k+2r-2\beta}\|\nabla_\oE^\beta {^b}\nabla^\alpha_{\oE} u\|^2_{L^2(\oE)},  
	\end{equation}
	where both the metric on forms and the volume form are induced from the ordinary metric $g_\oE$.
\end{defn}

Our Sobolev spaces on $\oE$ satisfy some natural properties: firstly, smooth functions on $\oE$ are dense in these spaces, and further they are Hilbert spaces. We also define standard Sobolev spaces on $\pa\oE$:
\begin{defn}
	We define the Sobolev spaces $L^2_{k+2r}(\pa\oE)$, where the metric and connection on $\pa\oE$ arise from restricting the ordinary metric and connection from $\oE$.
\end{defn}

\begin{rem}
We can write down our Sobolev norm for a function $u$ supported in a region $\Sigma\times [0,\ep)$ near the boundary of $\oE$, for better clarity on how it behaves near the boundary $s=0$. It is given by,\[
\|u\|^2_{\check{L}^2_{k+2r,r}(\oE)}=\sum_{\beta\leq r}\sum_{|\alpha|\leq k+2r-2\beta} \int_{[0,\ep)}\int_\Sigma |\nabla^{\alpha_{\Sigma}}_{\Sigma}(\pa_s^\beta(s\pa_s)^{\alpha_1}u)|^2dV_\Sigma ds, 
\]
for the multi-index $\alpha = (\alpha_1,\alpha_\Sigma)$.
\end{rem}

\subsubsection{The operator in $(\rho,\sigma^i)$-coordinates}
Now, we note that the operator $\rho^2\pa_v-\cL$ is translation-invariant in $\rho$ with the following decomposition:

\begin{lem}\label{lem_h1nonzero}
	The operator \eqref{eq_linearisedoperatorE} can be written in coordinates $\oE\times \llbracket0,\infty)$ on $\bE$ as 
	\[
	(\rho^2\pa_v-\cL)u = -h_1(\rho\pa_\rho)^2u+\rho\pa_\rho(w+ h_2)u - P_0u,
	\]
	where $P_0$ is a degenerate linear elliptic operator on $\oE$, $w$ is a smooth $b$-vector field on $\oE$ and $h_1,h_2$ are smooth functions on $\oE$. More precisely, we can write these as pull backs \[
	-P_0 = \Phi^*((r\pa_r)^\top-\Delta_E), \; w = \Phi^*(\pa_r),\; h_1 = \Phi^*(|(r\pa_r)^\top|^2_{E}),
	\]
	for $\Phi$ as in \eqref{eq_phi}.	As a result, $h_1$ is a non-zero function on $\oE$ converging to $2m$ at $\pa\oE$.
\end{lem}
\begin{proof} Let us begin by writing down the time-like vector field in the embedding $\{\sqrt{2t}E\}_{t\in(0,\infty)}\subset\bC^m\times(0,\infty)$. At a point $(x,t)\in\sqrt{2t}E\times\{t\}$, it is given by \[
	v = \pa_t + (2t)^{-1}x^\top.
	\]
	Therefore, the operator in $(\tau,w^i)$-coordinates is given by
	\begin{equation}\label{eq_operatorffcoordinatessimple}
	\rho^2\pa_v-\cL = \tw^2(\tau\pa_\tau-w\pa_w^\top-\Delta_E).
	\end{equation}
	Now, we note the change of coordinates from $(\sigma,r)$ to $(w^i,\tau)$ as \[
	(\sigma,\rho)\mapsto (\frac{\sigma^i}{\sqrt{2}\sigma^{2m+1}}, \rho (\sigma^{2m+1})^{-1/2}).
	\]
	Then, we note that the mapping $(\sigma^i)\mapsto (w^i)$ is $\rho$-invariant. Now, we note the effect of weighting the operator \eqref{eq_operatorffcoordinatessimple} by $\rho^\lambda = \tau^\lambda \tw^\lambda$. Then, the claim follows by comparing the coefficients of $\lambda^j$ with the coefficients of $(\rho\pa_\rho)^j$for $j=0,1,2$.
\end{proof}

\subsubsection{Fourier transform of the weighted operator}

 Since the mapping \eqref{eq_sobolevmappingE} is between $\rho$-weighted spaces, it is equivalent to consider the conjugated mapping between un-weighted spaces,
\[
	T^2_{2,1,\lambda} = (\rho^{-\lambda}(\rho^2\pa_v-\cL)\rho^\lambda, |_{\botf}):L^2_{2,1}(\bE)\rt L^2_{0,0}(\bE)\oplus L^2_{2}(\bE|_{\textrm{bf}}),
\]
for the function $\tu = \rho^{-\lambda }u$. Now, we can write the operator $\rho^{-\lambda}(\rho^2\pa_v-\cL)\rho^\lambda$ in $(\rho,\sigma^i)$-coordinates as\[
\rho^{-\lambda}(\rho^2\pa_v-\cL)\rho^{\lambda} = -(\rho\pa_\rho)^2h_1+\rho\pa_\rho\otimes w+ h_2(\rho\pa_\rho)  - 2\lambda \rho\pa_\rho + \lambda w - \lambda^2h_1 +\lambda h_2 -P_0.
\]
 We consider the Fourier transform of the above operator in the variable $\log(\rho)$. This gives us for each $\xi\in \bR$, an operator on $\oE$, which we denote $-P_{\lambda+i\xi}$, and we get that \[
 -P_{\lambda+i\xi}:= -P_0 + (\lambda+i\xi) (w  +h_2) -(\lambda+i\xi)^2h_1.
 \]

Therefore, we need to study the mapping properties of the operator $(-P_z,|_{\pa\oE})$ on $\oE$ for $z = \lambda +i \xi\in \bC$. In particular, we need to determine the rates $\lambda\in\bR$ for which this operator is uniformly invertible for all $\xi \in \bR$, between the natural $L^2$-Sobolev spaces (for complex-valued functions) on $\oE$ obtained from restriction. First, we note that the operator $-P_z$ defines a well-defined continuous map, 
	\begin{equation}\label{eq_mappingsobolevff2}
		(-P_z,|_{\pa\oE}) : \check{L}^2_{2,1}(\oE) \rt \check{L}^2_{0,0}(\oE)\oplus L^2_{1}(\pa\oE).
	\end{equation}

We define the sets $\cC_E,\cD_E$ as follows:

\begin{defn}
	The \textit{set of critical rates} $\cC_E$ is defined as:
	\[
	\cC_E:=\{ z\in \bC : \textrm{for } z= \lambda+i\xi, \textrm{ the operator } (-P_{\lambda+i\xi}, |_{\pa\oE})  \textrm{ has a nonzero kernel}\},
	\]
	and $\cD_E$ is the projection of this set to $\bR$,\[
	\cD_E:= \Re(\cC_E).
	\]
\end{defn}

Now, we state the main result we will prove about the parametric operator:

\begin{theorem}\label{thm_mappingff}
	We have that:
	\begin{enumerate}[label=(\alph*)]
		\item\label{item_zeroindex} The mapping \eqref{eq_mappingsobolevff2} is Fredholm over all parameters $\lambda + i\xi \in \bC$, with zero index (Proposition \ref{prop_zeroindex}).
		 
		\item\label{item_isomorph} The set $\cC_E$ has no elements $z$ with $\Re(z)>\min(2-m,0)$ (Proposition \ref{prop_nokernel}).
		\item\label{item_sparsekernel} The set $\cC_E$ is a discrete subset of $\bC$ (Proposition \ref{prop_discretecritrates}).
		\item\label{item_finitestrip} The set $\cC_E$ is finite in any bounded real strip $\{z : a < \Re(z) < b\}$ (Proposition \ref{prop_sparsekernel}).
	\end{enumerate}
\end{theorem}
This will be proved in the following sections. As a corollary to \ref{item_zeroindex} and \ref{item_isomorph} above, we have that 
\begin{coro}
	The mapping \eqref{eq_mappingsobolevff2} is an isomorphism for $\Re(z)>\min(2-m,0)$.
\end{coro}
\begin{proof}
	By \ref{item_zeroindex}, the mapping is Fredholm with zero index. Therefore, it has a finite-dimensional kernel and cokernel, of the same dimension. By \ref{item_isomorph}, for $\Re(z)>\min(2-m,0)$, we do not have any critical rates, i.e.\ the dimension of the kernel is zero for $\Re(z)>\min(2-m,0)$. So, the dimension of the cokernel is zero as well, which implies the mapping is an isomorphism.
\end{proof}

\subsection{Elliptic estimates for the operator on $\oE$}
We first note a-priori estimates for the mapping \eqref{eq_mappingsobolevff2}, obtained by restricting the estimate \eqref{eq_aprioriestimateE} to $\oE$:
\begin{lem}
	For each $z\in \bC$ and $u\in \chL^2_{k+2r,r}(\oE)$, we have the estimate, 
	\begin{equation}\label{eq_interiorestimateff}
		\|u\|_{\chL^2_{k+2r,r}} \leq C_{z}(\|-P_{z}(u)\|_{\chL^2_{k+2r-2,r-1}} + \|u\|_{\chL^2}+\|u|_{\pa \oE}\|_{L^2_{k+2r-1}})
	\end{equation}
\end{lem}
\begin{proof}
	Given a $u\in \chL^2_{k+2r,r}(\oE)$, define the function $v(\sigma,r):= \phi(r)r^zu(\sigma)$ on $\bE$, for $\phi(r)$ being a smooth, compactly supported bump function on $\bR_+$ supported on $[1/2,1]$. Then, we write the a-priori estimate on $\bE$ for $v$ to deduce the a-priori estimate on $\oE$.
\end{proof}
We recall the following functional analytic result: 
\begin{lem}\label{lem_funcana}
	Let $X,Y,Z$ be Banach spaces and $T: X\rightarrow Y$ be a bounded, linear map and $K: X \rightarrow Z$ be a bounded, linear, compact map. Suppose we have the estimate \[
	\|u\|_X \leq C(\|Tu\|_Y + \|Ku\|_Z)
	\]
	for a constant $C>0$. Then, we have that $T$ has finite dimensional kernel and closed image.
\end{lem}
\begin{proof}
	\textit{Finte dimensional kernel:} Suppose not, then by Riesz's Lemma, there is an infinite sequence of linearly independent elements in the kernel, say $\{r_i\}_{i\geq 1}$, such that by scaling we have $\|r_i\|=1$ and $\|r_i-r_j\|_X\geq \frac{1}{2}$. Then, we have that for $i\neq j$,\[
	\frac{1}{2}\leq \|r_i - r_j\|_X\leq C(\|K(r_i - r_j)\|_Z ) \implies \|K(r_i - r_j)\|_Z \geq \frac{1}{2C} >0 .
	\]
	However, due to compactness of $K$, we have that there is a convergent subsequence, say $\{r_{i_j}\}_{j\geq 1}$, such that $Kr_{i_j} \rightarrow s \in Z$ as $j\rightarrow \infty$. This implies that $K(r_{i_j})$ is a Cauchy sequence in $Y$, which contradicts the above inequality.\\
	
	\textit{Closed image:} Since we have a finite dimensional kernel, we can consider the quotient mapping\[
	T : X/\textrm{ker}(T) \rightarrow Y,
	\]
	which is injective. Now, we recall that for a bounded, linear, injective map, showing closed image is equivalent to showing an inequality of the form $\|u\|_X\leq C\|Tu\|_Y$ for a constant $C>0$. Suppose this does not hold, then we can find a sequence of linearly independent elements $\{s_i\}\subset X/\textrm{ker}(T)$ such that $\|s_i\|_{X/\textrm{ker}(T)} =1$ , $\|s_i-s_j\|_{X/\textrm{ker}(T)}\geq \frac{1}{2}$ and $\|Ts_i\|_Y\rightarrow 0$ as $i\rightarrow \infty$. Therefore, we have from our estimate that \[
	\frac{1}{2} \leq \|s_i -s_j\|_{X/\textrm{ker}(T)} \leq C(\|T(s_i-s_j)\|_Y+\|K(s_i - s_j)\|_Z) \implies \|K(s_i - s_j)\|_Z \geq \frac{1}{2C}-\delta>0,
	\] 
	for $i\neq j$ and $i,j \gg 0$ so that $\|T(s_i) - T(s_j)\|<\delta$. Now, due to compactness of $K$, we can again find a convergent subsequence in $\{K(s_i)\}$, which implies it must be Cauchy in $Z$, which is a contradiction to the above inequality.
\end{proof}

Now, we note the following analogue of Rellich-Kondrachov:

\begin{lem}\label{lem_rellichkondrachovE}
	The inclusion $\chL^2_{k+2r,r} \hookrightarrow \chL^2_{k+2r-1,r-1}$ is compact.
\end{lem}
\begin{proof}
	For a function $u$ defined on $\oE$, we can consider the function on $\bE$ defined as $\tu = \phi(r)u$ for $\phi(r)$ being a smooth, compactly supported bump function on $\bR_+$ supported on $[1/2,1]$. Then, under this correspondence, it is easy to see that a sequence of $u_i\in \chL^2_{k+2r,r}$ defines a sequence of functions $\tu_i \in L^2_{k+2r,r}(\bL)$ which are uniformly supported away from the front face. Therefore, by the usual Rellich-Kondrachov (which can be invoked since the functions $u_i$ are supported uniformly away from the front face), we can find a convergent subsequence in $L^2_{k+2r-1,r-1}$. Restricting this to the front face proves the claim.
\end{proof}

As a corollary to the above, using the fact that $w$ is a $b$-vector field on $\oE$ we obtain:	\begin{coro}\label{lem_compactfield}
	The mapping $u\mapsto w(u)$ from $\chL^2_{k+2r,r} \rt \chL^2_{k+2r-2,r-1}$ is compact.
\end{coro}
We also obtain that:
\begin{coro}\label{coro_closedimageE}
	The mapping \eqref{eq_mappingsobolevff2} has closed image.
\end{coro}
\begin{proof}
	This follows from the a-priori estimate \eqref{eq_interiorestimateff}, along with the compactness of the inclusion $\chL^2_{2,1}(\oE)\hookrightarrow \chL^2(\oE)$ and invoking Lemma \ref{lem_funcana}.
\end{proof}

Now, we have a regularity result for the kernel of \eqref{eq_mappingsobolevff2}:
\begin{lem}\label{lem_regularityE}
	Suppose that $u\in \ker(P_z,|_{\pa\oE})$. Then, we have that $u\in C^\infty(\oE)$ (note that here, we have the ordinary smooth structure on $\oE$ near the boundary).
\end{lem}
\begin{proof}
	By the interior estimate and $u$ lying in the kernel, we have that \[
	\|u\|_{\chL^2_{k+2r,r}} \leq C_{k,r}(\|u\|_{L^2}), 
	\]
	for $k,r\geq1$ by induction.
	Now, we use the Sobolev embedding, $\chL^2_{k+2r,r}(\oE) \hookrightarrow C^{r-1,\alpha}(\oE)$ for all $r\gg 0$, to obtain the result. Note that the Sobolev embedding follows since the norm $\chL^p_{k+2r,r}$ has $r$ ordinary derivatives on $\oE$.
\end{proof}
As a corollary to this, we obtain:

\begin{coro}\label{coro_zerokernel}
	Suppose that $u\in \ker(P_z, |_{\pa\oE})$. Then, $u=O(s^\infty)$ near the boundary $s=0$ of $\oE$. 
\end{coro}
\begin{proof}
	This follows from writing $u$ in a Taylor series expansion in $s$ and noting that all coefficients must be zero, since the operator $P_z$ contains a non-vanishing $\pa_s$-term.
\end{proof}

\subsection{The work of Lotay--Neves}\label{subsec_LN}
In the coordinates $(\tau,w^i)$ near the front face, the operator can be written as \[
\rho^2\pa_v - \cL = \tw^2(\tau\pa_\tau - (w\pa_w)^\top -\Delta_E),
\]
 where $\tw := (\frac{1}{2}+|w|^2)^{1/2}$, and $w\pa_w$ denotes the vector field $\sum w^i\pa_{w^i}$.
Now, a deformation of the expander corresponds in these coordinates to a solution for the above equation of the form $\tau^2 u(w^i)$. Plugging this in to the equation, we get the operator \[
\tw^2(2- (w\pa_w)^\top -\Delta_E)u.
\]
 Therefore, to study the deformation theory of the expander, we can study the operator on $E$ given by 

\begin{equation}\label{eq_operatorfflin}
	\cK := 2 - (r\pa_r)^\top -\Delta_E.
\end{equation}

\subsubsection{Fredholm theory for the operator \eqref{eq_operatorfflin}}\label{subsubsec_LNanalysis}
Now, we need to construct Sobolev spaces on the expander $E$ such that the map defined by \eqref{eq_operatorfflin} is an isomorphism. Such Sobolev spaces were constructed by Lotay--Neves, and they showed that \eqref{eq_operatorfflin} is an isomorphism.

They define the Sobolev spaces $H^k_*(E)$ and $H^k(E)$, where $H^k$ is the usual Hilbert space of functions having $k$ weak derivatives in $L^2$ and $H^k_*$ is defined by the norm:\[
\|u\|^2_{H^k_*(E)}:= \sum_{i=0}^k\|\nabla^iu\|^2_{L^2(E)}+\sum_{i=1}^{k-1}\|\la \nabla^iu, x^\top\ra\|^2_{L^2(E)}.
\] 
Then, $\cK$ is a well-defined continuous mapping between the Sobolev spaces, 
\begin{equation}\label{eq_ffpdesobolev}
	\cK: H^{k+2}_*(E)\rt H^k(E).
\end{equation}
Then, Lotay--Neves prove the following (they prove this for a Lagrangian expander asymptotic to the pair of planes $\Pi_\phi\sqcup \Pi_0$, but the same argument works for a general asymptotic cone which is a union of special Lagrangians):
\begin{theorem}\label{eq_fredholmff}[c.f.\ Theorem 4.3 of \cite{LN13}]
	The mapping $\cK$ between the Sobolev spaces \eqref{eq_ffpdesobolev} is an isomorphism.
\end{theorem}

Now, we can write the norm $H^k(E)$ on the compactification $\oE$ as:
\[
\|u\|^2_{H^k(E)} = \sum_{j=0}^k \int_{E}|\nabla^j_Eu|^2dV_E = \sum_{j=0}^k \int_{\oE}|^b\nabla^j_{\oE}u|^2s^{-m/2-2}dV_\oE,
\]
and similarly,
\[
\|u\|^2_{H^k_*(E)} = \sum_{j=0}^k \int_\oE |^b\nabla^j_{\oE}u|^2s^{-m/2-2}dV_\oE + \sum_{j=0}^{k-2} \int_{\oE}|\pa_s(^b\nabla^j_{\oE}u)|^2s^{-m/2-2}dV_\oE.
\]
Therefore, the Sobolev spaces used in their analysis are the same as our Sobolev spaces $\chL^2_{k,1}$, up to a weight of $s^{-m/2-2}$. Due to the weight $s^{-m/2-2}$ blowing up at the boundary of $\oE$, we note that:
\begin{lem}
	If a function $u\in H^k_*(E)$, then the pull back of the function to $\oE$ has zero trace on the boundary $\pa\oE$.
\end{lem}

Now, we write the operator $\cK$ on $\oE$. For a function $u$ defined on $E$, we have,
\[
\cK(u) = \rho^{-2}(\rho^2\pa_v-\cL)(\rho^2su) = (-P_2)(su).
\]
Therefore, $\cK(u)=f$ translates to $-P_2(su) = f$ on $\oE$. So, from Theorem \ref{eq_fredholmff} we get:
\begin{lem}\label{lem_Lotay--Neves}
	Let $f$ be a smooth, compactly supported function on $\oE$, supported away from $\pa\oE$. Then, there exists a unique function $u\in L^2_{k,1}(\oE)$ such that $u|_{\pa\oE}=0$ and $(-P_2)u=f$. 
\end{lem}
\begin{proof}
	We can pull back $f$ to $E$ to obtain an element in $H^k(E)$, since it has compact support. Now, we can find an element $u\in H^k_*(E)$, such that $\cK(u)=f$. Now, since $u$ vanishes quickly near infinity on $E$, we can pull it back to obtain a smooth function on $\oE$. Then, we have that $-P_2(su)=f$, which implies the claim.
\end{proof}

As a result, we can show that:

\begin{lem}\label{lem_p2isomorphism}
	The mapping $(-P_2,|_{\pa\oE})$ defined in \eqref{eq_mappingsobolevff2} is an isomorphism.
\end{lem}
\begin{proof}
	First, we note that the mapping has no kernel -- this follows since if it has an element $u$ in the kernel, then by Lemma \ref{lem_regularityE} and Corollary \ref{coro_zerokernel}, $u$ is a smooth function on $\oE$ vanishing to all orders at the boundary. Therefore, we can consider the function $us^{-1}$ and push it forward to $E$ to obtain an element in the kernel of $\cK$, which is a contradiction. 
	
	Now, we claim that the mapping is surjective. To show this, by Corollary \ref{coro_closedimageE} it is enough to show surjectivity onto pairs of smooth functions $v\in C^\infty(\oE)$, $g\in C^\infty(\pa\oE)$, since such pairs $(v,g)$ are dense in the space $\chL^2(\oE)\oplus L^2_1(\pa\oE)$ (note that $\chL^2(\oE)$ is the usual $L^2$-space on $\oE$ with the ordinary metric). Now, given such a pair, we can find a smooth function $u\in C^\infty (\oE)$ such that $u|_{\pa\oE}=g$ and $-P_2(u)-v$ vanishes to all orders at the boundary of $\oE$, by solving in a Taylor series expansion in $s$ at the boundary of $\oE$. Therefore, we have reduced to showing surjectivity onto pairs of functions $(v,0)$ such that $v$ vanishes to all orders at the boundary of $\oE$. Since such functions can be approximated by compactly supported functions away from $\pa\oE$, we are reduced to the case of Lemma \ref{lem_Lotay--Neves}. Thus, the claim follows. 
\end{proof}

Therefore, we have that 
\begin{prop}\label{prop_zeroindex}
	The mapping $(-P_z,|_{\pa\oE})$ has zero index for all $z\in\bC$.
\end{prop}
\begin{proof}
	Note that since the terms $zw+zh_2-z^2h_1$ in the parametric operator $-P_z$ are compact, it is sufficient to show that $(-P_z,|_{\pa\oE})$ has zero index for some $z\in \bC$. However, we know that $(-P_2,|_{\pa\oE})$ is an isomorphism from Lemma \ref{lem_p2isomorphism}. Thus, the claim follows.
\end{proof}

\subsection{No critical rates for $\Re(z)>\min(2-m,0)$}

In this section, we will prove part \ref{item_isomorph} of Theorem \ref{thm_mappingff}. The method we use is the following: by Corollary \ref{coro_zerokernel}, we have that if the operator $(-P_z,|_{\pa\oE})$ has a nontrivial element in the kernel, it must decay to all orders at the boundary. Therefore, we can consider our operator weighted by a power of the boundary-defining function $s^{(\lambda+i\xi)/2+\mu}$, for $\mu>0$ being a positive rate. If $-P_z$ has a nontrivial kernel, so would this weighted operator. Then, we can apply the maximum principle to the weighted operator to rule out a nontrivial kernel, proving that $(-P_z,|_{\pa\oE})$ has a trivial kernel at the specific rate. In particular, we use the following, which follows from considering the point at which $|u|^2$ attains its maximum.

\begin{lem}\label{lem_maximumprinciple}
	Let $\Omega$ be a bounded, smooth domain with smooth boundary. Consider a weakly elliptic partial differential operator $P$ on $\Omega$, given by  \[
	P(u):= a^{ij}\pa^2_{ij}u+b^i\pa_iu + cu,
	\]
	such that 
	\begin{enumerate} 
		\item the coefficients are smooth and bounded, with $a^{ij},b^i$ being real-valued functions, while $c$ being a complex-valued function satisfying $\Re(c)>0$ on $\Omega$. 
		\item The matrix $[a^{ij}]$ is positive definite at each interior point of the domain.
	\end{enumerate}
	Suppose that a complex-valued $C^{2}$ function $u$ satisfies $\Re(P(u)\cdot \bu)\geq 0$. Then, $|u|$ attains its maximum on the boundary of $\Omega$.
\end{lem}

 \subsubsection{The operator in front face coordinates}
 Recall that the operator $\rho^2\pa_v-\cL$ in front face coordinates is given as in \S \ref{subsec_LN} by
 \begin{equation}\label{eq_operatorffcoordinates}
 	\rho^2\pa_v-\cL = \tw^2(\tau\pa_\tau - (w\pa_w)^\top -\Delta_E),
 \end{equation}
 where $\tw := (\frac{1}{2}+|w|^2)^{1/2}$, and $w\pa_w$ denotes the vector field $\sum w^i\pa_{w^i}$. Now, due to $E$ being an expander, we have (c.f.\ proof of Proposition 3.2 in \cite{LN13}) \[
 \Delta_E r^2 = 2m + 2|r^\perp|^2 , \: \: |\nabla_E r^2|^2=4r^2-4|r^\perp|^2.
 \]

\subsubsection{Weighting by $s$}

First, note that we define the boundary-defining function $s$ for $\pa\oE$ as \[
s:=\sigma_{2m+1}^2.
\]
Then, we have that $t=\rho^2s$. Now, consider the operator \[
P'_z(u) := s^{-(\lambda+i\xi)/2}P_z(s^{(\lambda+i\xi)/2}u).
\]
Due to $s$ being independent of $r$, we can equivalently write this as a restriction, \[
P'_z  = s^{-(\lambda+i\xi)/2}(r^{-\lambda+i\xi}(\rho^2\pa_v-\cL)r^{\lambda+i\xi})(s^{(\lambda+i\xi)/2})|_{\oE}.
\]
Now, this becomes \[
P'_z = \tau^{-(\lambda+i\xi)}(\rho^2\pa_v-\cL)\tau^{(\lambda+i\xi)}|_{\oE}.
\]
Since we know that the operator is given as \eqref{eq_operatorffcoordinates}, we can compute in front face coordinates to get that \[
\tau^{-(\lambda+i\xi)}(\rho^2\pa_v-\cL)\tau^{(\lambda+i\xi)} = \rho^2\pa_v-\cL + (\lambda+i\xi)s^{-1}.
\]
Therefore, since we know that the restriction of the operator $\rho^2\pa_v-\cL$ is $-P_0$, we get \[
P'_z = -P_0 + (\lambda+i\xi)s^{-1}.  
\]
 Therefore, we have:
\begin{prop}\label{prop_nozerokernel}
The parametric operator $(-P_z,|_\oE)$ has trivial kernel for $\Re(z)>0$.	
\end{prop}
\begin{proof}
	Suppose to the contrary that $u$ is an element in the kernel of $(-P_z,|_{\oE})$. By Corollary \ref{coro_zerokernel}, $u$ vanishes to all orders at the boundary of $\oE$. Therefore, $s^{(\lambda+i\xi)/2}u$ is a well-defined, smooth function on $\oE$. So, $s^{(\lambda+i\xi)/2}u$ is in the kernel of the operator $P'_z$, which vanishes at the boundary. But by the maximum principle (Lemma \ref{lem_maximumprinciple}), this implies that $s^{(\lambda+i\xi)/2}u$ is identically zero.
\end{proof}

\begin{rem}
	We didn't use any geometric facts about the expander here. The same argument would work for any other asymptotically conical Lagrangian.
\end{rem}

Now, we further weight the operator by $s_C^\mu$, for some $\mu >0$, where $s_C = (C+|w|^2)^{-1}$ for $C>0$, and show that there are no nontrivial zeroes for $\lambda >2-m$. Similarly as before, we can define the operator \[
P''_\mu(u) := s_C^{-\mu}P'_z(s_C^\mu u).  
\]
Then, we can again write this as a restriction, \[
P''_\mu = s_C^{-\mu}\tau^{-(\lambda+i\xi)}(\rho^2\pa_v-\cL)\tau^{(\lambda+i\xi)}s_C^\mu|_{\oE},
\]
which on writing in $(\tau,w^i)$-coordinates, with $\tw_C = (C+|w|^2)^{1/2}$, gives us
\begin{equation}
	\begin{split}
&s_C^{-\mu}\tau^{-(\lambda+i\xi)}(\rho^2\pa_v-\cL)\tau^{(\lambda+i\xi)}s_C^\mu \\ &= \tw^2(\tau\pa_\tau -(w\pa_w)^\top-\Delta_E) \\ &+ \tw^2\left(\lambda+i\xi + 2\mu(\tw_C^{-1}(w\pa_w)^\top \tw_C)-2\mu(2\mu+1)|\tw_C^{-1}\nabla_E\tw|^2 +2\mu \tw_C^{-1}\Delta_E\tw_C\right).
\end{split}
\end{equation}
Let us denote the function \[
N:= \tw^2\left(\lambda+i\xi + 2\mu(\tw_C^{-1}(w\pa_w)^\top \tw_C)-2\mu(2\mu+1)|\tw_C^{-1}\nabla_E\tw|^2 +2\mu \tw_C^{-1}\Delta_E\tw_C\right).
\]
Then, we have that 
\[
P''_\mu = s_C^{-\mu}\tau^{-(\lambda+i\xi)}(\rho^2\pa_v-\cL)\tau^{(\lambda+i\xi)}s_C^\mu|_{\oE} = -P_0 + N.
\]
Therefore, by the same application of the maximum principle and noting that any element in the kernel must vanish to all orders, $\Re(N)>0$ on $\oE$ implies that the kernel is trivial. Now, we compute the function $N$. We have,
\[
N = \tw^2(\lambda+i\xi)+2\mu \tw^2\tw_C^{-2}|w|^2 +(2\mu^2+4\mu)\tw^2\tw_C^{-4}|w^\perp|^2+2\mu(m-2-\mu)\tw^2\tw_C^{-2}.
\]
Now, let us pick $\mu = \frac{m-2}{2}$. Then, we get\[
\Re(N) = \tw^2\left(\frac{(\lambda+m-2)|w|^2}{|w|^2+C} + \frac{C\lambda}{|w|^2+C} + \frac{(m-2)^2}{2(|w|^2+C)}\right)+ \frac{m^2-4}{2}\tw_C^{-4}\tw^2|w^\perp|^2.
\]

Then, we get that $\Re(N)$ is uniformly bounded away from $0$ for $\lambda > 2-m$, by choosing $0<C\ll1$. Thus, we have shown:
\begin{prop}\label{prop_nokernel}
The parametric operator $(-P_z,|_{\pa\oE})$ has trivial kernel for $\Re(z) >\min(2-m,0)$.
\end{prop}

\subsection{Discreteness of critical rates}
In this section, we will prove part \ref{item_sparsekernel} of Theorem \ref{thm_mappingff}. The proof is very similar to the proof that an elliptic operator on a compact manifold has discrete spectrum. Recall that the parametric operator is given by \[
-P_z =  -P_0 + z\otimes w  +zh_2 -z^2h_1,
\]
for $h_1,h_2$ being smooth function on $\oE$, and $w$ a $b$-vector field on $\oE$.

Suppose to the contrary that we have a collection of critical rates $z_j \rt z$, so that we have elements $u_j$ such that $(-P_{z_j},|_{\pa\oE})u_j=(0,0)$. We choose a $z_0$ which is not a critical rate (thus $(-P_{z_0},|_{\pa\oE})$ is an isomorphism) and such that $z_0\neq \pm z$; in particular, by Proposition \ref{prop_nokernel}, we can take any $z_0$ with $\Re(z_0)> |\Re(z)|$. Since the set of non-critical rates is open (by continuity of the inverse norm when it exists), there are no critical rates in a neighbourhood of $z_0$. Let us denote $Z:=\bigcup_{i\geq 1} z_i$. If $Z$ is finite, then we already have discreteness since the kernel at each critical rate is finite-dimensional by Proposition \ref{prop_zeroindex}. Therefore, we can assume $Z$ is infinite. Let us define the vector subspaces formed by the $u_i$ as $V_k:= \textrm{span} (u_1,\dots,u_k)$, and their union $V=\bigcup_{i\geq 1}V_i$. First, we show that:

\begin{lem}
	The vector space $V$ is infinite-dimensional.
\end{lem}
\begin{proof}

	Suppose to the contrary, that $V$ is spanned by $k$ linearly independent solutions $\{u_1,\dots,u_k\}$. The main claim here is that there exists $j \geq 1$ such that the $k$ functions\[
	\{(w+h_2)u_i+(z_j+z_i)h_1u_i\}_{i=1}^k
	\]
	are linearly independent. Indeed, this follows by writing out projections of these functions to $V$ in terms of the basis $\{u_1,\dots,u_k\}$, and noting that the determinant of these is given by evaluating a polynomial $P(z_j)$, where $P(z) = \Pi_{i=1}^k(z+z_i)h_1^k + \{\textrm{lower order terms}\}$ with coefficients being functions. Due to the coefficient of $z^k$ being a non-zero function (as $h_1$ is non-zero by Lemma \ref{lem_h1nonzero}), this has only finitely many zeroes -- combined with $Z$ being infinite, this implies the existence of $z_j$ such that the determinant is nonzero, hence the functions are linearly independent.
	
	Then, choosing such a $z_j$, we note that since $P_{z_i}(u_i)=0$ for $i=1,\dots,k$, we have for $j>k$ \[
	P_{z_j}\left(\sum_{i=1}^ka_iu_i\right) = \sum_{i=1}^k\alpha_i((z_j-z_i)(w+h_2)(u_i)+(z_j^2-z_i^2)h_1u_i), 
	\]
	which cannot be zero by linear independence. This is a contradiction.	 
\end{proof}
Now, we extract an infinite subspace of $V$ inductively in the following way: at the first step, we choose $V'_1:= (u_1)$. At the $k^\textrm{th}$-step, having defined $V'_k$, we choose the smallest $l>k$ such that $u_l \not\in \la u_1,\dots,u_k, (w+h_2)u_1,\dots,(w+h_2)u_k\ra$. Then, we define $V'_{k+1}:=(u_l)+V'_k$. 

Then, we choose an orthogonal set of functions $w_i\in \check{L}^2_{0,0}(\oE)\oplus L^2_{2}(\pa\oE)$, such that $\la w_1,\dots,w_k\ra = V'_k$, $w_k\perp V'_{k-1}$, and normalised to $\|w_k\|_{\check{L}^2_{0,0}\oplus L^2_{2}}=1$. In particular, we can write $w_k=\sum_{i=1}^ka_iu_i$ for some $\{a_i\}$. Now, denoting $R = (-P_{z_0},|_{\pa\oE})$, we can write, \begin{equation}\begin{split}
		(-P_{z_k},|_{\pa\oE}) &= R + (z_k-z_0)(w+h_2)+(z_k^2-z_0^2)h_1 \\
		&= (z^2_k-z_0^2)R\circ((z_k^2-z_0^2)^{-1} +\frac{(z_k-z_0)}{(z_k^{2}-z_0^2)}(w+h_2)\circ R^{-1}+h_1R^{-1}\circ \textrm{Id}). 
	\end{split}
\end{equation}
Then, $(-P_{z_k},|_{\pa\oE})(u_k)=0$ implies that \begin{equation}\label{eq_infinitekernels}
	((z_k^2-z_0^2)^{-1} +\frac{1}{(z_k+z_0)}(w+h_2)\circ R^{-1}+h_1R^{-1}\circ \textrm{Id})u_k=0.
\end{equation}
Now, we note that both the mappings $h_1R^{-1}\circ \textrm{Id}$ and $(w+h_2)\circ R^{-1}$ from $\check{L}^2_{0,0}(\oE)\oplus L^2_{2}(\pa\oE)$ to itself are compact by Lemma \ref{lem_rellichkondrachovE} and Corollary \ref{lem_compactfield}, and $\frac{1}{(z_k+z_0)}$ is convergent to a finite value since $z_0\neq -z$. Hence we can find a subsequence of $w_k$ such that the sequence\[ \left(\frac{1}{(z_k+z_0)}(w+h_2)\circ R^{-1}+h_1R^{-1}\circ \textrm{Id}\right)w_k\] is convergent. From equation \eqref{eq_infinitekernels} and writing $w_k=\sum_{i=1}^ka_iu_i$, we get \[
\left(\frac{1}{(z_k+z_0)}(w+h_2)\circ R^{-1}+h_1R^{-1}\circ \textrm{Id}\right)w_k= -(z_k^2-z_0^2)^{-1}w_k +  w_k',
\]
for a function $w_k' \in \la u_1,\dots,u_{k-1}, (w+h_2)u_1,\dots,(w+h_2)u_{k-1}\ra$. Therefore for $l >k$, we have \begin{equation}\begin{split}
&\left(\frac{1}{(z_l+z_0)}(w+h_2)\circ R^{-1}+h_1R^{-1}\circ \textrm{Id}\right)w_l - \left(\frac{1}{(z_k+z_0)}(w+h_2)\circ R^{-1}+h_1R^{-1}\circ \textrm{Id}\right)w_k \\
& = -(z_l^2-z_0^2)^{-1}w_l + w''_l,
\end{split}
\end{equation}
for a function $w''_l \in \la u_1,\dots,u_{l-1},(w+h_2)u_1,\dots,(w+h_2)u_{l-1}\ra$.
Since $w_k$ is orthogonal to $w''_k$ by construction, we have that the norm of the above difference is $\geq \frac{1}{2}|z^2-z_0^2| \neq 0$ for $l$ sufficiently large, contradicting convergence. Therefore, we obtain:

\begin{prop}\label{prop_discretecritrates}
	The set of critical rates $\cC_E$ is a discrete subset of $\bC$.
\end{prop}

\subsection{Proof of Theorem \ref{thm_mappingE}}
In this section, we will prove part \ref{item_finitestrip} of Theorem \ref{thm_mappingff} and Theorem \ref{thm_mappingE}. We first prove the following:
\begin{prop}\label{prop_sparsekernel}
	The set of critical rates in any bounded real strip $\{z : a < \Re(z) < b\}$ is finite.
\end{prop}
\begin{proof}To show this, we will use the Fourier transform. Let $u$ be a smooth function with compact support away from the front face, and let us denote $\tu := \rho^{-\lambda}u$. From the a-priori estimate on $\bE$, we have the estimate \[
	\|^m\nabla_{\rho\pa_\rho}\tu\|^2_{L^2(\bE)}\leq C_\lambda(\|\rho^{-\lambda}(\rho^2\pa_v-\cL)\rho^\lambda\tu \|^2_{L^2(\bE)}+\|\tu\|^2_{L^2(\bE)}+\|\tu|_{\botf}\|^2_{L^2_{1}(\pa\bE)}).
	\]
	
	Taking the Fourier transform in the variable $\rho'=\log(\rho)$ on $(\sigma,\rho)\in\oE\times \llbracket0,\infty) = \bE$, we get by Plancherel's theorem \begin{equation}\label{eq_bigxiestimate}
		\begin{split}
		&\int_{\bR}\|\xi\hat{u}(\xi)\|^2_{L^2(\oE)}d\xi\\
		&\leq C_\lambda\left(\int_{\bR}\|-P_{\lambda+i\xi}(\hat{u}(\xi))\|^2_{L^2(\oE)}d\xi + \int_{\bR}\|\hat{u}(\xi)\|^2_{L^2(\oE)}d\xi + \int_{\bR}\|\hat{u}(\xi)|_{\pa\oE}\|^2_{L^2_1(\pa\oE)}d\xi\right),
		\end{split}
	\end{equation}
	where $\hat{u}:= \cF(\tu)$. Note that we have used \[
	\cF(^m\nabla_{\rho\pa_\rho}\tu)(\xi) = i\xi\hat{u}(\xi), \;\: \cF(\rho^{-\lambda}(\rho^2\pa_v-\cL)\rho^\lambda\tu)(\xi) = -P_{\lambda+i\xi}(\hat{u}(\xi)).
	\]
	Now, we claim that there are no critical rates for $|\xi|>2\sqrt{C_\lambda}$. Suppose to the contrary, we have a nontrivial element $u_{\xi_0}$ in the kernel of $(-P_{\lambda+i\xi_0},|_{\botf})$. Then, by multiplying with a cutoff function $\psi(\xi)$, we can extend $u_{\xi_0}$ to a nonzero function $\tu_{\xi_0}:= \psi(\xi)\cdot u_{\xi_0}$ on all of $\bE$, such that we have an estimate \[
	\int_{\bR}\|-P_{\lambda+i\xi}(\tu_{\xi_0}(\xi))\|^2_{L^2(\oE)}d\xi + \int_{\bR}\|\tu_{\xi_0}(\xi)|_{\pa\oE}\|^2_{L^2_1(\pa\oE)}d\xi < \ep\cdot \int_{\bR}\|\tu_{\xi_0}(\xi)\|^2_{L^2(\oE)}d\xi
	\]
	with $\ep$ being arbitrarily small.
	Then, since $\tu_{\xi_0}(\xi)$ has rapidly deacaying derivatives of all orders, it is the Fourier transform of a $L^2_{2,1}(\bE)$-function. Therefore, by estimate \eqref{eq_bigxiestimate} and $|\xi_0|>2\sqrt{C_\lambda}$, we arrive at a contradiction.

	So, it follows that there are no critical rates for $|\xi|>2\sqrt{C_\lambda}$. Combining the fact that critical rates are discrete (Proposition \ref{prop_discretecritrates}) and the continuity of the constant $C_\lambda$, we obtain the proposition. 
\end{proof}
Therefore, it follows that the set $\cD_E$ is discrete. Now, we have that:

\begin{lem}\label{lem_fredholmnoncritical}
	Suppose that $\lambda\in \bR \backslash \cD_E$ is not the real part of a critical rate. Then, there exists a constant $C''_\lambda$ for which \[
	\|u\|_{L^2_{\lambda}(\bE)} \leq C''_\lambda(\|(\rho^2\pa_v - \cL)u\|_{L^2_{\lambda}(\bE)}  +\|u|_{\botf}\|_{L^2_{1,\lambda}(\pa\bE)}).
	\]  
\end{lem}
\begin{proof}
	We have from the estimate \eqref{eq_bigxiestimate} that 
	\begin{equation}\label{eq_fourierestimate1}
		\begin{split}
		&\int_{\bR}\|\xi\hat{u}(\xi)\|^2_{L^2(\oE)}d\xi\\
		&\leq C_\lambda\left(\int_{\bR}\|-P_{\lambda+i\xi}(\hat{u}(\xi))\|^2_{L^2(\oE)}d\xi + \int_{\bR}\|\hat{u}(\xi)\|^2_{L^2(\oE)}d\xi + \int_{\bR}\|\hat{u}(\xi)|_{\pa\oE}\|^2_{L^2_1(\pa\oE)}d\xi\right).
		\end{split}
	\end{equation}
	Since we assumed that $\lambda \not \in \cD_E$, due to continuity of the norm, we have that there exists a constant $C'_\lambda$ such that \begin{equation}\label{eq_fourierestimate2}
		\|\hat{u}(\xi)\|_{L^2(\oE)} \leq C'_\lambda (\|-P_{\lambda+i\xi}(\hat{u}(\xi))\|_{L^2(\oE)} + \|\hat{u}(\xi)|_{\pa\oE}\|_{L^2_1(\pa\oE)}) \textrm{ for } |\xi|^2 \leq C_\lambda+1.
	\end{equation}
	Therefore, we can use \eqref{eq_fourierestimate2} to estimate the term $\|\hat{u}(\xi)\|^2_{L^2(\oE)}$ in \eqref{eq_fourierestimate1} for $|\xi|^2\leq C_{\lambda}+1$, while the LHS itself dominates $\|\hat{u}(\xi)\|^2_{L^2(\oE)}$ for $|\xi|^2\geq C_{\lambda}+1$. So, we obtain a constant $C''_\lambda$ such that \[
	\int_{\bR}\|\hat{u}(\xi)\|^2_{L^2(\oE)}d\xi\leq C''_\lambda\left(\int_{\bR}\|-P_{\lambda+i\xi}(\hat{u}(\xi))\|^2_{L^2(\oE)}d\xi + \int_{\bR}\|\hat{u}(\xi)|_{\pa\oE}\|^2_{L^2_1(\pa\oE)}d\xi\right).
	\]
	Then, taking the inverse Fourier transform and invoking Plancherel's theorem gives us the claimed estimate. 
\end{proof}
As a corollary, we obtain:
\begin{coro}\label{coro_semiFredholm}
	The mapping $T^2_{k+2r,r,\lambda}$ has closed image and is injective for all $\lambda \in \bR\backslash \cD_E$.
\end{coro}
\begin{proof}
	First, we have the a-priori estimate on $\bE$,
	\[
	\|u\|_{L^2_{k+2r,r,\lambda}}
	\leq C_{\lambda}\left(\|(\rho^2\pa_v-\cL)u\|_{L^2_{k+2r-2,r-1,\lambda}}+\|u\|_{L^2_\lambda}+\|u|_{\mathrm{bf}}\|_{L^2_{k+2r,\lambda}}\right).
	\]
	Then by Lemma \ref{lem_fredholmnoncritical}, we get \[
	\|u\|_{L^2_{\lambda}} \leq C''_\lambda(\|(\rho^2\pa_v - \cL)u\|_{L^2_{\lambda}}  +\|u|_{\textrm{bf}}\|_{L^2_{k+2r,\lambda}}),
	\]
	for some constant $C''_\lambda$. Combining these, we obtain \[
	\|u\|_{L^2_{k+2r,r,\lambda}}
	\leq C'''_{\lambda}\left(\|(\rho^2\pa_v-\cL)u\|_{L^2_{k+2r-2,r-1,\lambda}}+\|u|_{\mathrm{bf}}\|_{L^2_{k+2r,\lambda}}\right),
	\]
	for a constant $C'''_\lambda$. This proves the claim.
\end{proof}

\begin{rem}
	It is in fact possible to prove that the mapping \eqref{eq_sobolevmappingE} is an \textit{isomorphism} away from the set of rates $\cD_E$. However, for our applications it is enough to know that the mapping is semi-Fredholm and injective.
\end{rem}

\section{Isomorphism of linearised operator}\label{sec_linearproblem}
Now, we will show that the mapping \eqref{eq_mapping} between the following spaces: 
\begin{equation}\label{eq_sobolevmappingl2}
(\rho^2\pa_v - \cL,|_{\botf}) : L^2_{k+2r,r,\lambda}(\bL) \rt L^2_{k+2r-2,r-1,\lambda}(\bL)\oplus L^2_{k+2r-1}(\widetilde{L_0})
\end{equation} 
is an isomorphism for the rates $\lambda > \min(2-m,0)$. 
\begin{rem}Note that for our problem, the relevant rate is $\lambda=2$.
\end{rem}
 The argument is essentially to write the operator on $\bL$ as a `gluing' of a Fredholm operator and the operator on $\bE$ via a partition of unity, and invoke Theorem \ref{thm_mappingE} to prove the mapping has closed image for $\lambda \in \bR\backslash \cD_E$. Then, we show that the mapping \eqref{eq_mapping} is injective for $\lambda>0$ via the maximum principle. Finally, due to our Lagrangians being compact, surjectivity is a trivial consequence of having a closed image. 

\subsection{Semi-Fredholmness of the operator}
First, we prove that the linearised operator is semi-Fredholm, i.e$.$ it has finite dimensional kernel and closed range on the whole of $\bL$, for $\lambda \in \bR\backslash \cD_E$. First, cover $\bL$ using two open sets $U_1 := \mathbb{L} \cap \{ 0 \leq \rho < \varepsilon \}$ and $U_2 := \mathbb{L} \cap \{ \rho > \varepsilon/2 \}$. Let $\{\phi_1,\phi_2\}$ be a partition of unity with respect to this cover of $\mathbb{L}$. Then, we have the following::

\begin{lem}\label{lem_semiFredholm2}
	For a given rate $\lambda$, suppose we have that
	\begin{equation}\label{eq_l2estimate2}
		\|u\|_{L^2_{0,0,\lambda}} \leq C(\|(\rho^2\partial_v - \mathcal{L})u\|_{L^2_{0,0,\lambda}} + \|u|_{\mathrm{bf}}\|_{L^2_{1,\lambda}} ),
	\end{equation}
	for $u$ being supported in $U_1$. Then, the mapping (\ref{eq_sobolevmappingl2}) has finite-dimensional kernel and closed range.
\end{lem}

In order to prove this, we need the following:

\begin{lem}\label{lem_RK}
	Let $\psi$ be a smooth function supported away from the front face of $\mathbb{L}$. Then,  the mapping $\psi : L^2_{2,1,\lambda} \rightarrow L^2_{0,0,\lambda}$, sending
	\[
	u \mapsto \psi u,
	\]
	is a compact mapping.
\end{lem}
\begin{proof}
	This basically follows from classical Rellich-Kondrachov, since the Sobolev spaces we are dealing with are non-singular as $\rho$ is bounded below in the support of $\psi$.
\end{proof}

Now, we can prove Lemma \ref{lem_semiFredholm2}: 
\begin{proof}                               
	Let $(\phi_1,\phi_2)$ be a partition of unity corresponding to the cover $(U_1,U_2)$ of $\mathbb{L}$. The main idea is to consider the estimates for $\phi_1u$ and $\phi_2u$, and add them up, such that we can get a global estimate of $u$ by the norms of $((\rho^2\partial_v - \cL)u,u|_{\textrm{bf}})$, and the norms of compact error terms. Then, we would get closed image and finite kernel by Lemma \ref{lem_funcana}. Now, for $u\in L^2_{2,1,\lambda}$, we have by the a-priori estimate and the given hypothesis that
	\begin{equation}
		\begin{split}
			\|\phi_1u\|_{L^2_{2,1,\lambda}} 
			&\leq C(\|(\rho^2\partial_v - \mathcal{L})\phi_1u\|_{L^2_{0,0,\lambda}} + \|\phi_1u|_{\mathrm{bf}}\|_{L^2_{1,\lambda}} +\|\phi_1u\|_{L^2_{0,0,\lambda}}) \\
			& \leq C'(\|(\rho^2\partial_v - \mathcal{L})\phi_1u\|_{L^2_{0,0,\lambda}} + \|\phi_1u|_{\mathrm{bf}}\|_{L^2_{1,\lambda}})\\
			&\leq C'(\|\phi_1(\rho^2\partial_v - \mathcal{L})u\|_{L^2_{0,0,\lambda}} + \|\phi_1u|_{\botf}\|_{L^2_{1,\lambda}} + A\|{^b}\nabla_\rel \phi_1 \cdot {^b\nabla_\rel} u\|_{L^2_{{0,0,\lambda}}} \\
			&   \: \: \: \: \: \: + B\|{^b\Delta_\rel} \phi_1 u\|_{L^2_{0,0,\lambda}}),\\
		\end{split}
	\end{equation} 
	for bounded constants $A,B$. Similarly, we have from the a-priori estimate, \begin{equation}
		\begin{split}
			\|\phi_2u\|_{L^2_{2,1,\lambda}} 
			&\leq C(\|(\rho^2\partial_v - \mathcal{L})\phi_2u\|_{L^2_{0,0,\lambda}} + \|\phi_2u|_{\mathrm{bf}}\|_{L^2_{1,\lambda}} +\|\phi_2u\|_{L^2_{0,0,\lambda}}) \\
			&\leq C(\|\phi_2(\rho^2\partial_v - \mathcal{L})u\|_{L^2_{0,0,\lambda}} + \|\phi_2u|_{\botf}\|_{L^2_{1,\lambda}} + \|\phi_2u\|_{L^2_{0,0,\lambda}}\\  & \: \: \: \: \: \: + A\|{^b\nabla_\rel} \phi_2 \cdot {^b\nabla_\rel} u\|_{L^2_{{0,0,\lambda}}} + B\|{^b\Delta_\rel} \phi_2 u\|_{L^2_{0,0,\lambda}}).\\
		\end{split}
	\end{equation}
	Therefore, we have by adding the above two estimates,
	\begin{equation}
		\begin{split}
			\|u\|_{L^2_{2,1,\lambda}} 
			&\leq \|\phi_1u\|_{L^2_{2,1,\lambda}}+\|\phi_2u\|_{L^2_{2,1,\lambda}}\\
			& \leq 2C(\|(\rho^2\partial_v - \mathcal{L})u\|_{L^2_{0,0,\lambda}} + \|u|_{\mathrm{bf}}\|_{L^2_{1,\lambda}} +\|\phi_2u\|_{L^2_{0,0,\lambda}} + A\|{^b\nabla_\rel} \phi_2 \cdot {^b\nabla_\rel} u\|_{L^2_{{0,0,\lambda}}} \\
			& \: \: \: \: \: \:  + B\|{^b\Delta_\rel} \phi_2 u\|_{L^2_{0,0,\lambda}} + A\|{^b\nabla_\rel} \phi_1 \cdot {^b\nabla_\rel} u\|_{L^2_{{0,0,\lambda}}} + B\|{^b\Delta_\rel} \phi_1 u\|_{L^2_{0,0,\lambda}}).\\
		\end{split}
	\end{equation}
	Now, since each of $\phi_2,{^b\nabla_\rel}\phi_i,{^b\Delta_\rel}\phi_i$ is supported away from the front face, we have that the last five terms in the above estimate are the norm of a compact operator by Lemma \ref{lem_RK}. Therefore, by Lemma \ref{lem_funcana}, we get that (\ref{eq_sobolevmappingl2}) has finite dimensional kernel and closed range for $k=0,r=1$. The case $k,r\geq 1$ follows similarly from the a-priori estimate for the mapping.
\end{proof}

Note that there is an a-diffeomorphism $\Phi:\bB_{\bE} \rt \bL$ (where $\bB_\bE$ is a ball around the exceptional divisor of $\bE$) to a neighbourhood of the front face of $\bL$, such that it is an isometry for the $b$-metrics at the front face. The key point is that $\bE$ is the `local model at infinity' for the approximate solution $\bL$.

Therefore, it remains to prove the estimate \eqref{eq_l2estimate2} to prove semi-Fredholmness. It follows essentially from Theorem \ref{thm_mappingE}, due to there being an a-smooth diffeomorphism $\Phi: U_\bL \rt U_\bE$ that is an isometry at the front face of $\bL$. Due to this, the operators $\rho^2\pa_v-\cL$ on $\bL$ and $\bE$, as well as the Sobolev norms, differ up to a term that vanishes at the front face. So, we can make the neighbourhood smaller, and obtain an isometry. In particular, we have:
\begin{lem}
	Let $\lambda \in \bR\backslash\cD_E $ be a given non-critical rate. Then, there exists a sufficiently small neighbourhood $U_\bL$ of the front face of $\bL$, such that for a function $u\in L^2_{2,1}(\bL)$ supported on $U_\bL$, we have \[
	\|u\|_{L^2_{2,1}}\leq C(\|(\rho^2\pa_v-\cL)u\|_{L^2}+\|u|_{t=0}\|_{L^2_1}),
	\]
	for a uniform constant $C$.
\end{lem}

Thus, we have shown:
\begin{prop}
	The mapping \eqref{eq_sobolevmappingl2} has closed image and finite-dimensional kernel for $\lambda\in \bR\backslash \cD_E$.
\end{prop}

\subsection{Proof of surjectivity and isomorphism}

To prove surjectivity of the mapping \eqref{eq_sobolevmappingl2}, we show surjectivity onto $C^\infty$ functions supported away from the front face, which we denote as before by $C^\infty_{\ff}$. Since such functions are dense in $L^2_{k+2r-2,r-1,\lambda}\oplus L^2_{k+2r-1,\lambda}$, we would be done as the image is closed.

Let $f,g$ be $C^\infty_{\ff}$ functions on $\mathbb{L}$ and $\{s =0\}$ respectively. We will construct a function in $L^2_{k+2r-2,r-1,\lambda}$ that maps to $(f,g)$ under \eqref{eq_sobolevmappingl2}. First, we construct a $C^\infty_\ff$ function $\tilde{u}$, supported away from the front face, such that \[
(\rho^2\partial_v - \mathcal{L})\tilde{u}=\tilde{f}, \;\; \tilde{u}|_{t=0} = g,
\]
such that $f-\tilde{f} \in \mathring{L}^2_{k+2r,r,\lambda}$, i.e. we can find a sequence of smooth functions $h_i$ supported away from $t=0$ such that $h_i \rightarrow f - \tilde{f} $ in $L^2_{k+2r,r,\lambda}$. To construct $\tilde{u}$, we will solve the operator for the function $\tilde{u}$ using a Taylor expansion at $t=0$. 

Suppose we have that $f(x,t) = \sum_{k\geq 0} q_k(x)t^k$, for $q_k(x)$ being functions supported on $L\backslash \overline{B_\ep(x)}$ for $\ep$ small. Then, if $\tilde{u} = \sum_{k\geq 0} s_k(x)t^k$, we have that \[
(\rho^2\partial_v  - \mathcal{L})\tilde{u} = \sum_{k\geq 0 } (\rho^2(k+1) s_{k+1} - \mathcal{L}s_k)t^k. 
\]
Now, note that $\mathcal{L}$ is not constant in time. Therefore, we need to further expand $\mathcal{L}(s_k)$ as a Taylor series. Doing so, we can recursively calculate the functions $s_k$ so that the Taylor series of the output matches the Taylor series of $f$, with $s_0 = g$. Further, the functions will also be uniformly supported away from the front face.  In order to construct a smooth function $\tu$ with these asymptotics, we invoke Borel's lemma, whose simplest form says that given any sequence of smooth functions at the boundary of a collar neighbourhood $U\times [0,\ep)$, there exists a smooth function on the neighbourhood such that its Taylor series expansion at the boundary has coefficients given by these functions. So, we get $C^\infty_\ff$ functions $\tu,\tilde{f}$ satisfying the conditions above.

Having constructed $\tilde{f}$, our problem reduces to showing surjectivity onto $C^\infty_\ff$ functions vanishing to sufficiently high orders at $t=0$, and boundary condition being the zero function. However, such functions can be approximated by smooth functions that are supported in $t\geq \varepsilon$ for $\varepsilon >0$. Clearly, we can solve for these functions as they are the standard Cauchy problem (i.e. without any degeneracy). Therefore, since the image is closed, we can surject onto functions vanishing to all orders, and hence onto all functions in the output space as claimed.

Therefore, we have shown the following:
\begin{theorem}\label{thm_isomorphismL}
	The linearised operator for LMCF \eqref{eq_sobolevmappingl2}, considered as a map between the weighted Sobolev spaces $L^2_{k+2r,r,\lambda}$ and $L^2_{k+2r-2,r-1,\lambda}\oplus L^2_{k+2r,\lambda}(\overline{L})$, is an isomorphism for $\lambda > 0$, and $k\geq 0, r\geq 1$.
\end{theorem}
\begin{proof}
	We already know that the mapping is surjective and has finite-dimensional kernel. The mapping is injective for $\lambda>0$ by an application of the maximum principle for $(\rho^2\pa_v-\cL)u = \rho^2(\pa_vu-\Delta_{L_t}u-H\tu + v^\perp \tu)$. In particular, if an element $u$ lies in the kernel at a rate $\lambda>0$, it implies that $u\in C^{k+2r,r,\lambda}$ for $k,r\gg1$ by the a-priori estimate and Sobolev embedding. Therefore, it follows that $\rho^{-\lambda} u$ is a bounded, smooth function on $\bL$, so $u$ is a bounded, smooth function on $\bL$ decaying to order $\rho^{-\lambda}$ at the front face, and $u|_{s=0}\equiv0$. Now, the maximum principle applied to $\pa_vu-\Delta_{L_t}u-H\tu+v^\perp \tu$ over a time interval $(\nu,\ep)$ tells us that the supremum of $u$ is bounded by $\sup(u|_{t=\nu})$. Letting $\nu \downarrow 0$, we know that $\sup(u|_{t=\nu}) \downarrow 0$ given the decay of $u$ at the front face and $u|_{s=0}\equiv 0$. Therefore, we obtain that $u$ is identically zero on $\bL$.    
\end{proof}

\section{Solving the nonlinear problem}\label{sec_nonlinearproblem}
In this section, we solve the nonlinear LMCF equation (\ref{1}) on potentials, and prove that the resulting solution corresponds to an a-smooth Lagrangian submanifold in $\bX$, satisfying the conditions of Theorem \ref{thm_main} as claimed. 
\subsection{The zeroth order error term}
Recall that the zeroth order error term in the LMCF P.D.E.\ \eqref{eq_potentialequation} is given by $\rho^2E$, where $E$ satisfies $\dr(E) = \dr\theta_0-\hat{\omega}(v^\perp,\_)$. Firstly, we note that:
\begin{lem}
	The relative $1$-form $d_\rel E$ is a-smooth. Moreover, we have that $d_\rel(E)=0$ at the boundaries of $\bL$.
\end{lem}
\begin{proof}
	Note that it is clear $d_\rel E$ is a smooth relative $1$-form at the bottom face of $\bL$. Therefore, it is enough to show a-smoothness near the front face. Now, we can write $\bL$ near the front face as the graph $\Gamma(d_\rel f)$ of an a-smooth $f$ function vanishing at rate $o(\rho^2)$ over the compactified family of expanders $\bE$. Since the family $\bE$ satisfies LMCF, it follows that the error for $\bL$ is given by the nonlinear part of the P.D.E.\ \eqref{eq_potentialequation} over $\bE$, plus a smooth error term arising from the embedding $\widetilde{\Upsilon}$. As a result, the error $1$-form is a-smooth due to $f$ being a-smooth and its derivatives decaying at rates faster than $\rho^2$.
	
	The second assertion follows due to the construction of approximate solution, and in particular our method of gluing construction. The choice of functions $\chi_1,\chi_2$ ensures that the approximate solution is the graph $\Gamma(td\theta_L)$ over $L$ outside of a shrinking ball, and it is a perturbation of the expander near the singularity. Therefore, $d_\rel(E)$ converges to zero on both the bottom face and front face of $\bL$. 
\end{proof}
Now, we note the following important lemma about the relative derivative on $\bL$:
\begin{lem}\label{lem_formregularity}
	Let $\alpha$ be an a-smooth, relatively exact, relative $1$-form on $\bL$ over the $b$-fibration $t:\bL\rt [0,\ep)$ such that $\alpha$ vanishes at the boundaries of $\bL$. Then, there exists an a-smooth function $f$ such that $\alpha=d_\rel f$, $f$ vanishes at the boundaries, and $f$ is unique up to adding a time dependent constant.
\end{lem}
\begin{proof}
	The uniqueness statement about $f$ is clear, since any other function $f'$ will differ from $f$ by a time-dependent constant, and since it is smooth at the ordinary boundaries of $\bL$, it is smooth in time. Therefore, it remains to prove existence.
	
	To prove existence, we consider the local models for $\bL$ and follow an inductive procedure, constructing the function $f$ on each piece and then patching them up by adding time-dependent constants. Note that such an inductive construction will give a globally well-defined function due to $\alpha$ being exact. Therefore, it is enough to construct such $f$ on local models, and in particular, for models near the front face. 
	
	Let us consider a local model $U_i$ near the front face. First, since $\alpha$ is a-smooth and vanishes near the front face, we have that $|\alpha| = O(\rho^\delta)$ for some $\delta>0$. We construct the function $f$ in the following manner: consider a smooth choice of points $\{x_t\}_{t\in [0,\ep)}$ on the time slices of $U_i$, converging to a point on the boundary. Then, we define $f_t$ on each slice to be the unique solution of $d_\rel f_t = \alpha, f_t(x_t)=0$. Clearly, this gives us a function $f:U_i\rt \bR$ which is smooth in the interior and at the boundary of $U_i$, and such that $f|_{\pa U_i} \equiv0$. Then, it remains to show that $f$ extends a-smoothly at the front faces. For each point $x_t\in U_i$ on a time-slice, we can construct a curve $\gamma_t(\rho)$ lying in the time slice, parametrised by the radial variable $\rho$, starting at $\rho=R$ and ending at $\rho=r$, such that $^b|\rho{\pa_\rho}\gamma_t|<C$, for a uniform constant $C$ independent of the point $x_t$. Then, we have $f(\gamma_t(r)) = \int_{R}^{r}d_\rel f (\pa_r(\gamma_t))dr$. Now, note that by choice of the curve, $|d_\rel f(\pa_r (\gamma_t))| =|\alpha(\pa_r\gamma_t)| \leq r^\delta\cdot{^b|}\pa_r(\gamma_t)| \leq C r^{\delta-1}$. Therefore, $f(x_t) \leq C (r^\delta-R^\delta) + \sup_{x\in U^c} f(x) $. Hence, $f$ is a bounded function on $U$. Now, we note that $\pa_v(d_\rel f)= d_\rel (\pa_v f)$ due to the time-like vector field $v$ being orthogonal to the time slices. Therefore, we get by the same method as before that $\rho^2\pa_v f = O(r^\alpha)$. Continuing similarly with higher order derivatives, we get that $f$ is a-smooth on $U_i$ and vanishes at the boundaries. This proves the claim.
\end{proof}
 Therefore, it follows that we can take the function $E$ to be an a-smooth function that vanishes at the boundaries. Now, by $a$-smoothness, this implies that $E$ decays at rate $r^\alpha$ near the front face for some $\alpha>0$, hence $E\in L^p$. Similarly, due to $E$ being a-smooth, we get $E\in L^2_{k+2r,r}$ for all $k,r\geq 1$. Hence, the error term $\rho^2E$ lies in $L^2_{k+2r,r,2}$ for all $k,r\geq 1$.

 \subsection{Solving formally near the bottom face}
Suppose we have constructed an approximate solution to LMCF as before, so that we can write the P.D.E.\ \eqref{eq_potentialequation} at the level of potentials, \[(\rho^2\pa_v-\cL)u = \rho^2E + \textrm{ higher-order nonlinear terms in }u, \]
where $E$ is the error which vanishes at the boundaries. In this subsection, we will solve the equation formally on the bottom face. To do so, we consider an ansatz, \[
u = \sum_{k=0}^\infty s^k\rho^2u_k,
\]
for $u_k : \widetilde{L_0} \rt \bR$ being a-smooth functions. Now, as we have that the (unweighted by $\rho^2$) nonlinear operator is \[
P = E + \pa_v - \rho^{-2}\cL + Q[\rho^{-2}\nabla_\rel u,\rho^{-2}\nabla^2_{\rel}u],
\]
 we can write
\begin{equation}
	\begin{split}
		P[u] \sim & E+  \sum_{k=0}^\infty (k+1)u_{k+1}s^k+ \sum_{k=0}^\infty s^k{\cL|_{L_0}}(u_{k}) \\
		& + \sum_{\substack{k=0 \\ i+j=k}}^\infty s^kP_{ij}''[\nabla_\rel u_i,\nabla_\rel u_j],
	\end{split}
\end{equation}
as $s\downarrow 0$, where $P_{ij}''$ is a partial differential operator with terms of degree $\geq1$, and a-smooth coefficients. From this, it is clear that once we fix $u_0\equiv 0$, $u_{k+1}$ is uniquely determined from $\{u_i\}_{1\leq i\leq k}$. Also, we note that $u_{k+1}$ is also a-smooth on $\widetilde{L_0}$ in the radial variable $\rho$, since the partial differential operators are a-smooth. Therefore, we have formally solved for the functions $\{u_k\}_{k\geq0 }$ on $\widetilde{L_0}$. Now, it is straightforward to extend Borel's lemma to this case to obtain an a-smooth function $\tu:\bL\rt \bR$ with the asymptotic expansion at $s=0$ as above. Therefore by considering our new approximate solution as $\Gamma(d_\rel \tu)$ over $\bL$, we have:
\begin{prop}\label{prop_bottomfacesolving}
There exists an a-smooth embedding $\bL\hookrightarrow \bX$ with the given boundary conditions, such that the error term in the nonlinear P.D.E.\ \eqref{eq_potentialequation} decays at rate $O(s^\infty)$ as $s\downarrow 0$.  
\end{prop}

\subsection{Estimates on the nonlinear operator}
We will prove an estimate on the nonlinear part of our P.D.E., similar to Pacini \cite{Pac13}, using the following Lemma:
\begin{lem}
	Let $V$ be a normed vector space, and let $Q: V\rt\bR$ be a smooth function satisfying $Q(0)=DQ(0)=0$. Then for $x,y\in B_R(0)$, $|Q(x)-Q(y)|\leq C|x+y|\cdot|x-y|$, where $C=\sup_{x\in B_R(0)}|D^2Q(x)|$.
\end{lem} 
First, we note that the derivatives of the nonlinear term $\rho^2Q(\rho^{-2}{d_\rel u}, \rho^{-2}{^b\nabla_\rel^2u})$ can be written as the composition of derivatives on $u$ with  an a-smooth function on a sum of vector bundles having nonlinear terms of degree at least $2$ (since differentiation preserves degree of nonlinearity). Then, we choose $k,r$ large enough so that there exists a Sobolev embedding $L^2_{k+2r,r}\hookrightarrow C^{k/2+r,r/2}$, and we have a continuous mapping\[
L^2_{k+2r,r}\times L^2_{k+2r,r}\rt L^2_{k+2r,r},
\] 
i.e.\ $\|fg\|_{L^2_{k+2r,r}}\leq C\|f\|_{L^2_{k+2r,r}}\cdot\|g\|_{L^2_{k+2r,r}}$. Given this, we have the following:

\begin{theorem}\label{thm_quadratic}
	Let us consider the nonlinear term of our P.D.E.\ \eqref{eq_potentialequation} given by $\rho^{-2}Q(\rho^{-2}d_\rel u, \rho^{-2} {^b\nabla_\rel^2u})$, which we abbreviate to $Q[u]$. For $k,r$ large enough such that the Sobolev embedding holds, and for \[
	\|u\|_{L^2_{k+2r,r,\lambda}},\|v\|_{L^2_{k+2r,r,\lambda}} \leq K \ll 1,
	\] there is an estimate 
	\begin{equation}\label{eq_nonlinearestimatel2}
	\|Q[u] -Q[v]\|_{L^2_{k+2r-2,r-1,\lambda}} \leq C_K  \|u -v\|_{L^2_{k+2r,r,\lambda}}\cdot(\|u\|_{L^2_{k+2r,r,2}}+\|v\|_{L^2_{k+2r,r,2}}),
	\end{equation}
	for a constant $C_K$ depending on $K$.
\end{theorem}
\begin{proof}
	First, we note that \[
	\nabla_\rel Q[u] = (\nabla \rho^{2}Q)(\rho^{-2}d_\rel u, \rho^{-2}{^b\nabla_\rel^2u})\cdot(\rho^{-2}\pa_xu, \rho^{-2}\pa_xd_\rel u, \rho^{-2}\pa_x\nabla^2_\rel u),
	\]
	and similarly \[
	\nabla_\rel^k\nabla^rQ[u] = \sum (\nabla^{i+j}\rho^2Q)(\rho^{-2}d_\rel u, \rho^{-2}{^b\nabla_\rel^2 u})\cdot P_{ij}[u],
	\]
	for $P_{ij}[u]$ being a polynomial of degree $i+j$ in the derivatives of $u$ weighted by $\rho^{-2}$, with $P_{ij}[0]=0$, and $(\nabla^{i+j}\rho^2Q)(\rho^{-2}d_\rel u, \rho^{-2}{^b\nabla_\rel^2 u})$ being an a-smooth function in $\rho^{-2}\nabla_\rel u, \rho^{-2}d_\rel u$ weighted by $\rho^2$. Now, we can write pointwise, by using the parallel transport along fibres,
	\[
	\nabla_\rel^k\nabla^rQ[u] - \nabla_\rel^k\nabla^rQ[v] = \sum ((\nabla^{i+j}\rho^2Q)[u]-(\nabla^{i+j}\rho^2Q)[v])\cdot P_{ij}[u] + \sum (\nabla^{i+j}\rho^2Q)[v] (P_{ij}[u]-P_{ij}[v]).
	\]
	Now, we note that we have the pointwise estimates with weighted norms,\[
	|Q[u]-Q[v]| \leq C|u-v|_{C^{2,0,2}}\cdot|u+v|_{C^{2,0,2}},\; |\nabla^k Q[u]-\nabla^k Q[v]| \leq C|u-v|_{C^{2,0,2}},
	\]
	\[
	|Q[u]| \leq C|u|_{C^{2,0,2}}^2,\; |\nabla Q[u]|\leq C|u|_{C^{2,0,2}}, \; |\nabla^kQ[u]|\leq C.
	\]
	Also, we have estimates on the polynomials,
	\[
	|P_{ij}[u]-P_{ij}[v]| \leq C|u-v|_{C^{k+2r-i-2j,r-j,2}},
	\]
	and for $i+j>1$, we have \[
	|P_{ij}[u] - P_{ij}[v]| \leq C|u-v|_{C^{k+2r-i-2j,r-j,2}}\cdot(|u|_{C^{k+2r-i-2j,r-j,2}}+|v|_{C^{k+2r-i-2j,r-j,2}}).
	\]
	Using these pointwise estimates and the continuity of multiplication, and noting down the weights of $\rho$, we have the claimed estimate.
\end{proof}

\begin{rem}
	It is likely possible to prove the theorem with a weaker hypothesis, that we merely have a Sobolev embedding $L^2_{k+2r,r}\hookrightarrow C^{0,0,\alpha}$.
\end{rem}

\subsection{Fixed point argument}
Consider the iteration map 
\begin{equation}\label{mapping}
	I: u\mapsto (\rho^2\partial_v-\mathcal{L})^{-1}(\rho^2Q[\rho^{-2}d_\rel u, \rho^{-2}{^b\nabla}_\rel^2u]+\rho^2E, 0|_{\botf}).
\end{equation}
For $u\in L^2_{k+2r,r,2}$ such that $\|u\|_{L^2_{k+2r,r,2}}$ is sufficiently small and $k,r$ are sufficiently large, we note that this is a well-defined mapping from $L^2_{k+2r,r,2}$ to itself. This is because both $\rho^2E$ and $\rho^2Q[u]$ lie in $L^2_{k+2r,r,2}$, and the operator $(\rho^2\pa_v-\mathcal{L})$ is invertible. Clearly, a fixed point for this mapping is a solution to the nonlinear P.D.E. Now, we will show that such a fixed point exists, while reducing the time interval of existence.
\subsubsection{Reducing the error}

In order to find a fixed point for (\ref{mapping}), we first reduce the error $E$ by multiplying it with a cutoff function:

\begin{defn}
	Define the \textit{reduced error} to be the function $E_\eta :=\chi(\eta^{-1} t)E$, where $\chi : [0,\infty)\rt\bR$ satisfies $\chi=1$ on $[0,1)$ and $\chi=0$ on $[2,\infty)$, and $\eta>0$ is a small positive constant.
\end{defn}
Note that we have $E_\eta =E$ for $t\in[0,\eta)$. Therefore, if we find a solution to the nonlinear problem with error $E_\eta$, we obtain a solution to the original problem on restricting to the time interval $[0,\eta)$. First, we show that we can make the norm of the error arbitrarily small by choosing $\eta$ small enough:
\begin{lem}
	Let $E$ vanish to all orders at the bottom face. Then, given any $\ep>0$, there exists $\eta>0$ such that $\|E_{\eta}\|_{L^2_{k+2r,r,2}} < \ep$.
\end{lem}
\begin{proof}
	We essentially need to show that the derivatives of $E_\eta$ have uniformly bounded pointwise norms over $\eta$; then, the claim follows since reducing $\eta$ reduces the support of $E$. First, we note that due to $E$ being smooth and vanishing to all orders at the bottom face, all the m-derivatives of $E$ also vanish to all orders at the bottom face. Now, we note the m-derivatives of the cutoff function $\chi(\eta^{-1}t)$, which consist of derivatives $\pa_s$ and $\rho\pa_\rho$:\[
	\pa_s(\chi(\eta^{-1}t))= \chi'(\eta^{-1}t)\cdot \eta^{-1}\rho^2\leq Cs^{-1},
	\]
	\[
	\rho\pa_\rho(\chi(\eta^{-1}t)) = \chi'(\eta^{-1}t) \cdot 2\eta^{-1}\rho^2s \leq C,
	\]
	with the constant $C$ being independent of $\eta$.	Similarly, taking higher derivatives, we would obtain that $|\pa_s^k{^b\nabla^\alpha_\rel\chi(\eta^{-1}t)}| \leq C_{k,\alpha} s^{-k}$, for a uniform constant $C_{k,\alpha}$ independent of $\eta$. Therefore, due to $E$ and all its derivatives vanishing to all orders in $s$, the derivatives of $E_\eta$ are pointwise bounded by sums of functions of the form $C_{k,\alpha}s^{-k}{^b\nabla}^rE$. Now, we note that these functions decay to the front face at the same rate as $E$, and these are smooth functions at the bottom face. Therefore, the functions $C_{k,\alpha}s^{-k}{^b\nabla^k}E$ are integrable on $\bL$. Thus, since reducing $\eta$ makes the support of $E_\eta$ smaller, and $E_\eta$ and its derivatives are bounded by $L^2$-integrable functions, the $L^2_{k+2r,r,2}$-norm of $E_\eta$ decreases to $0$ as $\eta\downarrow 0$.
\end{proof}

 Now, we show that \eqref{mapping} is a contraction mapping in a small ball:
\begin{lem}
	There exist sufficiently small positive constants $\eta,\mu>0$ such that (\ref{mapping}) with error term $E_\eta$ is a well-defined contraction mapping $I: B_\mu \rt B_\mu$, where $B_\mu$ is the open ball of radius $\mu$ in $L^2_{k+2r,r,2}$. 
\end{lem}
\begin{proof}
	Let $C_1$ be the norm of the inverse linear operator, and $C_2$ be the constant for the nonlinear estimate \eqref{eq_nonlinearestimatel2}. For the image to lie in the ball $B_\mu$, we must have that \begin{equation}\label{eq_contractioncondition1}
	C_1(C_2\mu^2+E_\eta)<\mu \iff  E_\eta < C_1^{-1}\mu-C_2\mu^2,
	\end{equation}
	while to be a contraction mapping we need \begin{equation}\label{eq_contractioncondition2}
	1>2C_1C_2\mu.
	\end{equation}
	So, we can first choose $\mu$ sufficiently  small, so that \eqref{eq_contractioncondition2} holds, and then make it smaller if necessary so that $C_1^{-1}\mu-C_2\mu^2 >0$, and finally choose $\eta$ sufficiently small so that \eqref{eq_contractioncondition1} holds. Thus for example, we can choose $\mu=\frac{1}{2(1+C_1)(1+C_2)}$, and $\eta$ small so that $E_\eta < \frac{1}{4(1+C_1)^2(1+C_2)}$.	
\end{proof}

Therefore, we obtain a weak solution to the nonlinear P.D.E.:
\begin{coro}
	There exists a weak solution to the nonlinear problem \eqref{eq_potentialequation} for a short time interval $t\in[0,\ep)$, in the space $L^2_{k+2r,r,2}$. 
\end{coro}

\subsection{Regularity and
	uniqueness of the solution}
To show that the weak solution we obtained has derivatives of all orders, we invoke some parabolic regularity results from the book by Ladyzhenskaya, Solonnikov and Ural'ceva \cite{LSU68}. First, we have Theorem 12.1 from Chapter 3 of \cite{LSU68}:
\begin{theorem}\label{thm_regularity}
	Let $\Omega$ be a bounded domain in $\mathbb{R}^m$, $[a^{ij}],b^i,c$ be continuous on $\Omega_T$, $[a^{ij}]$ being elliptic, and define the linear operator \[
	Lu = \partial_t u - \partial_i(a^{ij}\partial_j u)- b^i\partial_i u - cu.
	\]
	Suppose that $u$ is a weak solution of $Lu = f$. If the coefficients of $L$ and $f$ lie in $C^{k,l,\alpha}((0,T)\times \Omega)$, then $u$ lies in $C^{k+2,l+1,\alpha}(I\times \Omega')$ for $I\subset \subset (0,T)$ and $\Omega'\subset\subset \Omega$.
\end{theorem}
We will need a version of this theorem going up to the boundary $t=0$, which follows from the previous result and Theorem 5.1 from Chapter 4 of \cite{LSU68} (c.f.\ the discussion following Theorem 12.1 of \cite{LSU68}):
\begin{prop}\label{prop_bootstrap}
	Let $\Omega$ be a bounded domain in $\mathbb{R}^m$, $[a^{ij}],b^i,c$ be continuous on $\Omega_T$, $[a^{ij}]$ being elliptic, and define the linear operator \[
	Lu = \partial_t u - \partial_i(a^{ij}\partial_j u)- b^i\partial_i u - cu.
	\]
	Suppose that $u$ is a weak solution of $Lu = f$, such that $u\in C^0(\Omega_T)$ and $u|_{t=0}\equiv0$. If the coefficients of $L$ and $f$ lie in $C^{k,l,\alpha}([0,T)\times \Omega)$, then $u$ lies in $C^{k+2,l+1,\alpha}([0,t)\times \Omega')$ for $t<T$ and $\Omega'\subset\subset \Omega$.
\end{prop}

To apply this theorem, we take our nonlinear P.D.E.\ and differentiate it to obtain linear P.D.E.s having continuous coefficients, such that they have weak solutions given by $\partial_iu$, and gain higher regularity by a bootstrapping argument.

\subsubsection{Bootstrapping}
By Sobolev embedding, the solution we have obtained in $L^2_{k+2r,r,2}$ lies in $C^{2+\delta,1+\delta/2}$ by choosing $k,r\gg 1$. Now, we will show that we can bootstrap to obtain that our solution has derivatives of all orders. Let us first note the form of our equation. We can write it as\[
E+\pa_vu -\Delta_{L_t}u+Q(\rho^{-2}d_\mathrm{rel}u,\rho^{-2}{^b\nabla^2_\mathrm{rel}u})=0,
\]
which we can take a derivative of by $^b\nabla_\mathrm{rel}^k$, and multiply it by $\rho^2$, to get the equation \begin{equation}\label{eq_bootstrap}
	\rho^2\pa_v(^b\nabla_\mathrm{rel}^ku) - \sum a_{ij}(^b\nabla_\mathrm{rel}^2u,{^b\nabla}_\mathrm{rel} u){^b\nabla_\mathrm{rel,}}^2_{ij}(^b\nabla_\mathrm{rel}^ku) +b(^b\nabla_\rel^{k-1}u,\dots,{^b\nabla}_\mathrm{rel} u, u)=0,
\end{equation}
for $a_{ij},b$ being a-smooth functions. Assuming that we have already shown $u$ has continuous relative derivatives up to order $k$, it follows that we have a weak solution $\tilde{u}:={^b\nabla^k_\mathrm{rel} u}$ of the P.D.E.\[
\rho^2\pa_v(\tilde{u}) - \sum a_{ij}(^b\nabla_\mathrm{rel}^2u,{^b\nabla}_\mathrm{rel} u){^b\nabla_\mathrm{rel,}}^2_{ij}(\tilde{u}) +b(^b\nabla_\rel^{k-1}u,\dots,{^b\nabla}_\mathrm{rel} u, u)=0,
\]
where the coefficients are in $C^0$. Now, we note that the above P.D.E. is just the linearisation of the nonlinear P.D.E. at the solution $u$. Therefore, the matrix $[a_{ij}]$ is positive-definite as well, and the equation is a linear parabolic P.D.E. Hence, by applying Proposition \ref{prop_bootstrap}, we get the existence of continuous relative derivatives of $u$ up to order $k+1$. Continuing in this fashion, we obtain by induction that all relative derivatives of $u$ are continuous. By taking further derivatives of the equation in $\pa_s$ and proceeding similarly, we get $u\in C^\infty_\loc$. Then, we can show that $u$ is a-smooth, by noting that we can do the iteration argument for any $k,r\gg1$ over a short interval of time and invoke Sobolev embedding.

\subsubsection{Uniqueness of the solution}

By the uniqueness inherent in the contraction mapping argument, our solution is unique in a small $L^2_{k+2r,r,2}$-neighbourhood of itself. However, it is also unique in the whole space $L^2_{k+2r,r,2}$ over the time interval $[0,\ep)$, provided the boundary condition $u|_{s=0}\equiv 0$. To show this, suppose there exists another solution $u\in L^2_{k+2r,r,2}$ to the nonlinear problem. Firstly, we note that by the regularity theorems of the previous section, the new solution is also smooth near the bottom face. Therefore, it follows that it must have the same asymptotic expansion near the bottom face as our solution. In particular, linearising the nonlinear P.D.E.\ at the solution we obtained, the new solution decays to all orders at the bottom face. We first note the following:

\begin{prop}
 
	Given that $f=(\rho^2\pa_v-\cL) u$ decays to all orders at the bottom face, and given any $\eta \in (0,\ep)$, we have the estimate, \[
	\|u\|_{L^2_{k+2r,r,2}} \leq C(\|f\|_{L^2_{k+2r-2,r-1,2}}),
	\]
	where the norm is taken over $\bL|_{\pi^{-1}(\llbracket0,\eta))}$, and the constant is independent of $\eta$. 
\end{prop}
\begin{proof}
	Let us work over $\bL|_{[0,\eta)}$. We have the following a-priori estimate from Lemma \ref{lem_apriorishort}, \[
	\|u\|_{L^2_{k+2r,r,2}}\leq C\left(\|f\|_{L^2_{k+2r-2,r-1,2}}+\sum_{i=1}^{r} \|\pa^i_sf|_{s=0}\|_{L^2_{k+2(r-i)-1,2}}+\|u\|_{L^2_2}\right),
	\] 
	where the constant $C$ is independent of $\eta$. The vanishing of $f$ to all orders implies that the restriction terms above vanish. Thus, we obtain an estimate \[
		\|u\|_{L^2_{k+2r,r,2}}\leq C\left(\|f\|_{L^2_{k+2r-2,r-1,2}}+\|u\|_{L^2_2}\right).
	\] Thus, we are reduced to showing the following estimate over $[0,\eta)$: \begin{equation}\label{eq_shorttimel2}
		\|u\|_{L^2_2} \leq C'(\|f\|_{L^2_{0,0,2}}+\|u|_{t=0}\|_{L^2_{1,2}}),
	\end{equation}
	such that $C'$ is independent of $\eta$, since then we can combine this with the previous estimate to conclude. To prove this, note that we know the estimate \eqref{eq_shorttimel2} holds over the time interval $[0,\ep)$ from Theorem \ref{thm_isomorphismL}. Now, for a given $u\in C^\infty(\bL|_{[0,\eta)})$, we can extend the function $f=Tu \in C^\infty(\bL|_{[0,\eta)})$ to a function $f'\in C^\infty(\bL|_{[0,\ep)})$ such that $|\|f\|_{L^2_{0,0,\lambda}}- \|f'\|_{L^2_{0,0,\lambda}}|$ is arbitrarily small and $f|_{[0,\eta)} \equiv f'|_{[0,\eta)}$. Further, we can extend $u$ to $u'$ such that we have $Tu'=f'$, by solving the initial value problem $u'|_{t=\eta}=u|_{t=\eta}$, $Tu'=f'$ on $[\ep',\ep)$. The solution $u'$ will be smooth on $[0,\ep)$ as well, due to the equation being uniformly parabolic. Therefore, from estimate \eqref{eq_shorttimel2} on $u'$, we obtain the corresponding estimate for $u$ with the same constant. This proves the estimate \eqref{eq_shorttimel2} with constant independent of $\eta$, thus proves the claim. 
\end{proof}

Now, we have that:

\begin{lem}
	Let us assume that $u$ decays to all orders at the bottom face and $\|u\|_{L^2_{k+2r,r,2}} \ll 1$. Then, we have that for any $\eta \in (0,\ep)$, the following estimate from Theorem \ref{thm_quadratic}, \[
	\|Q[u]\|_{L^2_{k+2r-2,r-1,2}}\leq C\|u\|^2_{L^2_{k+2r,r,2}},
	\]
	for a constant $C$ independent of $\eta$.
\end{lem}
\begin{proof}
	First, we apply Lemma \ref{prop_sobolevtime} to obtain a uniform Sobolev estimate for $u$ with constant independent of $\eta$. Then, we follow the proof of Theorem \ref{thm_quadratic} -- the constant in the theorem does not change on shortening the time interval, since the Sobolev embedding constant does not change.
\end{proof}
Now, given our new solution $u$ to the nonlinear problem, it is stationary under the iteration mapping, i.e.\ \[
u \mapsto (\rho^2\pa_v-\cL)^{-1}(Q[u]) = u.
\]
However, we have from the previous lemmas the estimate, \[
\|(\rho^2\pa_v-\cL)^{-1}(Q[u])\|_{L^2_{k+2r,r,2}} \leq C'\|u\|^2_{L^2_{k+2r,r,2}},
\]
for a constant $C'$ independent of $\eta$. Therefore, we can reduce the time interval to $\llbracket0,\eta)$ to make $\|u\|_{L^2_{k+2r,r,2}}$ small enough, in order to conclude that $\|u|_{\llbracket0,\eta)}\|_{L^2_{k+2r,r,2}} =0$ for a sufficiently small $\eta$. This implies that the new solution must be identically zero for short time, hence identically zero on the whole of $[0,\ep)$.

\subsubsection{Uniqueness as an embedding}
Now, we want to show our solution is unique as an a-smooth embedding of the manifold with corners $\bL$ into $\bX$, such that it matches with the initial condition on the bottom face and the choice of expander on the front face.
\begin{prop}
	The embedding $\iota: \bL\hookrightarrow \bX$ for the given restrictions to the boundaries, $\iota_{\botf}:\bL|_{\botf} \hookrightarrow \bX|_{\botf}$ and $\iota_{\ff} : \bL|_{\ff} \rightarrow \bX|_{\ff}$, such that the time-slices satisfy LMCF, is unique. More precisely, given any other embedding $\iota': \bL \hookrightarrow \bX$ satisfying the same boundary conditions, we can decrease the time interval if necessary to $[0,\eta)$, such that there is an a-smooth diffeomorphism $\phi: \bL|_{[0,\eta)} \rightarrow \bL$ such that $\iota'|_{[0,\eta)} = \iota \circ \phi$.
\end{prop}
\begin{proof}
First, due to a-smoothness and the given boundary conditions, we can use an implicit function theorem to write $\bL'$ as the graph of a relative closed a-smooth $1$-form $\alpha$ over the Lagrangian neighbourhood of $\bL$ for sufficiently small time, such that the form vanishes at the boundaries. Then, $\alpha$ is exact because the Lagrangian mean curvature flow dictates the Hamiltonian isotopy classes in which the Lagrangians must lie. Then, by Lemma \ref{lem_formregularity}, there exists an a-smooth function $f:\bL\rt \bR$ such that $d_\rel f =\alpha$, and $f|_{\pa\bL}\equiv 0$. By making the time interval short enough, the uniqueness of solution to the nonlinear P.D.E.\ implies that $f\equiv 0$. Therefore, the mapping $\iota'$ is onto the image of the mapping $\iota$ for short time. Due to both being a-smooth embeddings into $\bX$, the claim follows.
\end{proof}

Therefore, we have proved the following main theorem:

\begin{theorem}\label{thm_main}
	Let $\iota: L \rt X$ be an embedding of a conically singular Lagrangian with a finite number of singular points $\{x_i\}_{i=1}^k\subset X$, such that each of its asymptotic cones $C_i$ is a union of special Lagrangian cones $C_i=\bigcup_{j=1}^n C_{ij}$ (which can have different phases). For each $i$, let $E_i$ be a Lagrangian expander which is asymptotic to the cone $C_i$, which decays at rate $-\infty$ to $C_i$ by Proposition \ref{prop_expanderdecay}. There exists a sufficiently small $\ep>0$, such that the following holds: Suppose we construct $\bX\in \mancac$ as explained before, by parabolically blowing up $X\times[0,\varepsilon)$ at the singular points $\{(x_i,0)\}_{i=1}^k$ at time $0$. Then, there exists $\bL\in\mancac$ and an a-smooth embedding $i:\bL\hookrightarrow\bX$ such that\begin{itemize}
		\item $\bL$ has an ordinary boundary, called the bottom face, such that $\bL|_{\mathrm{bf}}=\overline{\widetilde{L_0}}$, where $\overline{\widetilde{L_0}}$ denotes the closure of the radial blow-up of the initial Lagrangian at its singular points, and $i|_{\mathrm{bf}}:\overline{\widetilde{L_0}}\rt \overline{\widetilde{X}}=\bX|_{\mathrm{bf}}$ is the lift of the embedding $L\subset X$ to the blow-ups. 
		\item $\bL$ has an a-boundary, which is the disjoint union of $k$ distinct pieces, called the front faces, such that $\bL|_{\mathrm{ff}_i}=\overline{E_i}$, where $E_i$ is the expander asymptotic to the cone $C_i$ as above, and $\iota|_{\mathrm{ff}_i}: \overline{E} \rt \overline{\bC^m}\cong \bX|_{\mathrm{ff}_i}$ is the compactification of the embedding of $E$ into the $i^{th}$ front face of $\bX$.
		\item The embeddings on the time slices, $\iota|_{t}:\bL_t\rt \bX_t\cong X$  evolve by LMCF, for $t\in(0,\varepsilon)$.
	\end{itemize}
	Moreover, given the prescribed embeddings on the bottom face and the front faces, the embedding is unique up to diffeomorphism.
\end{theorem}

\subsection{Asymptotics of the solution near the front face}\label{sec_polyhomogeneouscon}

In this subsection, we make the assumption that the initial Lagrangian $L_0$ has a polyhomogeneous conormal expansion with index set $\cI$ near it singularities, i.e.\ the pull-back $\Upsilon^{-1}(\iota(L))$ can be written as a graph $\Gamma(df)$ over its special Lagrangian tangent cone $C$, with a polyhomogeneous conormal function $f$  of the form \[
f \sim \sum_{(z,p)\in \cI} r^z\log(r)^pf_{z,p},
\]
for $f_{z,p}$ being smooth functions on the link $\Sigma$, and $\cI \subset \bC\times \bN_0$ being the index set for $f$ (for details about the asymptotics defined by $\sim$, see for example Grieser \cite{Gre01}). Since we assume convergence to the cone with rate $\mu >2$, we have $\Re(\cI) \subset (2,\infty)$. In our convention of index sets, the set $\cI$ is finite in any half-plane $\{z:\Re(z)<s\}\times \bN_0$, and $(z,p)\in \cI$ implies that $(z,q)\in \cI$ for $q\leq p$. We also assume that $(j,0)\in \cI$ for $j\in \bN$ and $j>2$. Then, we first note that:
\begin{lem}\label{lem_embeddingpolycon}
	Consider the approximate solution $\iota:\bL \hookrightarrow \bX$ as constructed in Proposition \ref{prop_bottomfacesolving}, such that the error term vanishes to all orders in $s$. Then, the embedding of this approximate solution into $\bX$ is also polyhomogeneous conormal with the same index set. More precisely, we can write $\tU^{-1}(\iota(\bL))$ near the front face of $\Bl$ as the graph of a polyhomogeneous conormal function over the compactification $\bE$, with the same index set $\cI$.
\end{lem}
\begin{proof}
	This follows from the construction of the approximate solution in \S \ref{sec_approxsol} -- the hypothesis that $L_0$ admits a polyhomogeneous conormal expansion and choice of smooth interpolating cutoff functions implies that the compactification of the approximate solutions also admits an expansion with the same index set over the compactification of the expander. 
\end{proof}

Now, we will similarly solve the nonlinear P.D.E.\ \eqref{eq_potentialequation} to infinite order at the front face. As before in Section \ref{subsec_Fredholmtheory}, we work with coordinates near the front face given by $(r,\sigma^i)$ for $r\in[0,\infty), \sigma^i\in \oE$.

\begin{lem}
	We can write the operator on $\bL$ near the front face in the coordinates $(\sigma,r)\in \oE\times[0,\infty)$, as a perturbation of the corresponding operator on $\bE$,
	\begin{equation}\label{eq_nonlinearpde}
		\rho^2P[u]= \rho^2\nu - (\rho\pa_\rho)^2u+(\rho\pa_\rho)(w+h)u - P_0u + Q((\rho\pa_\rho+{^b\nabla}_{\oE})u,(\rho\pa_\rho+{^b\nabla}_{\oE})^2u)+\rho^2Q'[u],
	\end{equation}
	where $\nu$ denotes a smooth error function arising from the embedding $\tU|_{V}: \bE|_{V} \rt \bX$, with $V$ being a small open neighbourhood of the front face of $\bE$. Further, we can write 
	\begin{equation}\label{eq_newindexset}
	\rho^2Q'[u] = \sum_{i,j \geq 0} a_{ij}\cdot(\rho^{-2}d_\rel u)^i\otimes(\rho^{-2}{^b\nabla^2_\rel}u)^j,
	\end{equation}
	such that the coefficients $a_{ij}$ admit polyhomogeneous conormal expansions near the front face, with index set given by 
	\begin{equation}\label{eq_solutionindexset}
		\cI'= \cup_{k\geq 1} k\cI+(\bN,0),
	\end{equation}
	where $k\cI$ denotes the sum of the index set with itself $k$ times. Moreover, the weighted error $\rho^2E$ for the embedding $\iota:\bL \rt \bX$ admits an expansion with index set $\cI'$.
\end{lem}
\begin{proof}
	Here, we first work over the neighbourhood $V$ of the front face of the compactification $\bE$. Then, we can write the nonlinear part of the P.D.E.\ over $V$ as a \textit{smooth} function on $(U_\bE\oplus \textrm{Sym}^2(^bT^*_\rel\bE))|_{V}$, since the embedding $\bE\hookrightarrow\bX$ is smooth instead of merely a-smooth. Then, we can Taylor expand it in the fibre directions, to get that\[
	Q[\rho^{-2}d_\rel u, \rho^{-2}{^b\nabla^2_\rel}u] = \sum_{i,j \geq 0} a_{ij}\cdot(\rho^{-2}d_\rel u)^i\otimes(\rho^{-2}{^b\nabla^2_\rel}u)^j,
	\]
	where the coefficients $a_{ij}$ are smooth $b$-tensors. Then, we Taylor expand the coefficients in the radial directions, which proves the claim for $\bE$, i.e.\ with index set $\cI = \bN_{\geq 2}$. To prove the claim for $\bL$, we can consider the nonlinear part $Q[\rho^{-2}d_\rel (u+f),\rho^{-2}{^b\nabla^2_\rel (u+f)}]$, where we write $\bL = \Gamma(d_\rel f)$ over $\bE$. From Lemma \ref{lem_embeddingpolycon}, we know the asymptotics of $f$, so the coefficients $a'_{ij}$ will have index sets as a subset of $\cI'$ as defined. This proves the expansion \eqref{eq_newindexset} for the coefficient of the operator on $\bL$. Noting that the error for the embedding $\iota$ corresponds to the term $\rho^2\nu + a_{00}$, it follows that $\rho^2E$ has index set $\cI'$. 
\end{proof}

 Now, we make the following ansatz for the solution, 
\begin{equation}\label{eq_seriesansatzff}
	u \sim \sum_{(z,p)\in \cI'} r^z\log(r)^pu_{z,p}(s) \textrm{ near the front face},
\end{equation}
for $u_{z,p}(s)$ being smooth functions on $\oE$. Here, the index set $\cI'$ is as defined earlier,\[
\cI'= \cup_{k\geq 1} k\cI+(\bN,0).
\]
Now, we find functions $u_{z,p}$ such that the corresponding function $u$ as in \eqref{eq_seriesansatzff} satisfies $P[u]\sim 0$ as $r\rt 0$, i.e.\ near the front face. First, we note that \[
(\rho\pa_\rho)(r^z\log(r)^p) = zr^z\log(r)^p + p r^z\log(r)^{p-1},\]
\[ (\rho\pa_\rho)^2(r^z\log(r)^p) = z^2r^z\log(r)^p + 2pz r^z\log(r)^{p-1} + p(p-1)r^z\log(r)^{p-2}.
\]
Therefore, we get that
\begin{equation}\begin{split}
		\left(- (\rho\pa_\rho)^2+(\rho\pa_\rho)(w+h) - P_0\right)(r^z\log(r)^pu_{z,p}) &= -r^z\log(r)^pP_z(u_{z,p}) \\
		&+ \textrm{ lower order terms in } \{ u_{z,p-1},u_{z,p-2}\}.
	\end{split}
\end{equation}
So, the P.D.E. \eqref{eq_nonlinearpde} along with the boundary condition $u|_{s=0}\equiv0$ can be written as a system of equations on $\oE$, \[
-P_{z}(u_{z,p}) = (\rho^2\nu+a_{0,0})_{z,p} + F(u_{z,p+1},u_{z,p+2},\dots) + \textrm{ nonlinear terms in }\{u_{z-k,q}\},\; u_{z,p}|_{\pa\oE} = 0,
\]
for $(z,p) \in \cI'$.

Now, the main idea is to solve for these equations as $(z,p) \in \cI'$ inductively: we introduce a partial order $\geq$ on the set $\cI'$, given by $(z_1,p_1)\geq(z_2,p_2)$ if we have $\Re(z_1)\geq\Re(z_2)$, or we have $\Re(z_1)=\Re(z_2)$ and $p_1\leq p_2$. Then, at each step of induction, we solve for the $u_{z,p}$ which has its index smallest in the partial order (note that such a $(z,p)$ may not be unique; however, we can solve for them independently). Due to the way the ordering is defined, the P.D.E.\ for $u_{z,p}$ has coefficients depending only on $u_{z',p'}$ with $(z',p')<(z,p)$. Moreover, the solutions $u_{z,p}$ are uniquely determined as the mapping $(-P_z,|_{\pa\oE})$ is an isomorphism for $\Re(z) >0$ by Proposition \ref{prop_nokernel}, and are smooth by applying the interior estimate \eqref{eq_interiorestimateff}. Therefore, we obtain a formal series solution solving the P.D.E.\ to high order at the front face, which extends to a polyhomogeneous conormal function via Borel's lemma for polyhomogeneous conormal expansions. So, we obtain:

\begin{prop}\label{lem_polyconnearff}
	There is a polyhomogeneous conormal solution $u'$ of the P.D.E.\ such that $P[u'] \sim 0$ as $r\rt 0$, and $u'|_{s=0}\equiv 0$.
\end{prop}
Therefore, we can prove that:
 \begin{theorem}\label{thm_asymptotics}
 	 Consider the setup of Theorem \ref{thm_main}. If the initial Lagrangian is further prescribed a polyhomogeneous conormal expansion at each of the singularities  with index set $\cI$, then the solution has a (uniquely determined) polyhomogeneous conormal expansion at the front faces with index set $\cI'$.
 \end{theorem}
 
 \begin{proof}
 	This follows from first solving the problem to infinite order in both $\rho,s$ by Propositions  \ref{prop_bottomfacesolving} and \ref{lem_polyconnearff}, and performing the iteration mapping argument for $L^2_{k+2r,r,\lambda}$ with $\lambda \gg 0$ -- the error lies in this space due to vanishing at all orders at the boundaries.
 \end{proof}

As a corollary to this, we obtain:
\begin{coro}
	In the case of the problem of desingularising an immersed Lagrangian by gluing in Joyce--Lee--Tsui expanders, the resulting solution is a smooth manifold with corners. 
\end{coro}
\begin{proof}
	In the course of the proof, we have shown that if the initial Lagrangian has a polyhomogeneous conormal expansion with index set given by $\cI$, then the solution can be taken to have index set $\cI'$ given as in \eqref{eq_solutionindexset}. For an immersed Lagrangian, the index set is $\bN_{\geq 2}$, therefore the index set of the solution also lies in $\bN_{\geq 2}$, i.e.\ it is a smooth embedding.
\end{proof}

\bibliographystyle{alphaurl}
\bibliography{bibliography}

\end{document}